\documentclass[12pt]{article}%letter,

\usepackage{amsmath}
\usepackage{amsthm}

\usepackage{amssymb}
\usepackage{mathptmx}

\usepackage{caption}%, subcaption

\usepackage{xparse} % for {{}}
\usepackage{stmaryrd} % for [[]]

\usepackage{mathrsfs}

\usepackage{multirow}

\usepackage{geometry}
\usepackage{bm}

\usepackage{enumitem}

\usepackage{float} % float 
\usepackage{booktabs} % \toprule、\midrule、\bottomrule

\usepackage{color}
\usepackage{graphicx}
\usepackage{subfigure}
\usepackage{epstopdf}

\makeatletter
\def\widebreve{\mathpalette\wide@breve}
\def\wide@breve#1#2{\sbox\z@{$#1#2$}%
	\mathop{\vbox{\m@th\ialign{##\crcr
				\kern0.08em\brevefill#1{0.8\wd\z@}\crcr\noalign{\nointerlineskip}%
				$\hss#1#2\hss$\crcr}}}\limits}
\def\brevefill#1#2{$\m@th\sbox\tw@{$#1($}%
	\hss\resizebox{#2}{\wd\tw@}{\rotatebox[origin=c]{90}{\upshape(}}\hss$}
\makeatletter

\usepackage{hyperref} 
\usepackage{xcolor}
\hypersetup{
	colorlinks,
	linkcolor={green!50!black},%{red!50!black},
	citecolor={green!50!black}%,
}

\allowdisplaybreaks[1]
\def\jump#1{\llbracket #1 \rrbracket }

\newtheorem{theorem}{Theorem}[section]
\newtheorem{proposition}{Proposition}[section]
\newtheorem{lemma}{Lemma}[section]
\newtheorem{definition}{Definition}[section]
\newtheorem{remark}{Remark}[section]

\title{{\large \bf  Positivity-Preserving Well-Balanced Central Discontinuous Galerkin Schemes for the Euler Equations under Gravitational Fields}}

\author{
	Haili Jiang\thanks{School of Mathematical Sciences, Peking University, Beijing 100871, P.R.~China. ({\tt jianghaili@pku.edu.cn})}, \quad
	Huazhong Tang\thanks{Nanchang Hangkong University, Jiangxi Province, Nanchang 330000, P.R.~China; HEDPS, Center for Applied Physics and Technology, and LMAM, School of Mathematical Sciences, Peking University, Beijing 100871, P.R.~China. ({\tt hztang@math.pku.edu.cn})}, \quad
	Kailiang Wu\thanks{Corresponding author. Department of Mathematics and SUSTech International Center for Mathematics, Southern University of Science and Technology, and National Center for Applied Mathematics Shenzhen (NCAMS), Shenzhen 518055, P.R.~China.  ({\tt wukl@sustech.edu.cn}). } %Tel: +86-755-88010575.
}

\date{}

\begin{document}
	
	\maketitle

	\vspace{-3mm}

\begin{abstract}
	
	This paper designs and analyzes  positivity-preserving well-balanced (WB)  central discontinuous Galerkin
	(CDG) schemes for the  Euler equations with gravity. A distinctive %desirable
	feature of these schemes is that they not only are
	 WB for a general known  stationary hydrostatic solution, but also can preserve the positivity
	of the fluid density and pressure.
	The standard CDG method does not possess this feature, while directly applying some existing WB techniques to the CDG framework may not accommodate the positivity and keep other important properties at the same time.
	In order to
	obtain the WB and positivity-preserving properties simultaneously while also maintaining the conservativeness and stability of the schemes,
	a novel spatial discretization is devised in the CDG framework based on suitable modifications to the numerical dissipation term and the source term approximation.
	The modifications are based on a crucial projection operator for the stationary hydrostatic solution, which is proposed for the first time in this work.
	This novel projection has the same order of accuracy as the standard $L^2$-projection, can be explicitly calculated,
	and is easy to implement without solving any optimization problems.
	More importantly, it ensures that the projected stationary solution has
	 the same cell averages on both the primal and dual meshes, which is a key to achieve the desired properties of our schemes.
	Based on some convex decomposition techniques, rigorous positivity-preserving analyses for the resulting WB CDG schemes are carried out. Several one- and two-dimensional numerical examples are performed to illustrate the desired properties of these schemes, including the high-order accuracy, the WB property, the robustness for simulations involving the low pressure or density, high resolution for the discontinuous solutions and the small perturbations around the equilibrium state.
	
	\vspace*{4mm}
	
\end{abstract}

\newpage

\section{Introduction }

The Euler equations under gravitational fields are widely adopted to model  physical phenomena in the atmospheric science
and astrophysics, such as numerical weather forecasting \cite{botta2004well}, climate modeling, and supernova explosions \cite{Mishra2014well}. In the one-dimensional case, this nonlinear system can be written into the form of the hyperbolic balance laws as
\begin{equation}\label{eq:system-1d}
	{\bf U}_{t} + {\bf F}({\bf U})_{x} = {\bf S}({\bf U},\phi_{x}) \,,
\end{equation}
with
\begin{equation*}%\label{eq:euler-gravity-1d}
	{\bf U} =
	\begin{pmatrix}
		\rho \\  m \\  E
	\end{pmatrix}
	 \,, \qquad
	{\bf F}({\bf U}) =
	\begin{pmatrix}
		\rho u \\  \rho u^{2}+p \\  (E+p)u
	\end{pmatrix} \,, \qquad
	{\bf S}({\bf U},\phi_{x}) =
	\begin{pmatrix}
		0 \\   -\rho \phi_{x} \\   - m \phi_{x}
	\end{pmatrix} \,.
\end{equation*}
Here $\rho$ denotes the fluid density, $u$ is the velocity, $m = \rho u$ is the momentum, $p$ represents the pressure, $E = \frac{1}{2}\rho u^2 + \rho e$ denotes the total non-gravitational energy, and $e$ is the specific internal energy.
The function $\phi(x)$ in the source terms is the static gravitational potential. In order to close the system \eqref{eq:system-1d}, an equation of state (EOS) is needed and can be written as $p = p(\rho,e)$; for {the} ideal gas it is given by
\begin{equation}\label{eq:eos}
	p = (\gamma - 1)\rho e = (\gamma - 1) \bigg( E - \frac{m^2}{2\rho} \bigg) \,,
\end{equation}
where the constant $\gamma > 1$ denotes the ratio of specific heats. This paper will mainly focus on the
ideal EOS (\ref{eq:eos}), and the proposed schemes are readily extensible to a general EOS.

The Euler system (\ref{eq:system-1d}) under {the} gravitational potential $\phi(x)$ admits non-trivial stationary hydrostatic  solutions,  where the velocity is zero and the gravity is exactly balanced by the pressure gradient:
\begin{equation}\label{eq:steady-state-1d}
	\rho = \rho(x) \,, \quad u = 0 \,, \quad p_{x} = -\rho \phi_{x} \,.
\end{equation}
Two types of equilibria appear frequently in practical applications. They are the isothermal \cite{Xing2013FDWB} and polytropic \cite{Mishra2014well}  hydrostatic equilibrium states. The temperature $T(x) = T_{0}$ is a constant under the isothermal assumption. For an isothermal ideal gas with $p = \rho RT$,
integrating (\ref{eq:steady-state-1d}) yields
\begin{equation*}%\label{eq:isothermal-1d}
	\rho = \hat{\rho}_{0}\exp\Big( -\frac{\phi}{RT_{0}}\Big),
	\quad  u = 0, \quad
	p = \hat{p}_{0}\exp\Big( -\frac{\phi}{RT_{0}}\Big),
\end{equation*}
where $R$ is the gas constant, $\hat{p}_{0}$ is the pressure at a reference position $x_{0}$,  and $\hat{p}_{0} = \hat{\rho}_{0}RT_{0}$. A polytropic hydrostatic equilibrium, which arises from the astrophysical applications, is characterized by
$
p = K_{0} \rho ^ \nu,
$
and for this equilibrium, integrating (\ref{eq:steady-state-1d}) gives
\begin{equation*}%\label{eq:polytropic-1d}
	\rho = \Big( \frac{\nu-1}{K_{0} \nu}(C-\phi) \Big)^\frac{1}{\nu-1},
	\quad  u = 0, \quad
	p = \frac{1}{K_{0}^{\frac{1}{\nu-1}}}\Big( \frac{\nu-1}{\nu}(C-\phi) \Big)^\frac{\nu}{\nu-1}{,}
\end{equation*}
with $K_{0}$ and $C$ being constants. A special case is $\nu=\gamma$, which corresponds to a constant entropy.

In order to correctly and accurately capture small perturbations around the equilibrium state (\ref{eq:steady-state-1d}), it is desirable to
develop well-balanced (WB) numerical methods that preserve the discrete version of {those} steady state solutions exactly up to machine accuracy.
In fact, a straightforward numerical discretization may not be WB and can
 lead to a numerical solution which is inaccurate or oscillates around the hydrostatic equilibrum after a long time simulation.
This problem may be improved if the mesh size is extremely refined, which, however, may cause the simulation time-consuming especially in {the} multidimensional cases. To reduce the computational cost, the exploration of {the} WB schemes
has attracted much attention in the past few decades.
Most of those schemes were devised for the nonlinear shallow water equations over varied bottom topology, another typical model of hyperbolic balance laws; see, e.g., \cite{li2017CDGWB, audusse2004fast, bermudez1994upwind, greenberg1996well, LeVeque1998balance, Xing2005high, Xing2014survey, Xing2010PP, Xu2002shallow}  for more details. 
In recent years, various WB schemes for {the Euler equations under gravitational fields} have  been developed within several different frameworks, including but not limited to the non-central finite volume methods \cite{botta2004well,Xu2010kinetic, Xu2010gas, Mishra2014well, kappeli2016well, %Mishra2016well,
Klingenberg2015well, Xing2016FVWB},  {the} central finite volume methods \cite{Klingenberg2016well,  Klingenberg2020well},
 {the} finite difference methods \cite{Xing2013FDWB,ghosh2016well,Xing2018FDWB}, and {the} discontinuous Galerkin (DG) methods \cite{Xing2016DGWB,Zenk2017WBNDG,Xing2018DGWB}, etc. A numerical comparison between the high-order DG method and {the} WB DG method was carried out in \cite{veiga2019capturing}.
Most of {those} works assumed that the target equilibrium is explicitly known, which is also adopted in our present work. It is worth mentioning that, recently, there exist some efforts \cite{berberich2020high, chertock2018well, franck2016finite, kappeli2016well, varma2019second} on developing {the} WB schemes for the Euler system under gravitational field, without requiring a prior knowledge of the  stationary hydrostatic solution.

In physics, the fluid density and thermal pressure are positive, implying that the conservative variables ${\bf U}$ must stay in the set of admissible states
\begin{equation*}%\label{set:admissible}
	{ G} := \left\{ {\bf U} = (\rho, m, E)^{\top}:~ \rho > 0, ~  p({\bf U}) = (\gamma -1)\left(
	E - \frac{m^2}{2\rho} \right) > 0 \right\} \, .
\end{equation*}
Given that the initial data in the set ${G} $, a scheme is defined to be positivity-preserving if {its} solutions are always belong to $G$. Over the past decade,  {studying the} positivity-preserving and more generally bound-preserving
high-order numerical methods has attracted much attention and achieved significant progresses for hyperbolic systems.
Most of those {high-order accurate} schemes are designed with two types of limiters: the simple scaling limiter proposed in \cite{zhang2010maximum,Zhang2010PP} and the flux-correction limiters proposed in \cite{hu2013positivity, xu2014parametrized}.
 Based on the simple scaling limiter, {the} high-order positivity-preserving DG schemes were designed for the Euler equations without source term in \cite{Zhang2010PP, zhang2012maximum} and with various source terms including the gravitational source term in \cite{zhang2011positivity}.
The bound-preserving methods were also extended
to, for example, the shallow water equations \cite{Xing2010PP},
the special relativistic Euler equations \cite{WuTang2015,qin2016bound,WuTang2017ApJS,WuMEP2021}, the compressible Navier--Stokes equations \cite{zhang2017positivity}, and the compressible magnetohydrodynamic systems   \cite{Wu2017a,WuShu2018,wu2019provably,WuTangM3AS,wu2021provably}, and the general relativistic Euler equations under strong gravitational fields \cite{Wu2017}. 
Recently, a universal framework, called geometric quasilinearization (GQL), was proposed in \cite{wu2021geometric} for studying general  bound-preserving problems involving nonlinear constraints, with applications to a wide variety of physical systems including the Euler equations.    
For more developments and applications,
{the readers are} referred to the review articles \cite{shu2016bound,XuZhang2017} and the references therein.

The present paper is concerned with {the} central DG (CDG) methods for solving the Euler equations under gravitational fields.
The CDG methods are a family of high-order numerical schemes based on the DG methods \cite{cockburn1998runge} and the central scheme framework \cite{nessyahu1990non,kurganov2000new,liu2005central}, which were originally introduced for solving the  hyperbolic conservation laws \cite{Liu2007CDG}, and have been applied to the  Hamilton-Jacobi equations \cite{li2010HJ}, the ideal magnetohydrodynamic equations \cite{li2011MHD,li2012arbitrary}, and the special relativistic {hydrodynamic \cite{zhao2017RHD} and} magnetohydrodynamic equations \cite{zhao2017RMHD}. Based on the simple scaling limiter, {the} high-order bound-preserving CDG schemes were constructed for
the scalar conservation laws and {the} Euler equations \cite{li2016maximum}, the relativistic Euler equations \cite{WuTang2017ApJS}, and the shallow water equations \cite{li2017CDGWB}. The CDG methods evolve two copies of numerical solutions defined on {two sets of meshes (e.g.~the primal
mesh and its dual mesh)}, avoiding using any exact or approximate Riemann solvers at the cell interfaces which can be extremely complicated and time-consuming in some cases. Although needing more memory space than standard DG methods, the CDG methods were proven to allow relatively larger time step-sizes \cite{li2015operator} and be more accurate in some numerical tests \cite{Liu2008CDG}.

The aim of this work is to design and rigorously analyze {the} high-order positivity-preserving WB CDG schemes for the Euler equations under gravitational fields. A second-order positivity-preserving WB finite volume scheme based on a relaxation Riemann solver was developed in \cite{thomann2019second} {for the Euler equations with gravity for arbitrary hydrostatic equilibria}. Based on the (non-central) DG framework, {the} arbitrarily high-order positivity-preserving WB methods were proposed in \cite{wu2021uniformly} for {the} Euler equations with gravitation. It is also worth mentioning that, in the context of {the} shallow water equations, several positivity-preserving WB schemes have been developed in the literature \cite{kurganov2007second, xing2011high, Xing2010PP, li2017CDGWB, zhang2021positivity}. However, within the CDG framework, the study of {the} positivity-preserving WB schemes for the Euler equations with gravitation  is still blank. 

	For the regular (non-central) DG methods \cite{wu2021uniformly}, a key to achieve the WB and positivity-preserving properties  simultaneously is based on a suitable modification of the HLLC numerical flux, which satisfies both the contact property and the positivity. 
	By contrast, the CDG methods have no numerical flux but possess an extra
	numerical dissipation term (not existing in the regular DG methods). As such, some existing WB and positivity-preserving techniques in the regular DG case \cite{wu2021uniformly} do not apply to the CDG case. 
	Therefore, the design and analysis of positivity-preserving WB schemes in 
	the CDG framework have quite different difficulties and require the development of new techniques. %, compared to the non-central DG case in \cite{wu2021uniformly}
	Most notably, we need to carefully deal with the numerical dissipation term by proposing a novel critical projection operator, so as to obtain 
	the WB and positivity-preserving properties simultaneously while also maintaining the conservativeness and stability of the CDG schemes. 
The main efforts in this paper are summarized as follows.
\begin{itemize}
	\item 	In order to obtain the WB and positivity-preserving properties simultaneously while also keeping the conservativeness and stability of the schemes,  {a novel} spatial discretization is devised in the CDG framework based on suitable modifications to the numerical dissipation term and {the} source term approximation.
	\item   The modifications are based on a crucial projection operator for the stationary  hydrostatic solution, which is proposed for the first time in this work. This novel projection has the same order of accuracy as the standard $L^2$-projection, can be explicitly calculated, and is easy to implement without solving any optimization problems. More importantly, it ensures that the projected stationary solution has the same cell averages on both the primal and dual meshes, which {is key} to achieve the desired properties of our schemes.
	\item  Based on some convex decomposition techniques, {a weak positivity property of the resulting WB CDG schemes is rigorously proved}, which implies that a simple limiter \cite{Zhang2010PP,Zhang2012robust}  can ensure the positivity-preserving property without losing the high-order accuracy.  The WB modifications of the numerical dissipation term and {the} approximate source term lead to additional difficulties in our positivity-preserving analyses, which are more complicated than the analysis for the standard CDG schemes.
\end{itemize}

The rest of the paper is organized as follows. {Section \ref{section:projection}   proposes}
the novel projection of the stationary hydrostatic  solutions. Section \ref{section:ppwb1d}  constructs the
high-order positivity-preserving WB CDG method for the one-dimensional Euler equations under gravitational fields.
The proposed CDG schemes are extended to the two-dimensional case in Section \ref{section:ppwb2d}.
Section \ref{section:numerical-example} gives several numerical examples to verify the high-order accuracy, robustness, and effectiveness of our schemes. Concluding remarks are finally presented in Section \ref{section:conclusion}.

\section{Novel projection of the stationary hydrostatic  solutions} \label{section:projection}

This section presents a novel projection of the stationary hydrostatic  solutions, which will play 
a crucial role in designing our positivity-preserving WB CDG method.

\subsection{Notations}
Let us introduce some standard notations.
 {The spatial domain $\Omega$ %= [x_{min},x_{max}]$
is uniformly divided into
 $ \{ I_{j} := (x_{j-\frac{1}{2}}, x_{j+\frac{1}{2}}) \} $
with constant stepsize} $ \Delta x = x_{j+\frac{1}{2}} - x_{j-\frac{1}{2}} $. {If denoting} $x_{j} = \frac{1}{2}(x_{j-\frac{1}{2}}+x_{j+\frac{1}{2}})$, then $\{ I_{j+\frac{1}{2}} := (x_{j}, x_{j+1})\}$ forms a dual partition.
To approximate  the exact solution ${\bf U}(x,t)$
{in the CDG framework,
 two discrete function spaces are defined} associated with the primal mesh $\{I_{j}\}_{j}$ and the dual mesh $\{I_{j+\frac{1}{2}}\}_{j}$, respectively, as
\begin{equation*}
	\mathbb{V}_{h}^{C,k} = \left\{v \in L^2(\Omega): v|_{I_{j}} \in \mathbb{P}^{k}(I_{j}) ~\forall j \right\}, \ \
	\mathbb{V}_{h}^{D,k} = \left\{w \in L^2(\Omega): w|_{I_{j+\frac{1}{2}}} \in \mathbb{P}^{k}(I_{j+\frac{1}{2}}) ~\forall j \right\} \,,
\end{equation*}
where $\mathbb{P}^{k}(I_{j})$ and $\mathbb{P}^{k}(I_{j+\frac{1}{2}})$ denote the space of {the} polynomials with degree at most $k$ on the cells $I_{j}$ and $I_{j+\frac{1}{2}}$, respectively.

\subsection{Motivation of the novel projection}

Assume that the target equilibrium state is known and denoted by
$ \big\{ \rho^s(x), u^s(x), p^s(x) \big\} $. This yields
\begin{equation*}
	u^s(x) = 0 \,, \quad (p^s(x))_{x} = - \rho^s(x) \phi_{x} \,.
\end{equation*}
Let $ {\bf U}^s(x) = \big( \rho^s(x), 0, p^s(x)/(\gamma-1) \big)^\top $.
%and assume the initial data ${\bf U}(x,t=0) = {\bf U}^s(x)$.
The standard CDG method is generally not WB for {the} stationary hydrostatic  solutions of the Euler system \eqref{eq:system-1d}, and some modifications are required.
As the WB (non-central) DG and finite volume schemes in \cite{wu2021uniformly,Xing2016DGWB,Xing2016FVWB},
our WB CDG methods proposed in the Section \ref{section:ppwb1d} are also achieved by suitable modifications based on the projection of the target
stationary hydrostatic  solution.
However, the standard $L^2$-projection is not a good choice in the present CDG framework, as it may lose
the conservative property and affect the positivity-preserving property of the CDG schemes, which will be clarified in Remarks \ref{rem:L2-00} and \ref{eq:L2-aaa}. In order to maintain the WB, positivity-preserving and conservative properties at the same time,
we need to seek a new projection, which ensures that the projected stationary solutions
${\bf U}_{h}^{s,C} \in [\mathbb{V}_{h}^{C,k}]^3 $ and 
${\bf U}_{h}^{s,D} \in [\mathbb{V}_{h}^{D,k}]^3$,
%on the primal and dual meshes respectively, 
have the same cell averages, namely,
\begin{equation}\label{eq:projected-conditions}
	\int_{I_{j}} {\bf U}_{h}^{s,C} {\rm d}x = \int_{I_{j}} {\bf U}_{h}^{s,D} {\rm d}x \,,  \quad
\int_{I_{j+\frac{1}{2}}} {\bf U}_{h}^{s,D} {\rm d}x = \int_{I_{j+\frac{1}{2}}} {\bf U}_{h}^{s,C} {\rm d}x,  \quad  \forall j.
\end{equation}
The projected stationary solutions ${\bf U}_{h}^{s,C}$ and ${\bf U}_{h}^{s,D}$ will be used to modify the numerical dissipation term and {discretized} source term for {the} WB property, and the desired condition \eqref{eq:projected-conditions} will be important in guaranteeing the provably positivity-preserving and conservative properties;  see more details in Section \ref{section:dissipation} and Section \ref{section:positivity-preserving}.

\begin{remark}
Note that, for the shallow-water equations with (non-flat) bottom topography function $b(x)$,
a similar projection of $b(x)$ is also required in designing {the}
positivity-preserving WB CDG methods in \cite{li2017CDGWB},
where the projection of $b(x)$ is defined by
%In designing positivity-preserving WB CDG method for the shallow-water equations, it was proposed in \cite{li2017CDGWB} the following approximation to the bottom topography function $b(x)$
by solving a constrained minimization problem
\begin{equation}\label{eq:LiProj}
	\begin{aligned}
	& \min_{b_{h}^{C} \in \mathbb{V}_{h}^{C,k}, ~b_{h}^{D} \in \mathbb{V}_{h}^{D,k}} \qquad
	\lVert b_{h}^{C} - b \lVert_{L^2(\Omega)}  + \lVert b_{h}^{D} - b \lVert_{L^2(\Omega)} \,,  \\
	& {\rm subject ~ to}  ~~\int_{I_{j}} b_{h}^{C} {\rm d}x = \int_{I_{j}} b_{h}^{D} {\rm d}x \,, \quad
	\int_{I_{j+\frac{1}{2}}} b_{h}^{D} {\rm d}x = \int_{I_{j+\frac{1}{2}}} b_{h}^{C} {\rm d}x.
\end{aligned}
\end{equation}
Although there is no rigorous proof, numerical tests in \cite{li2017CDGWB} indicate that the above projected approximations $b_{h}^{C}$ and  $b_{h}^{D}$ have the same high-order accuracy as the standard $L^2$-projection, provided that $b(x)$ is a smooth function.
The notion of \eqref{eq:LiProj} can be extended to our case to construct a projection satisfies \eqref{eq:projected-conditions}. However, this projection is not easy to implement due to the involved optimization problem, and its accuracy has not yet been theoretically justified.
\end{remark}

We find a new projection, which satisfies \eqref{eq:projected-conditions} and is more efficient than \eqref{eq:LiProj}.
Our new projection can be {\em explicitly calculated} without solving
any (constrained) optimization problems,
and { thus can be easily implemented}. Moreover, we can rigorously prove that
our new projection also has the same order of accuracy as the standard $L^2$-projection.

\subsection{Definition of the novel projection}
%In this section, we introduce a  different approach to  design  projection of the hydrostatic stationary solutions.

We first define our new projection operator on the primal mesh. Let
 $ \mathcal{P}_{h}^{C} : L^2(\Omega) \longrightarrow \mathbb{V}_{h}^{C,k} $ denote an operator, which maps
 any function $f(x) \in L^2(\Omega)$ onto the piecewise polynomial space $\mathbb{V}_{h}^{C,k}$, and satisfies
\begin{align} \label{eq:projection-primal}
	\int_{I_{j}^{-}}  \mathcal{P}_{h}^{C} (f) {\rm d}x = \int_{I_{j}^{-}} f {\rm d}x \,, \quad
	\int_{I_{j}^{+}}  \mathcal{P}_{h}^{C} (f) {\rm d}x = \int_{I_{j}^{+}} f {\rm d}x, \quad  \forall j \,,
\end{align}
with $I_{j}^{-} = ( x_{j-\frac{1}{2}}, x_{j} )$, $I_{j}^{+} = ( x_{j}, x_{j+\frac{1}{2}} )$.
Note that with only the condition \eqref{eq:projection-primal}, the mapping $\mathcal{P}_{h}^{C} $ is not uniquely determined.
We uniquely define the mapping $\mathcal{P}_{h}^{C} $ by
\begin{align} \label{eq:projection-special}
	\int_{I_{j}^{-}}\mathcal{P}_{h}^{C} (f) {\rm d}x = \int_{I_{j}^{-}} f {\rm d}x \,,  \quad
	\int_{I_{j}} \mathcal{P}_{h}^{C} (f) v {\rm d}x  = \int_{I_{j}} f v {\rm d}x,  \quad
	\forall  ~ v \in  \mathbb{P}^{k}(I_{j}) \backslash  \mathrm{span}\{{\Phi}_{1}\},
\end{align}
where {$\mathbb{P}^{k}(I_{j}) \backslash\mathrm{span}\{{\Phi}_{1}\} := \mathrm{span} \{ {\Phi}_{0}\,, {\Phi}_{2}\,, {\Phi}_{3}\,, \cdots \,, {\Phi}_{k} \}$, and $\{{\Phi}_{i} \}_{i=0}^{k}$ denotes an orthogonal basis $\mathbb{P}^{k}(I_{j})$ and is taken as
the scaled Legendre polynomials
\begin{align}\label{eq:Legendre}
	{\Phi}_{0}(\xi) = 1 \,, 
\ \ {\Phi}_{1}(\xi) = \xi \,,  
\ \ {\Phi}_{2}(\xi) = \xi^2 - \frac{1}{3} \,,
\ \ {\Phi}_{3}(\xi) = \xi^3 - \frac{3}{5}\xi \,,  
\ \ \dots \,,
\end{align}
with $\xi = {2(x-x_{j})}/{\Delta x} $}. Take $v = \Phi_{0} = 1$ in (\ref{eq:projection-special}), one has the equality $\int_{I_{j}} \mathcal{P}_{h}^{C} (f) {\rm d}x  = \int_{I_{j}} f  {\rm d}x $, which implies that the  projection $\mathcal{P}_{h}^{C}$ {satisfies} the desired condition (\ref{eq:projection-primal}).
We will {show} in Lemma \ref{lemma:operator-polynomial} that
the mapping $\mathcal{P}_{h}^{C} $ defined by \eqref{eq:projection-special} is
a projection operator from $L^2(\Omega)$ onto the space $\mathbb{V}_{h}^{C,k}$.

Similarly, we can define the projection $\mathcal{P}_{h}^{D} : L^2(\Omega) \longrightarrow \mathbb{V}_{h}^{D,k} $ on the dual mesh {by
\begin{align*} %\label{eq:projection-dual-special}
	\int_{I_{j}^{+}}\mathcal{P}_{h}^{D} (f) {\rm d}x = \int_{I_{j}^{+}} f {\rm d}x \,,  \quad
	\int_{I_{j+\frac12}} \mathcal{P}_{h}^{D} (f) v {\rm d}x  = \int_{I_{j+\frac12}} f v {\rm d}x{,}  \quad
	\forall  ~ v \in  \mathbb{P}^{k}(I_{j+\frac12}) \backslash  \mathrm{span}\{{\Phi}_{1}\}{,}
\end{align*}
where ${\Phi}_{i}(\xi)$, $i=0,1,\cdots,k$,  defined in \eqref{eq:Legendre} are with $\xi={2(x-x_{j+\frac12})}/{\Delta x}$.}
%\begin{align}
%	\phi_{0} = 1 \,, \qquad \phi_{1} = \xi \,,  \qquad \phi_{2} = \xi^2 - \frac{1}{3} \,,
%	\qquad \phi_{3} = \xi^3 - \frac{3}{5}\xi \,,  \qquad \cdots
%\end{align}
%and $\xi = \frac{2(x-x_{j+\frac12})}{\Delta x} $.
It can be shown that $\mathcal{P}_{h}^{D}$ satisfies
\begin{align}\label{eq:projection-dual}
	\int_{I_{j}^+} \mathcal{P}_{h}^{D}(f) {\rm d}x  = \int_{I_{j}^+} f  {\rm d}x \,,  \quad
	\int_{I_{j+1}^-} \mathcal{P}_{h}^{D}(f) {\rm d}x  = \int_{I_{j+1}^-} f  {\rm d}x{,}  \quad \forall j \,.
\end{align}

Combining  {(\ref{eq:projection-primal}) with (\ref{eq:projection-dual}) gives}
\begin{align}\label{eq:ave:ID}
	\int_{I_{j}} \mathcal{P}_{h}^{C}(f) {\rm d}x =
	\int_{I_{j}} \mathcal{P}_{h}^{D}(f) {\rm d}x \,, \quad
	\int_{I_{j+\frac{1}{2}}} \mathcal{P}_{h}^{D}(f) {\rm d}x =
	\int_{I_{j+\frac{1}{2}}} \mathcal{P}_{h}^{C}(f) {\rm d}x{,} \quad \forall j \,.
\end{align}
If assuming that ${\bf U}_{h}^{s,C}$ and ${\bf U}_{h}^{s,D}$ denote the new projections of
 each component of the stationary solutions ${\bf U}^s(x)$ onto the space $\mathbb{V}_{h}^{C,k}$ and $\mathbb{V}_{h}^{D,k}$, respectively, then it follows from \eqref{eq:ave:ID} that
\begin{align}\label{eq:average-identity}
\int_{I_{j}} {\bf U}_{h}^{s,C} {\rm d}x = \int_{I_{j}} {\bf U}^{s} {\rm d}x =
\int_{I_{j}} {\bf U}_{h}^{s,D} {\rm d}x \,, \
\int_{I_{j+\frac{1}{2}}} {\bf U}_{h}^{s,D} {\rm d}x = \int_{I_{j+\frac{1}{2}}} {\bf U}^{s} {\rm d}x =
\int_{I_{j+\frac{1}{2}}} {\bf U}_{h}^{s,C} {\rm d}x{,} \ \forall j \,.
\end{align}
%It is easy to find that the new projections keeps the local conservation property in each cell $I_{j}$ and $I_{j+\frac{1}{2}}$.
which yields \eqref{eq:projected-conditions}.
%the projected stationary solutions ${\bf U}_{h}^{s,C}$ and ${\bf U}_{h}^{s,D}$, on the primal and dual meshes respectively, have
%the same cell averages, as required by \eqref{eq:projected-conditions}.
The identity (\ref{eq:average-identity}) plays an important role in the positivity-preserving analyses in Section \ref{section:positivity-preserving}.

\subsection{Properties of the novel projection}
For convenience, we will mainly discuss the properties of the operator $\mathcal{P}_{h}^{C}$ on the primal mesh in detail, as those of $\mathcal{P}_{h}^{D}$ on the dual mesh are very similar.

Note that, on each cell $I_{j}$, the projected solution $\mathcal{P}_{h}^{C} (f)$ is a  piecewise polynomial and can be  explicitly written as
\[
f_{h}^{C}(x) = \sum_{i=0}^{k} N_{i}(f) {\Phi}_{i}(x) \,, \quad  x \in  I_{j}{,}
\]
with the coefficients $\{ N_{i}(f) \}_{i=0}^k$ given by
\begin{align*}
	N_{i}(f) = \frac{ \int_{I_{j}} f {\Phi}_{i} {\rm d}x }{\int_{I_{j}} {\Phi}_{i}^{2} {\rm d}x} \,, \quad  i \neq 1 \,,  \quad
	N_{1}(f) =  \frac{ \int_{ I_{j}^{-} } \big( f - \sum\limits_{i\neq 1} N_{i}(f) {\Phi}_{i} \big) {\rm d}x  }{ \int_{ I_{j}^{-} } {\Phi}_{1} {\rm d}x } \,,
\end{align*}
where $\int_{ I_{j}^{-} } {\Phi}_{1} dx = -\dfrac{\Delta x}{4} \neq 0 $. It is easy to show that the operator $\mathcal{P}_{h}^{C}$ satisfies the following properties.

\begin{lemma}\label{lemma:operator-polynomial}
	For any function $ g(x) \in \mathbb{V}_{h}^{C,k}$,  one has $ \mathcal{P}_{h}^{C}(g)(x) = g(x)$, and consequently $ \mathcal{P}_{h}^{C}$
	is a projection operator.
\end{lemma}

\begin{proof}
	For each $j$, on {the} cell $I_{j}$ the function $g(x) \in \mathbb{V}_{h}^{C,k}$ can be expressed as a linear combination of the orthogonal basis $ \{{\Phi}_{i} \}_{i=0}^{k} $, i.e.
	\[
	g(x) = \sum_{i=0}^{k} \alpha_{i} {\Phi}_{i}(x) \,, \quad  x \in I_{j} \,.
	\]
	Noting that
	\[
	N_{i}(g) =  \alpha_{i} \,,  \quad 0 \leq i \leq k \,,
	\]
	we obtain
	\begin{align*}
		\mathcal{P}_{h}^{C}(g)(x) = \sum_{i=0}^{k} N_{i}(g) {\Phi}_{i}(x) = \sum_{i=0}^{k} \alpha_{i} {\Phi}_{i}(x) = g(x) \,,
		\quad x \in I_{j} \,.
	\end{align*}
	For any functions $f_{1},f_{2} \in L^{2}(\Omega)$ and any real numbers $\alpha_{1},\alpha_{2}$, one has
	\[
	N_{i} (\alpha_{1}f_{1} + \alpha_{2}f_{2}) = \alpha_{1} N_{i} (f_{1}) + \alpha_{2} N_{i} (f_{2}){,}
	\qquad  0 \leq i \leq k \,,
	\]
	It follows that
	\begin{align*}
	\mathcal{P}_{h}^{C} (\alpha_{1}f_{1} + \alpha_{2}f_{2}) = \alpha_{1} \mathcal{P}_{h}^{C} (f_{1}) +
	\alpha_{2} \mathcal{P}_{h}^{C} (f_{2}) \,,
	\end{align*}
	which implies $\mathcal{P}_{h}^{C}$ is a linear operator.
	Furthermore, for any $f \in L^2(\Omega) $, we have $\mathcal{P}_{h}^{C}(f) \in \mathbb{V}_{h}^{C,k}$
	and thus $\mathcal{P}_{h}^{C} \big(\mathcal{P}_{h}^{C}(f)\big) = \mathcal{P}_{h}^{C}(f)$. This implies
	 $ \mathcal{P}_{h}^{C}$ is a projection operator. The proof is completed.
\end{proof}

\begin{lemma}\label{lemma:operator-bound}
	The projection operator $\mathcal{P}_{h}^{C}$ is bounded. % in the space $L^{2}(I_{j})$.
\end{lemma}

\begin{proof}
	Let us consider an arbitrary function $f \in L^{2}(I_{j})$. Applying the triangular inequality gives
	\begin{align*}
		\lVert \mathcal{P}_{h}^{C} (f) \lVert_{L^{2}(I_{j})}  \leq  \sum_{i=0}^{k} |N_{i}(f)| \lVert {\Phi}_{i} \lVert_{L^{2}(I_{j})} \,.
	\end{align*}
	If $ i \neq 1 $, one can derive that
	\[
	| N_{i}(f) | = \frac{ | \int_{I_{j}} f {\Phi}_{i} {\rm d}x | }{ \lVert {\Phi}_{i} \lVert_{L^{2}(I_{j})}^2 }
	\leq \frac{ \lVert f \lVert_{L^{2}(I_{j})} \lVert {\Phi}_{i} \lVert_{L^{2}(I_{j})}  }{ \lVert {\Phi}_{i} \lVert_{L^{2}(I_{j})}^2 }
	 = \frac{ \lVert f \lVert_{L^{2}(I_{j})} }{ \lVert {\Phi}_{i} \lVert_{L^{2}(I_{j})} } \,,
	\]
	which leads to
	\begin{align} \label{eq:operator-bound-first}
		|N_{i}(f)| \lVert {\Phi}_{i} \lVert_{L^{2}(I_{j})} \leq  M_{i} \lVert f \lVert_{L^{2}(I_{j})} \,, \quad  \mbox{with~~~~} M_{i} = 1 \,.
	\end{align}
	If $ i = 1 $, we have
	\begin{align} \notag
		|N_{1}(f)| &\leq \frac{ |\int_{ I_{j}^{-} } f {\rm d}x|  + \sum\limits_{i \neq 1} |N_{i}(f)| |\int_{ I_{j}^{-} } {\Phi}_{i} {\rm d}x| }
		{ |\int_{ I_{j}^{-} } {\Phi}_{1} {\rm d}x | }
		= \frac{ |\int_{ I_{j}^{-} } f {\rm d}x| }{ |\int_{ I_{j}^{-} } {\Phi}_{1} {\rm d}x | } + \sum_{i \neq 1}
		\frac{|\int_{ I_{j}^{-} } {\Phi}_{i} {\rm d}x| }{ |\int_{ I_{j}^{-} } {\Phi}_{1} {\rm d}x | } |N_{i}(f)| \,.
	\end{align}
	Note that $I_{j}^{-} \subseteq I_{j}$ and ${\Phi}_{0} = 1$, one has
	\[
	\left|\int_{ I_{j}^{-} } f {\rm d}x \right| = \left|\int_{ I_{j}^{-} } f {\Phi}_{0} {\rm d}x \right| \leq
	\lVert f \lVert_{L^{2}(I_{j})} \lVert {\Phi}_{0} \lVert_{L^{2}(I_{j}^-)} \,,  \quad
	\left|\int_{ I_{j}^{-} } {\Phi}_{i} {\rm d}x \right|  \leq
	\lVert {\Phi}_{i} \lVert_{L^{2}(I_{j})} \lVert {\Phi}_{0} \lVert_{L^{2}(I_{j}^-)} \,.
	\]
	It follows that
	\begin{align*}
		|N_{1}(f)| &\le \frac{  \lVert f \lVert_{L^{2}(I_{j})} \lVert {\Phi}_{0} \lVert_{L^{2}(I_{j}^-)}  }{ |\int_{ I_{j}^{-} } {\Phi}_{1} {\rm d}x | } + \sum_{i \neq 1}
		\frac{  \lVert {\Phi}_{0} \lVert_{L^{2}(I_{j}^-)} }{ |\int_{ I_{j}^{-} } {\Phi}_{1} {\rm d}x | }  \lVert {\Phi}_{i} \lVert_{L^{2}(I_{j})} |N_{i}(f)|
		\\
		&\leq  \frac{  \lVert f \lVert_{L^{2}(I_{j})} \lVert {\Phi}_{0} \lVert_{L^{2}(I_{j}^-)}  }{ |\int_{ I_{j}^{-} } {\Phi}_{1} {\rm d}x | } + \sum_{i \neq 1}
		\frac{  \lVert {\Phi}_{0} \lVert_{L^{2}(I_{j}^-)} }{ |\int_{ I_{j}^{-} } {\Phi}_{1} {\rm d}x | }  \lVert f \lVert_{L^{2}(I_{j})}
		\\
		&
		=(k+1) \frac{   \lVert {\Phi}_{0} \lVert_{L^{2}(I_{j}^-)} }
		{ |\int_{ I_{j}^{-} } {\Phi}_{1} {\rm d}x |}  \lVert f \lVert_{L^{2}(I_{j})} \,, \quad
	\end{align*}
where (\ref{eq:operator-bound-first}) has been used in the second inequality.
Therefore, we obtain
\begin{equation}\label{eq:operator-bound-second}
	|N_{1}(f)| \lVert {\Phi}_{1} \lVert_{L^{2}(I_{j})}   
\le  M_1  \lVert f \lVert_{L^{2}(I_{j})}
\end{equation}
	with
	$$
	M_{1} := (k+1) \times \frac{ \lVert {\Phi}_{1} \lVert_{L^{2}(I_{j})} \lVert {\Phi}_{0} \lVert_{L^{2}(I_{j}^-)} }
	{ |\int_{ I_{j}^{-} } {\Phi}_{1} {\rm d}x |} = \frac{2\sqrt{6}}{3} (k+1).
	$$
	{Combining  (\ref{eq:operator-bound-first}) with (\ref{eq:operator-bound-second}) yields}
	\begin{align*}
		\lVert \mathcal{P}_{h}^{C} (f) \lVert_{L^{2}(I_{j})}  \leq  M \lVert f \lVert_{L^{2}(I_{j})} \,,  \quad
		M = \sum_{i=0}^{k} M_{i} = k + \frac{2\sqrt{6}}{3} (k+1) \,.
	\end{align*}
	This finishes the proof.
\end{proof}

\begin{theorem} \label{theorem:operator-accuracy}
	For any function $f \in W_{2}^{k+1}(I_{j})$,  one has
	\begin{align} \notag
		\lVert f - \mathcal{P}_{h}^{C} (f) \lVert_{L^{2}(I_{j})} & \leq  C_{k+1} (\Delta x)^{k+1} |f|_{W_{2}^{k+1}(I_{j})} \,.
	\end{align}
	where $W_{2}^{k+1}(I_{j})$ is the standard Sobolev space, $|\cdot|_{W_{2}^{k+1}(I_{j})}$ is the Sobolev seminorm of order $k+1$, and $C_{k+1}$ is a constant only  depending on $k$.
\end{theorem}

\begin{proof}
	For any $g \in  \mathbb{P}^{k}(I_{j})$,
	applying Lemmas \ref{lemma:operator-polynomial} and \ref{lemma:operator-bound} and then using the triangular inequality give
	\begin{align} \notag
		\lVert f - \mathcal{P}_{h}^{C} (f) \lVert_{L^{2}(I_{j})}
		&\leq  \lVert f - g \lVert_{L^{2}(I_{j})}  +  \lVert g - \mathcal{P}_{h}^{C} (f) \lVert_{L^{2}(I_{j})}   \\  \notag
		& =  \lVert f - g \lVert_{L^{2}(I_{j})}  +  \lVert \mathcal{P}_{h}^{C} (f - g ) \lVert_{L^{2}(I_{j})}
		\quad ( \mathrm{Lemma ~ \ref{lemma:operator-polynomial}} )\\ \notag
		& \leq  \lVert f - g \lVert_{L^{2}(I_{j})}  +  M \lVert f - g  \lVert_{L^{2}(I_{j})}
		\quad ( \mathrm{Lemma ~ \ref{lemma:operator-bound}} )  \\
		& = ( 1 + M  )  \lVert f - g  \lVert_{L^{2}(I_{j})} \,.
	\end{align}
	Consequently,
	\begin{align} \notag
		\lVert f - \mathcal{P}_{h}^{C} (f) \lVert_{L^{2}(I_{j})} & \leq  ( 1 + M )
		\inf_{ g \in  \mathbb{P}^{k}(I_{j}) } \lVert f - g  \lVert_{L^{2}(I_{j})}   \\  \notag
		& \leq ( 1 + M ) \tilde{C}_{k+1} (\Delta x)^{k+1} |f|_{W_{2}^{k+1}(I_{j})}  \quad (\mathrm{Bramble-Hilbert}) \\ \notag
		& = C_{k+1} (\Delta x)^{k+1} |f|_{W_{2}^{k+1}(I_{j})} \,.
	\end{align}
	where have used the Bramble-Hilbert Lemma \cite{dekel2004bramble}, and $\tilde{C}_{k+1}$ is a constant only  depending on $k$.
\end{proof}

\begin{remark}
	Similarly, the operator $\mathcal{P}_{h}^{D}$ on the dual mesh is also linear, bounded, and a projection. Furthermore, for any function $f \in W_{2}^{k+1}(I_{j-\frac{1}{2}})$
	\begin{align} \notag
		\lVert f - \mathcal{P}_{h}^{D} (f) \lVert_{L^{2}(I_{j-\frac{1}{2}})}
		& \leq  C_{k+1} (\Delta x)^{k+1} |f|_{W_{2}^{k+1}(I_{j-\frac{1}{2}})} \,,
	\end{align}
	where $C_{k+1}$ is a constant only  depending on $k$.
\end{remark}

\begin{remark}
By similar arguments, one can prove that the errors $\lVert f - \mathcal{P}_{h}^{C} (f) \lVert_{L^{q}(\Omega)}$ and $\lVert f - \mathcal{P}_{h}^{D} (f) \lVert_{L^{q}(\Omega)}$ are  of order  ${\mathcal O} ( (\Delta x)^{k+1} )$ for a general $q$ ($1\le q \le +\infty$) and $f \in W^{k+1}_q(\Omega)$. The details are omitted here.
\end{remark}
\section{Positivity-preserving WB CDG schemes in one dimension} \label{section:ppwb1d}

To solve the {system} (\ref{eq:system-1d}),
{the} CDG schemes evolve two copies of numerical solutions, denoted by
${\bf U}_{h}^{C}(x,t)$ and ${\bf U}_{h}^{D}(x,t)$, on the primal and dual meshes, respectively.

\subsection{Review of the standard CDG method}

The semi-discrete formulations of the standard CDG method %uses both spaces $\mathbb{V}_{h}^{C,k}$ and
%$\mathbb{V}_{h}^{D,k}$, and
%its
 are given as follows: for any test function $v \in \mathbb{V}_{h}^{C,k}$
and $w \in \mathbb{V}_{h}^{D,k}$, look for {the} numerical solutions 
 ${\bf U}_{h}^{C} \in [\mathbb{V}_{h}^{C,k}]^3$ and ${\bf U}_{h}^{D} \in 
[\mathbb{V}_{h}^{D,k}]^3$  satisfying
\begin{align}\notag
	\int_{I_{j}} & \frac{\partial {\bf U}_{h}^{C}}{\partial t} v {\rm d}x
	= \frac{1}{\tau_{max}}\int_{I_{j}} ( {\bf U}_{h}^{D} - {\bf U}_{h}^{C}) v {\rm d}x
	+ \int_{I_{j}} {\bf F}({\bf U}_{h}^{D}) v_{x} {\rm d}x    
\\ \label{eq:CDG-primal}
	& - \Big( {\bf F}({\bf U}_{h}^{D}(x_{j+\frac{1}{2}})) v(x_{j+\frac{1}{2}}^{-})
	-       {\bf F}({\bf U}_{h}^{D}(x_{j-\frac{1}{2}})) v(x_{j-\frac{1}{2}}^{+}) \Big)
	+  \int_{I_{j}} {\bf S}( {\bf U}_{h}^{D}, (\phi_{h}^{D})_{x} ) v {\rm d}x \,,
\end{align}
\begin{align}\notag
	\int_{I_{j+\frac{1}{2}}}& \frac{\partial {\bf U}_{h}^{D}}{\partial t} w {\rm d}x
	= \frac{1}{\tau_{max}}\int_{I_{j+\frac{1}{2}}} ({\bf U}_{h}^{C} - {\bf U}_{h}^{D}) w {\rm d}x
	+ \int_{I_{j+\frac{1}{2}}} {\bf F}({\bf U}_{h}^{C}) w_{x} {\rm d}x    \\\label{eq:CDG-dual}
	& - \Big( {\bf F}({\bf U}_{h}^{C}(x_{j+1}))w(x_{j+1}^{-})
	- {\bf F}({\bf U}_{h}^{C}(x_{j}))w(x_{j}^{+}) \Big)
	+ \int_{I_{j+\frac{1}{2}}} {\bf S}({\bf U}_{h}^{C}, (\phi_{h}^{C})_{x} ) w {\rm d}x  \,.
\end{align}
Here $\tau_{max} = \tau_{max}(t)$ is the maximal time step allowed by the CFL condition at time $t$, and $f(x^{\pm}) = \lim_{\epsilon \to 0+}f(x\pm\epsilon)$ denotes the limits at point $x$ taken from the left and right sides, respectively.
%The standard CDG method can be written in the ODE form as
%\begin{equation}\label{eq:SCDG-ode}
%	\frac{\mathrm{d} {\bf U}_{h}^C }{\mathrm{d} t} = {\bf \mathcal{L}}_{h}^C ({\bf U}_{h}^C, {\bf U}_{h}^D) \,, \quad
%	\frac{\mathrm{d} {\bf U}_{h}^D }{\mathrm{d} t} = {\bf \mathcal{L}}_{h}^D ({\bf U}_{h}^D, {\bf U}_{h}^C) \,,
%\end{equation}
%after choosing suitable basis of $\mathbb{V}_{h}^{C,k}, \mathbb{V}_{h}^{D,k}$ and representing each component of
%${\bf U}_{h}^C, {\bf U}_{h}^D$ as linear combination of the corresponding basis functions.
In general, the standard CDG method is neither positivity-preserving nor WB and does not maintain the stationary hydrostatic  solution (\ref{eq:steady-state-1d}).
% and may produce solutions with spurious oscillations when they are in or close to a stationary state.

\subsection{Outline of the positivity-preserving WB CDG schemes}

To define the positivity-preserving CDG schemes, we introduce the following two sets
\[
\overline{\mathbb{G}}_{h}^{C,k} := \Big \{ {\bf v} \in [\mathbb{V}_{h}^{C,k}]^3 : ~
\frac{1}{\Delta x} \int_{I_{j}} {\bf v} (x) {\rm d}x \in G \,, ~ \forall j  \Big\}  \,,
\]
\[
\mathbb{G}_{h}^{C,k} := \Big \{ {\bf v} \in \overline{\mathbb{G}}_{h}^{C,k} :
{\bf v}\big|_{I_{j}} (x) \in G, ~ \forall x \in \mathbb{S}_{j}   \,, ~ \forall j  \Big\}  \,,
\]
where $\mathbb{S}_{j}$ denotes the set of some critical points in $I_{j}$ which will be specified in \eqref{eq:sj-primal}.
Similarly, we can define the sets $\overline{\mathbb{G}}_{h}^{D,k}$ and $\mathbb{G}_{h}^{D,k}$ on the
dual mesh as
\[
\overline{\mathbb{G}}_{h}^{D,k} := \Big \{ {\bf w} \in [\mathbb{V}_{h}^{D,k}]^3 : ~
\frac{1}{\Delta x} \int_{I_{j+\frac{1}{2}}} {\bf w} (x) {\rm d}x \in G \,, ~ \forall j  \Big\}  \,,
\]
\[
\mathbb{G}_{h}^{D,k} := \Big \{ {\bf w} \in \overline{\mathbb{G}}_{h}^{D,k} :
{\bf w}|_{I_{j+\frac{1}{2}}} (x) \in G, ~ \forall x \in \mathbb{S}_{j+\frac{1}{2}}   \,, ~ \forall j  \Big\}  \,,
\]
where $\mathbb{S}_{j+\frac12}$ denotes the set of some critical points in $I_{j+\frac12}$ which will be specified in \eqref{eq:sj-dual}.

We will describe in Section \ref{section:spatial-discretization}
a suitable CDG spatial discretization of the Euler system \eqref{eq:system-1d}, and the resulting
semi-discrete CDG schemes can be written in the ODE form as follows
\begin{equation}\label{eq:ODE-system}
	\frac{\mathrm{d} {\bf U}_{h}^C }{\mathrm{d} t} = {\bf L}_{h}^C ({\bf U}_{h}^C, {\bf U}_{h}^D) \,, \quad
	\frac{\mathrm{d} {\bf U}_{h}^D }{\mathrm{d} t} = {\bf L}_{h}^D ({\bf U}_{h}^D, {\bf U}_{h}^C) \,,
\end{equation}
where ${\bf L}_{h}^C$ and ${\bf L}_{h}^D$ are respectively spatial discretization operators on the primal and dual meshes obtained from suitable modifications to the standard CDG discretization.

\begin{definition}
	Suppose the initial data satisfy ${\bf U}_{h}^{C}(x,0) = {\bf U}_{h}^{s,C}$, ${\bf U}_{h}^{D}(x,0) = {\bf U}_{h}^{s,D}$,
	a CDG scheme is defined to be WB if the flux and source term approximations balance each other,
	namely $ {\bf L}_{h}^C ( {\bf U}_{h}^{s,C}, {\bf U}_{h}^{s,D} )  = {\bf 0} $,  ${\bf L}_{h}^D ( {\bf U}_{h}^{s,D}, {\bf U}_{h}^{s,C} )  = {\bf 0} $.
\end{definition}

\begin{definition}
	A CDG scheme is defined to be positivity-preserving if its numerical solutions ${\bf U}_{h}^{C}$,
	${\bf U}_{h}^{D}$ stay in sets $\mathbb{G}_{h}^{C,k}$ and $\mathbb{G}_{h}^{D,k}$, respectively. For clarity,
	if a CDG scheme preserves the numerical solutions in set $\overline{\mathbb{G}}_{h}^{C,k}$ and $\overline{\mathbb{G}}_{h}^{D,k}$, then we say {that} it satisfies a weak positivity-preserving property.
\end{definition}

We aim at designing {the} high-order accurate CDG schemes that satisfy the  WB and
positivity-preserving properties simultaneously. This goal will be achieved by following three steps:
\begin{itemize}
	\item  First,  we seek spatial discretization operators  ${\bf L}_{h}^C({\bf U}_{h}^C, {\bf U}_{h}^D)$ and
	${\bf L}_{h}^D({\bf U}_{h}^D, {\bf U}_{h}^C)$ satisfying both the WB property:
	\begin{equation}\label{eq:wb-property}
		{\bf L}_{h}^C ( {\bf U}_{h}^{s,C}, {\bf U}_{h}^{s,D} )  = 0 \,, \quad
		{\bf L}_{h}^D ( {\bf U}_{h}^{s,D}, {\bf U}_{h}^{s,C} )  = 0 \,,
	\end{equation}
	and the weak positivity-preserving property: if ${\bf U}_{h}^C \in \mathbb{G}_{h}^{C,k}$,
	${\bf U}_{h}^D \in \mathbb{G}_{h}^{D,k}$, then
	\begin{align}\label{eq:pp-property}
		{\bf U}_{h}^C  + \Delta t {\bf L}_{h}^C ({\bf U}_{h}^C, {\bf U}_{h}^D) \in  \overline{\mathbb{G}}_{h}^{C,k} \,,\quad
		{\bf U}_{h}^D  + \Delta t {\bf L}_{h}^D ({\bf U}_{h}^D, {\bf U}_{h}^C) \in  \overline{\mathbb{G}}_{h}^{D,k} \,,
	\end{align}
	under some CFL-type condition on time {stepsize} $\Delta t$.
	\item Second, we further discretize the ODE system (\ref{eq:ODE-system}) in time using a
	strong-stability-preserving (SSP) explicit Runge-Kutta method \cite{GottliebShuTadmor2001}.
	\item Finally, a local scaling positivity-preserving limiting procedure, which will be introduced in Section \ref{section:limiter},
	is applied to the intermediate solutions of the Runge-Kutta discretization. This procedure
	corresponds to two operators ${\Pi}_{h}^C : \overline{\mathbb{G}}_{h}^{C,k} \longrightarrow \mathbb{G}_{h}^{C,k}$
	and ${\Pi}_{h}^D : \overline{\mathbb{G}}_{h}^{D,k} \longrightarrow \mathbb{G}_{h}^{D,k}$, which satisfy
	the  conservative property
	\[
	\frac{1}{\Delta x} \int_{I_{j}} {\Pi}_{h}^C ({\bf v}) dx
	= \frac{1}{\Delta x} \int_{I_{j}} {\bf v} {\rm d}x \quad
	\forall j\,, ~ \forall {\bf v} \in \overline{\mathbb{G}}_{h}^{C,k}\,,
	\]
	\[
	\frac{1}{\Delta x} \int_{I_{j+\frac{1}{2}}} {\Pi}_{h}^D ({\bf w}) dx
	= \frac{1}{\Delta x} \int_{I_{j+\frac{1}{2}}} {\bf w} {\rm d}x \quad
	\forall j\,, ~ \forall {\bf w} \in \overline{\mathbb{G}}_{h}^{D,k}\,.
	\]
	{For the first order CDG scheme $(k = 0)$, both ${\Pi}_{h}^C$ and ${\Pi}_{h}^D$ become the identity operators, so that
	the positivity-preserving  limiting procedure is only operated}  for {the} high-order CDG schemes with $k \geq 1${.}

\end{itemize}

Let ${\bf U}_{h}^{C,n}$ and ${\bf U}_{h}^{D,n}$ denote the numerical solution at time $t = t_n$. The resulting fully
discrete positivity-preserving  WB CDG method, with the third-order accurate SSP Runge-Kutta time discretization as example, is then given as follows.
\begin{itemize}
	\item Set ${\bf U}_{h}^{C,0} = {\Pi}_{h}^C \big[ {\bf U}_{h}^{C}(x,0) \big]$ and ${\bf U}_{h}^{D,0} = {\Pi}_{h}^D \big[ {\bf U}_{h}^{D}(x,0) \big]$, where ${\bf U}_{h}^{C}(x,0)$ and ${\bf U}_{h}^{D}(x,0)$ denote the
	novel projections of {the} initial data ${\bf U}(x,0)$ onto the space
	$[\mathbb{V}_{h}^{C,k}]^3$ and $[\mathbb{V}_{h}^{D,k}]^3$, respectively.
	\item For $n = 0,\cdots, N_t - 1$, compute ${\bf U}_{h}^{C,n+1}$ and ${\bf U}_{h}^{D,n+1}$ as follows:
	\begin{description}
		\item[(i)] Compute the intermediate solutions ${\bf U}_{h}^{C,(1)}$ and ${\bf U}_{h}^{D,(1)}$ via
		\[
		{\bf U}_{h}^{C,(1)} = {\Pi}_{h}^C \big[ {\bf U}_{h}^{C,n}
		+ \Delta t_n {\bf L}_{h}^C ({\bf U}_{h}^{C,n}, {\bf U}_{h}^{D,n}) \big] \,,
		\]
		\[
		{\bf U}_{h}^{D,(1)} = {\Pi}_{h}^D \big[ {\bf U}_{h}^{D,n}
		+ \Delta t_n {\bf L}_{h}^D ({\bf U}_{h}^{D,n}, {\bf U}_{h}^{C,n}) \big] \,.
		\]
		\item[(ii)] Compute the intermediate solutions ${\bf U}_{h}^{C,(2)}$ and ${\bf U}_{h}^{D,(2)}$ via
		\[
		{\bf U}_{h}^{C,(2)} = {\Pi}_{h}^C \bigg[ \frac{3}{4} {\bf U}_{h}^{C,n} + \frac{1}{4} \bigg( {\bf U}_{h}^{C,(1)}
		+ \Delta t_n {\bf L}_{h}^C ({\bf U}_{h}^{C,(1)}, {\bf U}_{h}^{D,(1)}) \bigg) \bigg] \,,
		\]
		\[
		{\bf U}_{h}^{D,(2)} = {\Pi}_{h}^D \bigg[ \frac{3}{4} {\bf U}_{h}^{D,n} + \frac{1}{4} \bigg( {\bf U}_{h}^{D,(1)}
		+ \Delta t_n {\bf L}_{h}^D ({\bf U}_{h}^{D,(1)}, {\bf U}_{h}^{C,(1)}) \bigg) \bigg] \,.
		\]
		\item[(iii)] Compute ${\bf U}_{h}^{C,n+1}$ and ${\bf U}_{h}^{D,n+1}$ via
		\[
		{\bf U}_{h}^{C,n+1} = {\Pi}_{h}^C \bigg[ \frac{1}{3} {\bf U}_{h}^{C,n} + \frac{2}{3} \bigg( {\bf U}_{h}^{C,(2)}
		+ \Delta t_n {\bf L}_{h}^C ({\bf U}_{h}^{C,(2)}, {\bf U}_{h}^{D,(2)}) \bigg) \bigg] \,,
		\]
		\[
		{\bf U}_{h}^{D,n+1} = {\Pi}_{h}^D \bigg[ \frac{1}{3} {\bf U}_{h}^{D,n} + \frac{2}{3} \bigg( {\bf U}_{h}^{D,(2)}
		+ \Delta t_n {\bf L}_{h}^D ({\bf U}_{h}^{D,(2)}, {\bf U}_{h}^{C,(2)}) \bigg) \bigg] \,,
		\]
	\end{description}
	where the SSP Runge-Kutta method has been written into a convex combination of the forward Euler method.
	In the following subsections, we will describe in detail the operators ${\bf L}_{h}^C$, ${\bf L}_{h}^D$ and
	${\Pi}_{h}^C$, ${\Pi}_{h}^D${, and  we shall state} the WB property (\ref{eq:wb-property})
	of our modified CDG discretization in Theorem \ref{theorem:wb-property} and the weak positivity-preserving
	property (\ref{eq:pp-property}) in Theorem \ref{theorem:pp-property}.
\end{itemize}

\subsection{Spatial discretization operators ${\bf L}_{h}^C$ and ${\bf L}_{h}^D$ } \label{section:spatial-discretization}

For convenience, we will mainly present the CDG spatial discretization on the primal mesh in detail, as that on the dual mesh is very similar.

\subsubsection{WB dissipation term} \label{section:dissipation}

As well-known,
the numerical dissipation term
\[
{\rm d}_{j}^{C} ( {\bf U}_{h}^{C}, {\bf U}_{h}^{D},v) = \frac{1}{\tau_{max}}\int_{I_{j}}
({\bf U}_{h}^{D} - {\bf U}_{h}^{C}) v {\rm d}x \,,  \quad v \in \mathbb{V}_{h}^{C,k} \,,
\]
in the standard CDG method (\ref{eq:CDG-primal}) is essential for {the} numerical stability \cite{Liu2008CDG}.
%of solving nolinear equations (\ref{eq:system-1d}).
However, it would destroy the WB property at the steady state. To address this issue, we propose to modify it into
\begin{equation}\label{eq:wb-dissipation}
	\tilde{{\rm d}}_{j}^{C} ({\bf U}_{h}^{C}, {\bf U}_{h}^{D},v) =
	\frac{1}{\tau_{max}} \int_{I_{j}} ({\bf U}_{h}^{D}   - {\bf U}_{h}^{C}) v {\rm d}x +
	\frac{1}{\tau_{max}} \int_{I_{j}} ({\bf U}_{h}^{s,C} - {\bf U}_{h}^{s,D}) v {\rm d}x  \,,
\end{equation}
so that
\begin{equation*}%\label{eq:WB-dis}
	\tilde{{\rm d}}_{j}^{C} ({\bf U}_{h}^{s,C}, {\bf U}_{h}^{s,D},v) =
	\frac{1}{\tau_{max}} \int_{I_{j}} ({\bf U}_{h}^{s,D} - {\bf U}_{h}^{s,C}) v {\rm d}x +
	\frac{1}{\tau_{max}} \int_{I_{j}} ({\bf U}_{h}^{s,C} - {\bf U}_{h}^{s,D} ) v {\rm d}x = 0 \,.
\end{equation*}

\begin{remark}\label{rem:L2-00}
Thanks to the key property \eqref{eq:average-identity} of our novel projection, it holds
$$
\tilde{{\rm d}}_{j}^{C} ({\bf U}_{h}^{C}, {\bf U}_{h}^{D}, 1) = {\rm d}_{j}^{C} ( {\bf U}_{h}^{C}, {\bf U}_{h}^{D},1),
$$
which implies that the modified numerical dissipation term \eqref{eq:wb-dissipation} would not destroy
the conservative property of our schemes. It is worth noting that using the standard $L^2$-projection for ${\bf U}_{h}^{s,C} $ and $ {\bf U}_{h}^{s,D}$ does not lead to this desirable feature.
\end{remark}

In addition to the above-mentioned advantages of the modified numerical dissipation term \eqref{eq:wb-dissipation}, we can also prove that our modification of dissipation term will not affect the high-order accuracy of the CDG schemes as indicated by the following proposition.

\begin{proposition}
	Assume ${\bf U}^{s} \in W_{2}^{k+1}(\Omega)$, $\tau_{max} = {\mathcal O}(\Delta x)$, then on each cell $I_{j}$, it holds that
	\begin{align*}
	    \left|\tilde{{\rm d}}_{j}^{C} ({\bf U}_{h}^{C}, {\bf U}_{h}^{D}, v) - {\rm d}_{j}^{C} ({\bf U}_{h}^{C},{\bf U}_{h}^{D},v)\right| =
		\frac{1}{\tau_{max}} \left| \int_{I_{j}} ({\bf U}_{h}^{s,C} - {\bf U}_{h}^{s,D} ) v {\rm d}x \right| \leq  {\mathcal O}((\Delta x) ^{k+1}) \,.
	\end{align*}
\end{proposition}

\begin{proof}
	Using the triangular inequality gives
	\begin{align} \notag
		\frac{1}{\tau_{max}} \left| \int_{I_{j}} ({\bf U}_{h}^{s,C} - {\bf U}_{h}^{s,D} ) v {\rm d}x \right|
		\leq \frac{1}{\tau_{max}}  \int_{I_{j}} |({\bf U}_{h}^{s,C} - {\bf U}^{s} ) v | {\rm d}x
		+ \frac{1}{\tau_{max}} \int_{I_{j}} |({\bf U}^{s} - {\bf U}_{h}^{s,D} ) v | {\rm d}x  \,. \notag
	\end{align}
	Under the condition $|{\bf U}^s|_{W_{2}^{k+1}(I_{j})} \lVert v \lVert_{L^{2}(I_{j})}
	= {\mathcal O}(\Delta x)$, Theorem \ref{theorem:operator-accuracy} implies
	\begin{align} \notag
		\frac{1}{\tau_{max}}  \int_{I_{j}} |({\bf U}_{h}^{s,C} - {\bf U}^{s}) v | {\rm d}x  &\leq
		\frac{1}{\tau_{max}} \lVert {\bf U}_{h}^{s,C} - {\bf U}^{s}  \lVert_{L^{2}(I_{j})} \lVert  v  \lVert_{L^{2}(I_{j})} \\ \notag
		&\leq C_{k+1} (\Delta x)^{k+1}
		\frac{|{\bf U}^{s}|_{W_{2}^{k+1}(I_{j})} \lVert v \lVert_{L^{2}(I_{j})}}{\tau_{max}}  = {\mathcal O}((\Delta x) ^{k+1}) \,.
	\end{align}
	Similarly, one has
	\begin{align} \notag
		& \frac{1}{\tau_{max}}  \int_{I_{j}} |({\bf U}^{s} - {\bf U}_{h}^{s,D} ) v | {\rm d}x  =
		\frac{1}{\tau_{max}}  \Big( \int_{I_{j}^-} |({\bf U}^{s} - {\bf U}_{h}^{s,D} ) v | {\rm d}x
		+ \int_{I_{j}^+} |({\bf U}^{s} - {\bf U}_{h}^{s,D} ) v | {\rm d}x \Big) \\ \notag
		\leq &  \frac{1}{\tau_{max}} \Big(
		\lVert {\bf U}^{s} - {\bf U}_{h}^{s,D} \lVert_{L^{2}(I_{j-\frac{1}{2}})} \lVert  v  \lVert_{L^{2}(I_{j}^{-})}
	  + \lVert {\bf U}^{s} - {\bf U}_{h}^{s,D} \lVert_{L^{2}(I_{j+\frac{1}{2}})} \lVert  v  \lVert_{L^{2}(I_{j}^{+})}\Big)
	  = {\mathcal O}((\Delta x) ^{k+1}) \,.
	\end{align}
	Combining these results, we conclude that
	$\frac{1}{\tau_{max}} \left| \int_{I_{j}} ({\bf U}_{h}^{s,C} - {\bf U}_{h}^{s,D} ) v {\rm d}x \right| \leq  {\mathcal O}((\Delta x) ^{k+1})$.
	%which indicates that our modification of the dissipation term will not affect the spatial accuracy of the CDG method.
\end{proof}

\subsubsection{Numerical flux and source term} \label{section:source-approximation}

Let $ \mathbb{Q}_{j}^{1,x} = \big\{ x_{j}^{1,\alpha} \big\}_{\alpha = 1}^{N} $
and $ \mathbb{Q}_{j}^{2,x} = \big\{ x_{j}^{2,\alpha} \big\}_{\alpha = 1}^{N} $
denote the $N$-point Gauss quadrature nodes transformed into the interval $\big[x_{j-\frac{1}{2}}, x_{j}\big]$
and  $\big[x_{j}, x_{j+\frac{1}{2}}\big]$, respectively, and $\big\{\omega_{\alpha}\big\}_{\alpha = 1}^{N}$
are the associated weights satisfying $\sum_{\alpha=1}^{N} \omega_{\alpha} = 1$, with $N \geq k+1$ for the
CDG accuracy requirement.

Because ${\bf U}_{h}^{D}(x,t)$ can be discontinuous at $x = x_{j}$, the element integral
$\int_{I_{j}} {\bf F}({\bf U}_{h}^{D}) v_{x} {\rm d}x$ in (\ref{eq:CDG-primal}) is usually divided into two parts
\begin{equation*}
	\int_{I_{j}}     {\bf F}({\bf U}_{h}^{D}) v_{x} {\rm d}x =
	\int_{I_{j}^{-}} {\bf F}({\bf U}_{h}^{D}) v_{x} {\rm d}x
	+ \int_{I_{j}^{+}} {\bf F}({\bf U}_{h}^{D}) v_{x} {\rm d}x \,,
\end{equation*}
which is then approximately by the numerical quadrature rule
\begin{equation*}
	\int_{I_{j}} {\bf F}({\bf U}_{h}^{D}) v_{x} {\rm d}x  \approx  \frac{\Delta x}{2}
	\sum_{\kappa = 1}^{2}\sum_{\alpha = 1}^{N} \omega_{\alpha}
	{\bf F}\big({\bf U}_{h}^{D}(x_{j}^{\kappa,\alpha}) \big) v_{x}(x_{j}^{\kappa,\alpha})  \,.
\end{equation*}

Next, we introduce a non-standard approximation to the source term integral in (\ref{eq:CDG-primal}) to achieve the WB property. 
This idea is similar to \cite{Xing2016DGWB} but has some key differences owing to carefully accommodate the positivity-preserving property; see Remark \ref{Wu:Remark}. 
 {Reformulate} and decompose the integral of the source term in the momentum equation as 
\begin{align*}
	\int_{I_{j}} S_{2}({\bf U}, \phi_{x}) v {\rm d}x  = - \int_{I_{j}} \rho \phi_{x} v {\rm d}x
	= \int_{I_{j}}  \frac{\rho}{\rho^s} p^s_{x} v {\rm d}x
	= \int_{I_{j}}  \Bigg( \frac{\rho}{\rho^s} -  \frac{ \overline{\rho}_{j} }{ \overline{\rho}^s_{j} }
	+ \frac{ \overline{\rho}_{j} }{ \overline{\rho}^s_{j} }  \Bigg)  p^s_{x} v {\rm d}x \,,
\end{align*}
where $\overline{\rho}_{j} = \frac{1}{\Delta x}\int_{I_{j}} \rho {\rm d}x$ denotes the cell average. Our numerical approximation to the source term takes the form of
\begin{align*}
	\int_{I_{j}} S_{2}( {\bf U}_{h}^{D}, (\phi_{h}^{D})_{x} ) v {\rm d}x &\approx  \int_{I_{j}}
	\Bigg( \frac{ \rho_{h}^{D} }{ \rho_{h}^{s,D} }
	- \frac{\overline{ (\rho_{h}^{D}) }_{j} }{ \overline{ (\rho_{h}^{s,D}) }_{j}}  \Bigg)
	(p_{h}^{s,D})_{x} v {\rm d}x  + \frac{\overline{ (\rho_{h}^{D}) }_{j} }{ \overline{ (\rho_{h}^{s,D}) }_{j} }
	\int_{I_{j}} (p_{h}^{s,D})_{x} v {\rm d}x   \,.
\end{align*}
Applying integration by parts {gives}
\begin{align*}
	\int_{I_{j}} (p_{h}^{s,D})_{x} v {\rm d}x
	=& - \int_{I_{j}} p_{h}^{s,D} v_{x} {\rm d}x - \Big( p_{h}^{s,D}(x_{j}^{+}) - p_{h}^{s,D}(x_{j}^{-})\Big) v(x_{j}) \\ \notag
	&+ \Big(p_{h}^{s,D}(x_{j+\frac{1}{2}}) v(x_{j+\frac{1}{2}}^{-})
	- p_{h}^{s,D}(x_{j-\frac{1}{2}}) v(x_{j-\frac{1}{2}}^{+}) \Big)  \\ \notag
	\approx & - \int_{I_{j}} p_{h}^{s,D} v_{x} {\rm d}x + \Big(p_{h}^{s,D}(x_{j+\frac{1}{2}}) v(x_{j+\frac{1}{2}}^{-})
	- p_{h}^{s,D}(x_{j-\frac{1}{2}}) v(x_{j-\frac{1}{2}}^{+}) \Big)  \,,
\end{align*}
where the following term has been omitted
\begin{align*}
	p_{h}^{s,D}(x_{j}^{+}) - p_{h}^{s,D}(x_{j}^{-})
	= \big(p_{h}^{s,D}(x_{j}^{+}) - p^s(x_{j}) \big) + \big( p^s(x_{j}) - p_{h}^{s,D}(x_{j}^{-}) \big)
	= {\mathcal O}(\Delta x^{k+1})  \,. 
\end{align*}
This leads to
\begin{align*}% \notag
	\int_{I_{j}} S_{2}( {\bf U}_{h}^{D}, (\phi_{h}^{D})_{x} ) v {\rm d}x &\approx  \int_{I_{j}}
	\Bigg( \frac{ \rho_{h}^{D} }{ \rho_{h}^{s,D} } - \frac{\overline{ (\rho_{h}^{D}) }_{j} }{ \overline{ (\rho_{h}^{s,D}) }_{j}}  \Bigg)
	(p_{h}^{s,D})_{x} v {\rm d}x  - \frac{\overline{ (\rho_{h}^{D}) }_{j} }{ \overline{ (\rho_{h}^{s,D}) }_{j} }
	\int_{I_{j}} p_{h}^{s,D} v_{x} {\rm d}x  \\
	&+ \frac{\overline{ (\rho_{h}^{D}) }_{j} }{ \overline{ (\rho_{h}^{s,D}) }_{j} }
	\Big(p_{h}^{s,D}(x_{j+\frac{1}{2}}) v(x_{j+\frac{1}{2}}^{-})
	-   p_{h}^{s,D}(x_{j-\frac{1}{2}}) v(x_{j-\frac{1}{2}}^{+}) \Big)  \,.
\end{align*}
Therefore, the source term $\int_{I_{j}} S_{2}({\bf U}_{h}^{D}, (\phi_{h}^{D})_{x}) v {\rm d}x$ in the momentum equation can further be approximated by
\begin{align}\nonumber
	& \big\langle S_{h,2}^{D}, v \big\rangle_{j}  =
	\frac{\overline{ (\rho_{h}^{D}) }_{j} }{ \overline{ (\rho_{h}^{s,D}) }_{j} }
	\Big(p_{h}^{s,D}(x_{j+\frac{1}{2}}) v(x_{j+\frac{1}{2}}^{-})
	-  p_{h}^{s,D}(x_{j-\frac{1}{2}}) v(x_{j-\frac{1}{2}}^{+}) \Big)  
\\ %\notag
	+ &\frac{\Delta x}{2} \sum_{\kappa = 1}^{2} \sum_{\alpha = 1}^{N} \omega_{\alpha}
	\Bigg[\Bigg( \frac{ \rho_{h}^{D}(x_{j}^{\kappa,\alpha}) }{ \rho_{h}^{s,D} (x_{j}^{\kappa,\alpha}) }
	- \frac{\overline{ (\rho_{h}^{D}) }_{j} }{ \overline{ (\rho_{h}^{s,D}) }_{j} } \Bigg)
	(p_{h}^{s,D})_{x}(x_{j}^{\kappa,\alpha}) v(x_{j}^{\kappa,\alpha})
	- \frac{\overline{ (\rho_{h}^{D}) }_{j} }{ \overline{ (\rho_{h}^{s,D}) }_{j} }
	p_{h}^{s,D}(x_{j}^{\kappa,\alpha}) v_{x}(x_{j}^{\kappa,\alpha}) \Bigg]    \,.
\label{eq:source-momentum}
\end{align}
Similarly, we approximate the source term $\int_{I_{j}} S_{3}({\bf U}_{h}^{D}, (\phi_{h}^{D})_{x}) v {\rm d}x$ in the energy equation by
\begin{align*}%\label{eq:source-energy}
	& \big\langle S_{h,3}^{D}, v \big\rangle_{j}  =
	\frac{\overline{ (m_{h}^{D}) }_{j} }{ \overline{ (\rho_{h}^{s,D}) }_{j} }
	\Big(p_{h}^{s,D}(x_{j+\frac{1}{2}}) v(x_{j+\frac{1}{2}}^{-})
	-  p_{h}^{s,D}(x_{j-\frac{1}{2}}) v(x_{j-\frac{1}{2}}^{+}) \Big)  
\\ %\notag
	+ &\frac{\Delta x}{2} \sum_{\kappa = 1}^{2} \sum_{\alpha = 1}^{N} \omega_{\alpha}
	\Bigg[\Bigg( \frac{ m_{h}^{D}(x_{j}^{\kappa,\alpha}) }{ \rho_{h}^{s,D} (x_{j}^{\kappa,\alpha}) }
	- \frac{\overline{ (m_{h}^{D}) }_{j} }{ \overline{ (\rho_{h}^{s,D}) }_{j} } \Bigg)
	(p_{h}^{s,D})_{x}(x_{j}^{\kappa,\alpha}) v(x_{j}^{\kappa,\alpha})
	- \frac{\overline{ (m_{h}^{D}) }_{j} }{ \overline{ (\rho_{h}^{s,D}) }_{j} }
	p_{h}^{s,D}(x_{j}^{\kappa,\alpha}) v_{x}(x_{j}^{\kappa,\alpha}) \Bigg]  \,.
\end{align*}

\subsubsection{Semi-discrete WB CDG schemes}

Combining {the} modified dissipation term in {Section \ref{section:dissipation}} with the discrete source term in Section \ref{section:source-approximation}, we obtain the final semi-discrete WB CDG method on the primal mesh
\begin{align}\nonumber
	\int_{I_{j}} \frac{\partial {\bf U}_{h}^{C}}{\partial t} v {\rm d}x
	& = \frac{1}{\tau_{max}} \int_{I_{j}} ({\bf U}_{h}^{D} - {\bf U}_{h}^{C}) v {\rm d}x +
	\frac{1}{\tau_{max}} \int_{I_{j}} ({\bf U}_{h}^{s,C} - {\bf U}_{h}^{s,D} ) v {\rm d}x   \\ \notag
	&  - \Big( {\bf F}({\bf U}_{h}^{D}(x_{j+\frac{1}{2}}))v(x_{j+\frac{1}{2}}^{-})
	- {\bf F}({\bf U}_{h}^{D}(x_{j-\frac{1}{2}}))v(x_{j-\frac{1}{2}}^{+}) \Big)  \\   
	&  + \frac{\Delta x}{2} \sum_{\kappa = 1}^{2} \sum_{\alpha = 1}^{N} \omega_{\alpha}
	{\bf F} \big({\bf U}_{h}^{D}(x_{j}^{\kappa,\alpha}) \big) v_{x}(x_{j}^{\kappa,\alpha})
	+ \big\langle {\bf S}_{h}^{D}, v \big\rangle_{j} \quad \forall v \in \mathbb{V}_{h}^{C,k}  \,,
\label{eq:CDG-modify-primal}
\end{align}
where $ \big\langle {\bf S}_{h}^{D}, v \big\rangle_{j} = \big(0, \big\langle S_{h,2}^{D}, v \big\rangle_{j}, \big\langle S_{h,3}^{D}, v \big\rangle_{j} \big) ^ \top$.

The WB CDG  spatial discretization on the dual mesh is very similar. Denote ${\bf Q} = (0, \rho, m)^\top$, one has ${\bf S}({\bf U}, \phi_{x} ) = -\phi_{x} {\bf Q} $,
and the modified source term approximation $\big\langle {\bf S}_{h}^{C}, w \big\rangle_{j+\frac{1}{2}}$ is given by
\begin{align*}%\label{eq:source-dual}
	& \big\langle {\bf S}_{h}^{C}, w \big\rangle_{j+\frac{1}{2}}  =
	\frac{\overline{ ({\bf Q}_{h}^{C}) }_{j+\frac{1}{2}} }{ \overline{ (\rho_{h}^{s,C}) }_{j+\frac{1}{2}} }
	\Big(p_{h}^{s,C}(x_{j+1}) w(x_{j+1}^{-})
	-  p_{h}^{s,C}(x_{j}) w(x_{j}^{+}) \Big)  \\ \notag
	+ &\frac{\Delta x}{2} \sum_{\kappa = 1}^{2} \sum_{\alpha = 1}^{N} \omega_{\alpha}
	\Bigg[\Bigg( \frac{ {\bf Q}_{h}^{C}(x_{j+\frac{1}{2}}^{\kappa,\alpha}) }{ \rho_{h}^{s,C} (x_{j+\frac{1}{2}}^{\kappa,\alpha}) }
	- \frac{\overline{ ({\bf Q}_{h}^{C}) }_{j+\frac{1}{2}} }{ \overline{ (\rho_{h}^{s,C}) }_{j+\frac{1}{2}} } \Bigg)
	((p_{h}^{s,C})_{x} w)(x_{j+\frac{1}{2}}^{\kappa,\alpha})
	- \frac{\overline{ ({\bf Q}_{h}^{C}) }_{j+\frac{1}{2}} }{ \overline{ (\rho_{h}^{s,C}) }_{j+\frac{1}{2}} }
	(p_{h}^{s,C}  w_{x}  )(x_{j+\frac{1}{2}}^{\kappa,\alpha}) \Bigg]  \,,
\end{align*}
where $w \in \mathbb{V}_{h}^{D,k} $,
$\big\{ x_{j+\frac{1}{2}}^{1,\alpha} \big\}_{\alpha = 1}^{N}  = \mathbb{Q}_{j}^{2,x}$ and
$\big\{ x_{j+\frac{1}{2}}^{2,\alpha} \big\}_{\alpha = 1}^{N}  = \mathbb{Q}_{j+1}^{1,x}$
denote the Gauss quadrature nodes transformed into the interval $\big[x_{j}, x_{j+\frac{1}{2}}\big]$
and  $\big[ x_{j+\frac{1}{2}}, x_{j+1} \big]$, respectively. Then the WB CDG method on the dual mesh reads
\begin{align}\nonumber
	\int_{ I_{j+\frac{1}{2}} } \frac{\partial {\bf U}_{h}^{D}}{\partial t} w {\rm d}x
	& = \frac{1}{\tau_{max}} \int_{I_{j+\frac{1}{2}}} ({\bf U}_{h}^{C} - {\bf U}_{h}^{D}) w {\rm d}x +
	\frac{1}{\tau_{max}} \int_{I_{j+\frac{1}{2}}} ({\bf U}_{h}^{s,D} - {\bf U}_{h}^{s,C}) w {\rm d}x   \\ \notag
	&  - \Big( {\bf F}({\bf U}_{h}^{C}(x_{j+1}))w(x_{j+1}^{-})
	- {\bf F}({\bf U}_{h}^{C}(x_{j}))w(x_{j}^{+}) \Big)  \\   
	&  + \frac{\Delta x}{2} \sum_{\kappa = 1}^{2} \sum_{\alpha = 1}^{N} \omega_{\alpha}
	{\bf F} \big({\bf U}_{h}^{C}(x_{j+\frac{1}{2}}^{\kappa,\alpha}) \big) w_{x}(x_{j+\frac{1}{2}}^{\kappa,\alpha})
	+ \big\langle {\bf S}_{h}^{C}, w \big\rangle_{j+\frac{1}{2}}  \quad \forall w \in \mathbb{V}_{h}^{D,k} \,.
\label{eq:CDG-modify-dual}
\end{align}
As the standard CDG schemes (\ref{eq:CDG-primal}) and (\ref{eq:CDG-dual}), the semi-discrete WB CDG schemes $(\ref{eq:CDG-modify-primal})$ and $(\ref{eq:CDG-modify-dual})$  can be rewritten in the ODE form as
\begin{equation*}
	\frac{\mathrm{d} {\bf U}_{h}^C }{\mathrm{d} t} = {\bf L}_{h}^C ({\bf U}_{h}^C, {\bf U}_{h}^D) \,, \quad
	\frac{\mathrm{d} {\bf U}_{h}^D }{\mathrm{d} t} = {\bf L}_{h}^D ({\bf U}_{h}^D, {\bf U}_{h}^C) \,,
\end{equation*}
after choosing suitable {bases} of $\mathbb{V}_{h}^{C,k}, \mathbb{V}_{h}^{D,k}$ and representing
${\bf U}_{h}^C, {\bf U}_{h}^D$ as linear combinations of the basis functions.

\begin{remark}\label{Wu:Remark}
	It is worth noting that the above WB discretization has carefully accommodated the positivity-preserving property. For example,
	if we are only concerned with the WB property (see, for example, \cite{Xing2016DGWB}), the choice of $\frac{\overline{ (\rho_{h}^{D}) }_{j} }{\overline{ (\rho_{h}^{s,D}) }_{j}}$ in (\ref{eq:source-momentum}) is not unique and can be replaced with any other suitable term that can reduce to one at the steady state (\ref{eq:steady-state-1d}), e.g.~$\frac{\rho_{h}^{D}(x_{\beta})}{\rho_{h}^{s,D}(x_{\beta})}$ with an arbitrary $x_{\beta} \in I_{j}$. 
However, our analyses show that choosing $ \frac{\overline{ (\rho_{h}^{D}) }_{j} }{ \overline{ (\rho_{h}^{s,D}) }_{j} } $  in the source term approximation of the momentum equation {is} advantageous for achieving the positivity-preserving
	property under a milder and more concise CFL condition. Similar consideration is also applied to the integral of the source term in the energy equation. For example, one can simply approximate
	\[
	%\int_{I_{j}} S_{3}({\bf U}_{h}^{D},\phi_{h}^{D}) v {\rm d}x =
	- \int_{I_{j}} (\rho u)_{h}^{D} (\phi_{h}^{D})_{x} v {\rm d}x
	\]
	by using any standard quadrature rule and does not affect the WB property. However, our analyses indicate that it is crucial to employ a ``unified'' discretization for {the} source terms in the momentum and energy {equations to simultaneously} accommodate the positivity-preserving property.
\end{remark}

%\begin{remark}
%	For the integral of the source term in the third equation, a simply approximation is
%	\[
%	\int_{I_{j}} S_{3}({\bf U}_{h}^{D},\phi_{h}^{D}) v {\rm d}x =
%	- \int_{I_{j}} (\rho u)_{h}^{D} (\phi_{h}^{D})_{x} v {\rm d}x
%	\]
%	and the standard quadrature rule is used to evaluate this integral \cite{Xing2018DGWB}. But in our  positivity-preserving WB CDG methods, we will approximate the source term in the energy equation as the momentum equation and a uniform formulation for the source term approximation is obtained \cite{wu2021uniformly}. This play a key idea in prove the positivity-preserving property of the WB CDG methods.
%\end{remark}

\subsection{Proofs of WB and positivity-preserving properties}

\subsubsection{WB property}

\begin{theorem}\label{theorem:wb-property}
	For the one-dimensional Euler equations $(\ref{eq:system-1d})$ under the gravitational field, the modified semi-discrete CDG schemes, given by $(\ref{eq:CDG-modify-primal})$ and  $(\ref{eq:CDG-modify-dual})$,  are WB for a general stationary hydrostatic  solution $(\ref{eq:steady-state-1d})$.
\end{theorem}

\begin{proof}
	Suppose that the initial solution is ${\bf U}^{s}$. By the construction of ${\bf U}_{h}^{s,C}, {\bf U}_{h}^{s,D}$, one has
	\[
	{\bf U}_{h}^{C}(x,0) =  {\bf U}_{h}^{s,C} \,, \quad {\bf U}_{h}^{D}(x,0) = {\bf U}_{h}^{s,D} \,,
	\]
	and
	\[
	\rho_{h}^{D}(x,0) = \rho_{h}^{s,D} \,, \quad u_{h}^{D}(x,0) = 0 \,, \quad p_{h}^{D}(x,0) = p_{h}^{s,D} \,.
	\]
	The modified dissipation term becomes
	\begin{align*}
		\tilde{\rm d}_{j}^{C} ({\bf U}_{h}^{C}, {\bf U}_{h}^{D},v) =
		\frac{1}{\tau_{max}} \int_{I_{j}} ({\bf U}_{h}^{D}   - {\bf U}_{h}^{C})    v {\rm d}x +
		\frac{1}{\tau_{max}} \int_{I_{j}} ({\bf U}_{h}^{s,C} - {\bf U}_{h}^{s,D} ) v {\rm d}x  = 0 \,.
	\end{align*}
	It is observed that the WB property holds for the density and energy equations, as both
	the flux and source term approximations in {those} equations become zero. For the momentum equation, because
	$\frac{\overline{ (\rho_{h}^{D}) }_{j} }{\overline{ (\rho_{h}^{s,D}) }_{j}} = 1 $, the modified source term becomes
	\begin{align}\notag
		\big\langle S_{h,2}^{D}, v \big\rangle_{j} & =
		\Big(p_{h}^{s,D}(x_{j+\frac{1}{2}}) v(x_{j+\frac{1}{2}}^{-})
		-  p_{h}^{s,D}(x_{j-\frac{1}{2}}) v(x_{j-\frac{1}{2}}^{+}) \Big)
		- \frac{\Delta x}{2} \sum_{\kappa = 1}^{2} \sum_{\alpha = 1}^{N} \omega_{\alpha}
		p_{h}^{s,D}(x_{j}^{\kappa,\alpha}) v_{x}(x_{j}^{\kappa,\alpha})  \,.
	\end{align}
	Since $u = 0$, the flux term $F_{2} = \rho u^{2} + p$ reduces to $p$, and its numerical approximation is given by
	\begin{align}\notag
		& \frac{\Delta x}{2} \sum_{\kappa = 1}^{2} \sum_{\alpha = 1}^{N} \omega_{\alpha}
		F_{2} \big( {\bf U}_{h}^{D}(x_{j}^{\kappa,\alpha}) \big) v_{x}(x_{j}^{\kappa,\alpha})
		- \Big( F_{2} ({\bf U}_{h}^{D}(x_{j+\frac{1}{2}}))v(x_{j+\frac{1}{2}}^{-})
		-       F_{2} ({\bf U}_{h}^{D}(x_{j-\frac{1}{2}}))v(x_{j-\frac{1}{2}}^{+}) \Big)     \\  \notag
		= &  \frac{\Delta x}{2} \sum_{\kappa = 1}^{2} \sum_{\alpha = 1}^{N} \omega_{\alpha}
		p_{h}^{s,D}(x_{j}^{\kappa,\alpha}) v_{x}(x_{j}^{\kappa,\alpha})
		- \Big( p_{h}^{s,D}(x_{j+\frac{1}{2}}) v(x_{j+\frac{1}{2}}^{-})
		-      p_{h}^{s,D}(x_{j-\frac{1}{2}}) v(x_{j-\frac{1}{2}}^{+}) \Big) = -\big\langle S_{h,2}^{D}, v \big\rangle_{j} \,.
	\end{align}
	Therefore, the flux and source term approximations balance each other, implying
	\[
	\int_{I_{j}} \frac{\partial {\bf U}_{h}^{C}}{\partial t} v {\rm d}x = 0 \,,  \quad  \forall v \in \mathbb{V}_{h}^{C,k} \,.
	\]
	Similarly, on the dual mesh, one can establish
	\[
	\int_{I_{j+\frac{1}{2}}} \frac{\partial {\bf U}_{h}^{D}}{\partial t} w {\rm d}x = 0 \,,  \quad  \forall w \in \mathbb{V}_{h}^{D,k} \,.
	\]
	Hence our CDG schemes, given by $(\ref{eq:CDG-modify-primal})$ and  $(\ref{eq:CDG-modify-dual})$, are WB for a general stationary hydrostatic  solution $(\ref{eq:steady-state-1d})$.
\end{proof}

\subsubsection{Positivity-preserving property} \label{section:positivity-preserving}

{This subsection will} discuss the positivity-preserving property of the WB CDG schemes
(\ref{eq:CDG-modify-primal}) and (\ref{eq:CDG-modify-dual}). The WB modifications of the numerical dissipation
and source terms lead to additional difficulties in the positivity-preserving analyses, which are more complicated than that for the standard CDG method. We introduce several basic properties of the admissible state set $G$, which will be useful in our positivity-preserving analyses.

\begin{lemma}[Convexity]
	The set $G$ is a convex set.
\end{lemma}
\begin{proof}
	This property can be verified by definition and Jensen's inequality; see \cite[{Page} 8919]{Zhang2010PP}.
\end{proof}

\begin{lemma}\label{lemmma:pp-source}
	For any ${\bf U} \in G$ {and  $b \in \mathbb{R}$, the state} $ {\bf U} + \lambda {\bf S}({\bf U},b)\in G$ under the condition
	\[
	|\lambda| < \frac{1}{|b|} \sqrt{\frac{2p}{(\gamma-1)\rho}}  \, .
	\]
\end{lemma}
\begin{proof}
	The proof can be found in \cite[{Page} A476]{wu2021uniformly} for the details.
\end{proof}

\begin{lemma}\label{lemmma:pp-flux}
	For any $ {\bf U} \in G $ {and $\lambda \in \mathbb{R}$,  the state} ${\bf U} - \lambda {\bf F}({\bf U}) \in G$
	under the condition
	$$|\lambda| a_{x} ({\bf U}) \le 1, \quad \mbox{with~~~}  a_{x} ({\bf U}) := |u| + \sqrt{ \frac{\gamma p}{\rho}  }.$$
\end{lemma}
\begin{proof}
	A proof can be found in \cite[Page 8921]{Zhang2010PP}. 
    	See also 
	\cite{wu2021geometric} for another simple proof based on the GQL approach.
\end{proof}

Next, we consider the semi-discrete scheme satisfied by the cell averages of the WB CDG solution.
Denote  %the cell averages of ${\bf U}_{h}^{C}$ over cell $I_{j}$ and ${\bf U}_{h}^{D}$ over cell $I_{j+\frac{1}{2}}$ by
\[
\overline{{\bf U}}_{j}^{C}(t) = \frac{1}{\Delta x}  \int_{I_{j}} {\bf U}_{h}^{C}(x,t) \mathrm{d}x \,,  \quad
\overline{{\bf U}}_{j+\frac{1}{2}}^{D}(t) = \frac{1}{\Delta x}  \int_{I_{j+\frac{1}{2}}} {\bf U}_{h}^{D}(x,t) \mathrm{d}x \,.
\]
Taking the test function $v = 1$ in (\ref{eq:CDG-modify-primal}) and $w = 1$ in (\ref{eq:CDG-modify-dual})  
and using the identities in (\ref{eq:average-identity})
gives
\begin{align}\label{eq:CDG-modify-average-primal}
\begin{aligned}
	\frac{d \overline{ {\bf U} }_{j}^{C} }{d t} =  {\bf L}_{j}^{C} \big( {\bf U}_{h}^{C}, {\bf U}_{h}^{D} \big)
	: = \frac{\Big( \overline{ {\bf U} }_{j}^{D} - \overline{ {\bf U} }_{j}^{C} \Big)}{\tau_{max}}
	- \frac{\Big( {\bf F}({\bf U}_{j+\frac{1}{2}}^{D}) - {\bf F}({\bf U}_{j-\frac{1}{2}}^{D}) \Big)}{\Delta x}
	+ \frac{\big\langle {\bf S}_{h}^{D}, 1 \big\rangle_{j}}{\Delta x}  \,,
\\ %\end{align}
%\begin{align}\label{eq:CDG-modify-average-dual} \notag
	\frac{d \overline{ {\bf U} }_{j+\frac{1}{2}}^{D} }{d t} =  {\bf L}_{j+\frac{1}{2}}^{D} \big( {\bf U}_{h}^{D}, {\bf U}_{h}^{C} \big)
	: = \frac{\Big( \overline{ {\bf U} }_{j+\frac{1}{2}}^{C} - \overline{ {\bf U} }_{j+\frac{1}{2}}^{D} \Big)}{\tau_{max}}
	- \frac{\Big( {\bf F}({\bf U}_{j+1}^{C}) - {\bf F}({\bf U}_{j}^{C}) \Big)}{\Delta x}
	+ \frac{\big\langle {\bf S}_{h}^{C}, 1 \big\rangle_{j+\frac{1}{2}}}{\Delta x}  \,,
\end{aligned}\end{align}
where ${\bf U}_{j+\frac{1}{2}}^{D} = {\bf U}_{h}^{D}(x_{j+\frac{1}{2}})$, ${\bf U}_{j+1}^{C} =
{\bf U}_{h}^{C}(x_{j+1})$.

\begin{remark}\label{eq:L2-aaa}
Recall that the two key identities in (\ref{eq:average-identity}) are derived from
the novel projection operators $\mathcal{P}_{h}^{C}$ and $\mathcal{P}_{h}^{D}$ and do not hold for the standard $L^2$-projection.
Benefited from this remarkable feature, our modification of the numerical dissipation term
does not destroy the positivity-preserving property. 	
\end{remark}

Let $\mathbb{L}_{j}^{1,x} = \big\{ \hat{x}_{j}^{1,\beta} \big\}_{\beta = 1}^{L} $
and $\mathbb{L}_{j}^{2,x} = \big\{ \hat{x}_{j}^{2,\beta} \big\}_{\beta = 1}^{L} $
denote the Gauss-Lobatto quadrature nodes transformed into the interval $\big[x_{j-\frac{1}{2}}, x_{j}\big]$ and  $\big[x_{j}, x_{j+\frac{1}{2}}\big]$, respectively, and $\big\{\hat{\omega}_{\beta}\big\}_{\beta = 1}^{L}$
are the associated weights satisfying $\sum_{\beta=1}^{L} \hat{\omega}_{\beta} = 1$.
We take $L=\lceil \frac{k+3}2 \rceil$, which gives $2L-3\geq k$, so that
the $L$-point Gauss-Lobatto quadrature rule is exact for polynomials of degree up to $k$. For each primal cell $I_{j}$, we define the point set
\begin{equation}\label{eq:sj-primal}
	\mathbb{S}_{j} = \mathbb{Q}_{j}^{1,x} \cup \mathbb{Q}_{j}^{2,x} \cup \mathbb{L}_{j}^{1,x} \cup \mathbb{L}_{j}^{2,x} \,, \quad
\mathbb{Q}_{j} = \mathbb{Q}_{j}^{1,x} \cup \mathbb{Q}_{j}^{2,x} \,,
\end{equation}
{and the} parameter $\tilde{\alpha}_{j}^D$ as
\[
\tilde{\alpha}_{j}^D = \tilde{\alpha}_{1,j}^D + \tilde{\alpha}_{2,j}^D \,, \quad
\tilde{\alpha}_{1,j}^D = \max_{ x\in \{ x_{j-\frac{1}{2}}, x_{j+\frac{1}{2}} \} }  a_{x}({\bf U}_{h}^{D})\,,
\]
\[
\tilde{\alpha}_{2,j}^D = \frac{\hat{\omega}_{1} \Delta x}{2} \max_{x \in \mathbb{Q}_{j}}
\Bigg\{|(\hat{\phi}_{h}^{D})_{x}| \sqrt{\frac{(\gamma-1)\rho_{h}^{D}}{ 2 p_{h}^{D} } } \Bigg\} \,, \quad
(\hat{\phi}_{h}^{D})_{x}  =
\frac{ p_{h}^{s,D}(x_{j}^{-}) - p_{h}^{s,D}(x_{j}^{+}) }{ \overline{ (\rho_{h}^{s,D}) }_{j} \Delta x}
- \frac{(p_{h}^{s,D})_{x}}{ \rho_{h}^{s,D} } \,.
\]
Similarly, in each dual cell $I_{j+\frac{1}{2}}$ , we define the point set
\begin{equation}\label{eq:sj-dual}
\mathbb{S}_{j+\frac{1}{2}} = \mathbb{Q}_{j}^{2,x} \cup \mathbb{Q}_{j+1}^{1,x} \cup \mathbb{L}_{j}^{2,x} \cup \mathbb{L}_{j+1}^{1,x} \,,
\quad \mathbb{Q}_{j+\frac{1}{2}} = \mathbb{Q}_{j}^{2,x} \cup \mathbb{Q}_{j+1}^{1,x}
\end{equation}
{and the} parameter $\tilde{\alpha}_{j+\frac{1}{2}}^C$ as
\[
\tilde{\alpha}_{j+\frac{1}{2}}^C = \tilde{\alpha}_{1,j+\frac{1}{2}}^C + \tilde{\alpha}_{2,j+\frac{1}{2}}^C \,, \quad
\tilde{\alpha}_{1,j+\frac{1}{2}}^C = \max_{ x\in \{ x_{j}, x_{j+1} \} }  a_{x}({\bf U}_{h}^{C})\,,
\]
\[
\tilde{\alpha}_{2,j+\frac{1}{2}}^C = \frac{\hat{\omega}_{1} \Delta x}{2} \max_{x \in \mathbb{Q}_{j+\frac{1}{2}}}
\Bigg\{|(\hat{\phi}_{h}^{C})_{x}| \sqrt{\frac{(\gamma-1)\rho_{h}^{C}}{ 2 p_{h}^{C} } } \Bigg\} \,, \quad
(\hat{\phi}_{h}^{C})_{x}  =
\frac{ p_{h}^{s,C}(x_{j+\frac{1}{2}}^{-}) - p_{h}^{s,C}(x_{j+\frac{1}{2}}^{+}) }{ \overline{ (\rho_{h}^{s,C}) }_{j+\frac{1}{2}} \Delta x}
- \frac{(p_{h}^{s,C})_{x}}{ \rho_{h}^{s,C} } \,.
\]
Then we have the following CFL-type condition for the high-order CDG schemes (\ref{eq:CDG-modify-primal}) and
(\ref{eq:CDG-modify-dual}) to be positivity-preserving.

\begin{theorem}\label{theorem:pp-property}
	Assume that the numerical solutions ${\bf U}_{h}^{C}(x,t), {\bf U}_{h}^{D}(x,t)$ and the projected stationary hydrostatic solutions
	${\bf U}_{h}^{s,C}(x), {\bf U}_{h}^{s,D}(x)$ satisfy
	\begin{equation}\label{pp-condition}
		\begin{cases}
			{\bf U}_{h}^{C}(x,t)   \in G,~~~ {\bf U}_{h}^{s,C}(x) \in G & \forall x \in \mathbb{S}_{j},~ \forall j,
			\\
			{\bf U}_{h}^{D}(x,t)   \in G,~~~ {\bf U}_{h}^{s,D}(x) \in G
			& \forall x \in \mathbb{S}_{j+\frac{1}{2}},~ \forall j.
		\end{cases}
	\end{equation}
	If $\overline{ {\bf U} }_{j}^{C}, \overline{ {\bf U} }_{j+\frac{1}{2}}^{D} \in G$, then
	the weak positivity-preserving property
	\begin{align*}
		\overline{ {\bf U} }_{j}^{C} + \Delta t  {\bf L}_{j}^{C} \big( {\bf U}_{h}^{C}, {\bf U}_{h}^{D} \big) \in G, \quad
		\overline{ {\bf U} }_{j+\frac{1}{2}}^{D} + \Delta t  {\bf L}_{j+\frac{1}{2}}^{D} \big( {\bf U}_{h}^{D}, {\bf U}_{h}^{C} \big) \in G,
		\quad  \forall j,
	\end{align*}
	holds under the CFL-type condition
	\begin{align*}%\label{CFL-1d}
		\frac{\Delta t}{\Delta x}  \tilde{\alpha}_{x}  < \frac{\theta \hat{\omega}_{1} }{2} \,, \quad
		\tilde{\alpha}_{x} = \max_{j} \max \{\tilde{\alpha}_{j}^D, \tilde{\alpha}_{j+\frac{1}{2}}^C\}  \,, \quad
		\theta = \frac{\Delta t}{\tau_{max}} \in (0,1] \,.
	\end{align*}
\end{theorem}

\begin{proof}
	Using (\ref{eq:CDG-modify-average-primal}) gives
	\begin{align*}
		\overline{ {\bf U} }_{j}^{C} + \Delta t  {\bf L}_{j}^{C} \big( {\bf U}_{h}^{C}, {\bf U}_{h}^{D} \big)
		&= (1-\theta)\overline{ {\bf U} }_{j}^{C} + \Big[ \eta\theta \overline{ {\bf U} }_{j}^{D}
		+ \lambda_{x} \big\langle {\bf S}_{h}^{D}, 1 \big\rangle_{j} \Big]\\ \notag
		&+ \Big[ (1-\eta)\theta \overline{ {\bf U} }_{j}^{D}
		- \lambda_{x}\Big( {\bf F}( {\bf U}_{j+\frac{1}{2}}^{D}) - {\bf F}({\bf U}_{j-\frac{1}{2}}^{D}) \Big) \Big] \\  \notag
		&= (1-\theta) \overline{ {\bf U} }_{j}^{C} + \eta\theta {\bf L}_{h,{\bf S}}
		+ (1-\eta) \theta {\bf L}_{h,{\bf F}} \,,
	\end{align*}
	where $\lambda_{x} = \frac{\Delta t}{\Delta x}, \theta = \frac{\Delta t}{\tau_{max}}, \eta \in (0,1)$ is a constant,
	and ${\bf L}_{h,{\bf F}}, {\bf L}_{h,{\bf S}}$ are given by
	\begin{align*}
		{\bf L}_{h,{\bf F}} = \overline{ {\bf U} }_{j}^{D} - \frac{\lambda_{x}}{(1-\eta)\theta }
		\Big( {\bf F}( {\bf U}_{j+\frac{1}{2}}^{D}) - {\bf F}( {\bf U}_{j-\frac{1}{2}}^{D}) \Big) \,, \quad
		{\bf L}_{h,{\bf S}} = \overline{ {\bf U} }_{j}^{D} + \frac{\lambda_{x}}{\eta\theta }
		\big\langle {\bf S}_{h}^{D}, 1 \big\rangle_{j} \,.
	\end{align*}
	Due to %the convexity of $G$ and
	the exactness of the Gauss-Lobatto quadrature rule, one has
	\begin{align*}
		%\overline{ {\bf U} }_{j}^{C}  = \sum_{\kappa = 1 }^{2}\sum_{\beta=1}^{L}
		%\frac{\hat{\omega}_{\beta}}{2} {\bf U}_{h}^{C}(\hat{x}_j^{\kappa,\beta}) \in G, \quad
		\overline{ {\bf U} }_{j}^{D}  = \sum_{\kappa = 1 }^{2}\sum_{\beta=1}^{L}
		\frac{\hat{\omega}_{\beta}}{2} {\bf U}_{h}^{D}(\hat{x}_j^{\kappa,\beta}).
	\end{align*}
	Let us first consider ${\bf L}_{h,{\bf F}}$ and reformulate it as follows
	\begin{align*}
		{\bf L}_{h,{\bf F}} & = \sum_{\kappa = 1 }^{2}\sum_{\beta=1}^{L} \frac{\hat{\omega}_{\beta}}{2}
		{\bf U}_{h}^{D}(\hat{x}_j^{\kappa,\beta})  - \frac{\lambda_{x}}{(1-\eta)\theta }
		\Big( {\bf F}({\bf U}_{j+\frac{1}{2}}^{D}) - {\bf F}( {\bf U}_{j-\frac{1}{2}}^{D} ) \Big)     \\ \notag
		& = \sum_{\beta=2}^{L}   \frac{\hat{\omega}_{\beta}}{2} {\bf U}_{h}^{D}(\hat{x}_j^{1,\beta})
		+ \dfrac{ \hat{\omega}_{1}}{2} \Big( {\bf U}_{j-\frac{1}{2}}^{D}
		+ \frac{2\lambda_{x}}{(1-\eta)\theta \hat{\omega}_{1}} {\bf F}( {\bf U}_{j-\frac{1}{2}}^{D}) \Big)  \\ \notag
		& + \sum_{\beta=1}^{L-1} \frac{\hat{\omega}_{\beta}}{2} {\bf U}_{h}^{D}(\hat{x}_j^{2,\beta})
		+ \dfrac{ \hat{\omega}_{L}}{2} \Big( {\bf U}_{j+\frac{1}{2}}^{D}
		- \frac{2\lambda_{x}}{(1-\eta)\theta \hat{\omega}_{L}} {\bf F}( {\bf U}_{j+\frac{1}{2}}^{D}) \Big) \\ \notag
		& = \sum_{\beta=2}^{L} \frac{\hat{\omega}_{\beta}}{2} {\bf U}_{h}^{D}(\hat{x}_j^{1,\beta})
		+ \dfrac{ \hat{\omega}_{1}}{2} \mathbb{E}_{j-\frac{1}{2}}^{+}
		+ \sum_{\beta=1}^{L-1} \frac{\hat{\omega}_{\beta}}{2} {\bf U}_{h}^{D}(\hat{x}_j^{2,\beta})
		+ \dfrac{ \hat{\omega}_{L}}{2} \mathbb{E}_{j+\frac{1}{2}}^{-}  \,,
	\end{align*}
	where
	\begin{align*}
		\mathbb{E}_{j-\frac{1}{2}}^{+}  =  {\bf U}_{j-\frac{1}{2}}^{D}
		+ \frac{2\lambda_{x}}{(1-\eta)\theta \hat{\omega}_{1}} {\bf F}({\bf U}_{j-\frac{1}{2}}^{D}) \,, \quad
		\mathbb{E}_{j+\frac{1}{2}}^{-}  =  {\bf U}_{j+\frac{1}{2}}^{D}
		- \frac{2\lambda_{x}}{(1-\eta)\theta \hat{\omega}_{L}} {\bf F}({\bf U}_{j+\frac{1}{2}}^{D}) \,.
	\end{align*}
	Thanks to the Lax-Friedrichs splitting property,  we have $\mathbb{E}_{j-\frac{1}{2}}^{+} \in G$ and
	$\mathbb{E}_{j+\frac{1}{2}}^{-} \in G$, as long as
	\begin{align*}
		\lambda_{x} \max_{ x\in \{ x_{j-\frac{1}{2}}, x_{j+\frac{1}{2}} \} }  a_{x}({\bf U}_{h}^{D}) = \lambda_{x} \tilde{\alpha}_{1,j}^D
		< \frac{(1-\eta)\theta \hat{\omega}_{1}}{2}  \,.
	\end{align*}
	Using the convexity of set $G$, we obtain ${\bf L}_{h,{\bf F}} \in G$. Next, we discuss the term ${\bf L}_{h,{\bf S}}$, and  reformulate the source term $\big\langle {\bf S}_{h}^{D}, 1 \big\rangle_{j}$ as follows
	\begin{align*}
		\big\langle S_{h,2}^{D}, 1 \big\rangle_{j}
		& = \frac{\overline{ (\rho_{h}^{D}) }_{j} }{ \overline{ (\rho_{h}^{s,D}) }_{j} }
		\Bigg( \Big(p_{h}^{s,D}(x_{j+\frac{1}{2}})  -  p_{h}^{s,D}(x_{j-\frac{1}{2}})  \Big)
		- \frac{\Delta x}{2} \sum_{\kappa = 1 }^{2} \sum_{\alpha = 1}^{N} \omega_{\alpha}
		(p_{h}^{s,D})_{x}(x_{j}^{\kappa,\alpha})  \Bigg)  \\ \notag
		& + \frac{\Delta x}{2} \sum_{\kappa = 1 }^{2} \sum_{\alpha = 1}^{N} \omega_{\alpha}
		\frac{\rho_{h}^{D}(x_{j}^{\kappa,\alpha}) }{ \rho_{h}^{s,D} (x_{j}^{\kappa,\alpha}) }
		(p_{h}^{s,D})_{x}(x_{j}^{\kappa,\alpha}) \,.
	\end{align*}
	Notice that
	\begin{align*}
		\frac{\Delta x}{2} \sum_{\kappa = 1 }^{2} \sum_{\alpha = 1}^{N} \omega_{\alpha}
		(p_{h}^{s,D})_{x}(x_{j}^{\kappa,\alpha})  = \int_{I_{j}} (p_{h}^{s,D})_{x} {\rm d}x  \,,
	\end{align*}
	which leads to
	\begin{align*} \notag
		\big\langle S_{h,2}^{D}, 1 \big\rangle_{j}
		& = \frac{\overline{ (\rho_{h}^{D}) }_{j} }{ \overline{ (\rho_{h}^{s,D}) }_{j} }
		\Big( p_{h}^{s,D}(x_{j}^{+})  -  p_{h}^{s,D}(x_{j}^{-}) \Big)
		+ \frac{\Delta x}{2} \sum_{\kappa = 1 }^{2} \sum_{\alpha = 1}^{N} \omega_{\alpha}
		\frac{\rho_{h}^{D}(x_{j}^{\kappa,\alpha}) }{ \rho_{h}^{s,D} (x_{j}^{\kappa,\alpha}) }
		(p_{h}^{s,D})_{x}(x_{j}^{\kappa,\alpha})  \\ \notag
		& = - \frac{\Delta x}{2} \sum_{\kappa = 1 }^{2} \sum_{\alpha = 1}^{N} \omega_{\alpha} \rho_{h}^{D}(x_{j}^{\kappa,\alpha})
		(\hat{\phi}_{h}^{D})_{x} (x_{j}^{\kappa,\alpha}){,}
	\end{align*}
	with
	\begin{equation*}
		- (\hat{\phi}_{h}^{D})_{x} (x_{j}^{\kappa,\alpha}) :=
		\frac{ p_{h}^{s,D}(x_{j}^{+}) - p_{h}^{s,D}(x_{j}^{-}) }{ \overline{ (\rho_{h}^{s,D}) }_{j} \Delta x}
		+ \frac{(p_{h}^{s,D})_{x}(x_{j}^{\kappa,\alpha})}{ \rho_{h}^{s,D} (x_{j}^{\kappa,\alpha}) } \,.
	\end{equation*}
	Similarly, one can derive
	\begin{align*} \notag
		\big\langle S_{h,3}^{D}, 1 \big\rangle_{j}
		& = \frac{\overline{ (m_{h}^{D}) }_{j} }{ \overline{ (\rho_{h}^{s,D}) }_{j} }
		\Big( p_{h}^{s,D}(x_{j}^{+})  -  p_{h}^{s,D}(x_{j}^{-}) \Big)
		+ \frac{\Delta x}{2} \sum_{\kappa = 1 }^{2} \sum_{\alpha = 1}^{N} \omega_{\alpha}
		\frac{m_{h}^{D}(x_{j}^{\kappa,\alpha}) }{ \rho_{h}^{s,D} (x_{j}^{\kappa,\alpha}) }
		(p_{h}^{s,D})_{x}(x_{j}^{\kappa,\alpha})  \\ \notag
		& = - \frac{\Delta x}{2} \sum_{\kappa = 1 }^{2} \sum_{\alpha = 1}^{N} \omega_{\alpha} m_{h}^{D}(x_{j}^{\kappa,\alpha})
		(\hat{\phi}_{h}^{D})_{x} (x_{j}^{\kappa,\alpha})   \,.
	\end{align*}
	Thus ${\bf L}_{h,{\bf S}}$ is reformulated as
	\begin{align*}
		{\bf L}_{h,{\bf S}} &= \overline{ {\bf U} }_{j}^{D} + \frac{\lambda_{x}}{\eta\theta } \big\langle {\bf S}_{h}^{D}, 1 \big\rangle_{j}
		= \overline{ {\bf U} }_{j}^{D} + \frac{\Delta t}{\eta\theta } \sum_{\kappa = 1 }^{2} \sum_{\alpha = 1}^{N} \frac{\omega_{\alpha}}{2}
		\hat{ {\bf S} }_{h}^{D}(x_{j}^{\kappa,\alpha})  \\ \notag
		& = \sum_{\kappa = 1 }^{2} \sum_{\alpha = 1}^{N} \frac{\omega_{\alpha}}{2}
		\Big( {\bf U}_{h}^{D}(x_{j}^{\kappa,\alpha})
		+ \frac{\Delta t}{\eta\theta }  \hat{ {\bf S} }_{h}^{D}(x_{j}^{\kappa,\alpha}) \Big) \,.
	\end{align*}
	where
	\begin{align*}
		\hat{ {\bf S} }_{h}^{D} := \left( 0, -\rho_{h}^{D}(\hat{\phi}_{h}^{D})_{x}, -m_{h}^{D}(\hat{\phi}_{h}^{D})_{x} \right)^\top \,.
	\end{align*}
	Thanks to Lemma \ref{lemmma:pp-source} , we have $ {\bf L}_{h,{\bf S}} \in G$  under the condition
	\begin{equation*}
		\Delta t < \eta\theta \min_{x \in \mathbb{Q}_{j} }
		\Bigg\{ \frac{1}{|(\hat{\phi}_{h}^{D})_{x}|} \sqrt{\frac{2 p_{h}^{D}}{(\gamma-1)\rho_{h}^{D}}} \Bigg\} \,.
	\end{equation*}
	or equivalently
	\begin{align*}
		\lambda_{x} \tilde{\alpha}_{2,j}^D < \frac{\eta \theta \hat{\omega}_{1}}{2} \,.
	\end{align*}
	Combining {those} results, we conclude that if
	\begin{align}\label{eq:1722}
		\lambda_{x} \in  \Big\{ \lambda \in \mathbb{R}^{+}:~
		\lambda \tilde{\alpha}_{1,j}^D <  \frac{(1-\eta)\theta \hat{\omega}_{1}}{2},~
		\lambda \tilde{\alpha}_{2,j}^D <  \frac{\eta \theta \hat{\omega}_{1}}{2} \Big\} \,,
	\end{align}
	then $\overline{ {\bf U} }_{j}^{C} + \Delta t  {\bf L}_{j}^{C} \big( {\bf U}_{h}^{C}, {\bf U}_{h}^{D} \big) \in G$. Since the parameters $\eta$ can be chosen arbitrarily in this proof, we specify $\eta = \tilde{\alpha}_{2,j}^D/\tilde{\alpha}_{j}^D = \tilde{\alpha}_{2,j}^D/ ( \tilde{\alpha}_{1,j}^D + \tilde{\alpha}_{2,j}^D ) $
	such that the condition \eqref{eq:1722} becomes
	\[
	\lambda_{x} (\tilde{\alpha}_{1,j}^D + \tilde{\alpha}_{2,j}^D ) = \lambda_{x} \tilde{\alpha}_{j}^D  < \frac{\theta \hat{\omega}_{1}}{2}  \,.
	\]
	Similar arguments show that $\overline{ {\bf U} }_{j+\frac{1}{2}}^{D} + \Delta t  {\bf L}_{j+\frac{1}{2}}^{D} \big( {\bf U}_{h}^{D}, {\bf U}_{h}^{C} \big) \in G$. The proof is completed.
\end{proof}

\subsection{Positivity-preserving limiting operators ${\Pi}_{h}^C$ and ${\Pi}_{h}^D$} \label{section:limiter}

A simple positivity-preserving limiter can be applied to enforce the condition (\ref{pp-condition}). Because the limiting procedures for ${\bf U}_{h}^C(x)$ and ${\bf U}_{h}^D(x)$ are similar and implemented separately, we only present that for ${\bf U}_{h}^C(x)$. For any ${\bf U}_{h}^C \in \overline{\mathbb{G}}_{h}^{C,k}$ with ${\bf U}_{h}^C \big|_{I_{j}} := {\bf U}_{j}^C(x)$, we follow \cite{Zhang2012robust,Zhang2010PP} and define the positivity-preserving limiting operator ${\Pi}_{h}^C : \overline{\mathbb{G}}_{h}^{C,k} \longrightarrow \mathbb{G}_{h}^{C,k}$ as follows
\[
({\Pi}_{h}^C {\bf U}_{h}^C) \big|_{I_{j}} =
\theta_{j}^{(2)} ( \hat{{\bf U}}_{j}^C(x) - \overline{{\bf U}}_{j}^C ) + \overline{{\bf U}}_{j}^C \,, \quad
\theta_{j}^{(2)} = \min \Bigg\{1, ~ \frac{ p(\overline{{\bf U}}_{j}^C) - \epsilon_{2} }{ p(\overline{{\bf U}}_{j}^C)
	- \min\limits_{x\in \mathbb{S}_j} p(\hat{{\bf U}}_{j}^C(x)) } \Bigg\} \,,
\]
where $\hat{{\bf U}}_{j}^C(x) = (\hat{\rho}_{j}^C(x), m_{j}^C(x), E_{j}^C(x) )^{\top} $, and  $\hat{\rho}_{j}^C(x)$
is a modification of the density $\rho_{j}^C(x)$ given by
\[
\hat{\rho}_{j}^C(x)  = \theta_{j}^{(1)} ( \rho_{j}^C(x) - \overline{\rho}_{j}^C ) + \overline{\rho}_{j}^C \,,\quad
\theta_{j}^{(1)} = \min \Bigg\{1, ~ \frac{\overline{\rho}_{j}^C - \epsilon_{1} }{ \overline{\rho}_{j}^C -
	\min\limits_{x\in \mathbb{S}_j} \rho_{j}^C(x) } \Bigg\} \,.
\]
Here $\epsilon_{1}$ and $\epsilon_{2}$ are two small positive numbers for avoiding the effect of round-off error, and in the computation, one can take
$\epsilon_{1} = \min\{ 10^{-13}, \overline{\rho}_{j}^C \} $, $\epsilon_{2} = \min\{ 10^{-13}, p(\overline{{\bf U}}_{j}^C) \} $. Note that such a local scaling limiter keeps the local conservation and does not destroy the high-order accuracy; 
see \cite{zhang2010maximum,zhang2017positivity} for more details. The positivity-preserving limiting operator ${\Pi}_{h}^D : \overline{\mathbb{G}}_{h}^{D,k} \longrightarrow \mathbb{G}_{h}^{D,k}$ defined on the dual mesh is similar.

Suppose the initial numerical solutions are defined as  ${\bf U}_{h}^{C,0} = {\Pi}_{h}^C \big[ {\bf U}_{h}^{C}(x,0) \big]$,  ${\bf U}_{h}^{D,0} = {\Pi}_{h}^D \big[ {\bf U}_{h}^{D}(x,0) \big]$. For the WB CDG schemes (\ref{eq:CDG-modify-primal}) and (\ref{eq:CDG-modify-dual}) coupled with an third order SSP Runge-Kutta method,
if the positivity-preserving limiter is used at each Runge-Kutta stage, then our fully discrete CDG schemes are positivity-preserving, {namely, ${\bf U}_{h}^{C,n} \in \mathbb{G}_{h}^{C,k}$ and ${\bf U}_{h}^{D,n} \in \mathbb{G}_{h}^{D,k}$}.

\begin{remark}[WB Implementation of Non-oscillatory Limiters]\label{rem:WB-WENO}
When the exact solution contains strong discontinuities, 
the above positivity-preserving limiter may not control the nonphysical numerical  	
oscillations in the CDG solutions, and 
a standard non-oscillatory limiter, such as the TVD/TVB or WENO limiter, is still needed in the ``troubled'' cells. 
We will adopt the WENO limiter \cite{qiu2005runge} in the numerical examples involving discontinuities (Examples 4, 9, and 11 in Section \ref{section:numerical-example}).  
 However, the traditional use of non-oscillatory limiters may 
destroy the WB property of our schemes. 
This issue can be easily addressed by slightly modifying the procedure of identifying the ``troubled'' cells, based on 
 the perturbations of the solutions and cell averages 
%the perturbation function $\tilde{\bf U}_{h}^C(x)$ and 
%cell average $\overline{\tilde{\bf U}}_{j}^C$ as 
\[
\tilde{\bf U}_{h}^C(x) = {\bf U}_{h}^C(x) - {\bf U}_{h}^{s,C}(x) \,,  \qquad
\overline{\tilde{\bf U}}_{j}^C = \overline{\bf U}_{j}^C - \overline{\bf U}_{j}^{s,C} \,.
\]
More specifically, for each $j$ we first use the TVB corrected minmod function (see, e.g., \cite{qiu2005runge})
\begin{equation}\label{minmod-corrected}
	\tilde{\tt m}(a_{1}, a_{2},\cdots,  a_{n})=
	\begin{cases}
		a_{1},    &  {\rm if} ~ |a_{1}| \leq M (\Delta x)^2  \,,  \\
		{\tt m}(a_{1}, a_{2},\cdots,  a_{n}),    &  {\rm otherwise}  \,,
	\end{cases} 
\end{equation}
to check if the cell $I_{j}$ is ``troubled'' based on the cell-averaged values $\overline{\tilde{\bf U}}_{j}^C$, 
$\overline{\tilde{\bf U}}_{j\pm 1}^C$, and the endpoint values $\tilde{\bf U}_{h}^C(x_{j-\frac{1}{2}}^{+})$, $\tilde{\bf U}_{h}^C(x_{j+\frac{1}{2}}^{-})$ on the cell $I_{j}$. Only if cell $I_{j}$ is identified as ``troubled'' cell, we then apply the WENO limiter on ${\bf U}_{h}^C(x)$ as usual before the positivity-preserving limiter. 
%in the local characteristic fields 
%  
The same implementation is also used separately on the dual mesh. 
Note that if the steady state is reached, then $\tilde{\bf U}_{h}^C(x)$ becomes zero so that no cell will be flagged as ``troubled'', and thus the WB property is preserved. 
Numerical results in Example 11 will further confirm that our implementation of the WENO limiter does not affect the WB property. 
\end{remark}

\section{Extension to the two-dimensional case} \label{section:ppwb2d}

This section will extend the positivity-preserving WB CDG schemes to the two-dimensional Euler equations under the gravitational field $\phi(x,y)$
\begin{equation}\label{eq:euler-source-2d}
	{\bf U}_{t} + \nabla\cdot {\bf F}({\bf U}) = {\bf S}({\bf U}, \nabla\phi) \,,
\end{equation}
where ${\bf U} = (\rho, \rho u_{1}, \rho u_{2}, E)^\top$ denotes the conservative variables,
${\bf F}({\bf U}) = \big( {\bf F}_{1} ({\bf U}), {\bf F}_{2} ({\bf U}) \big)$ with
\[
{\bf F}_{1} ({\bf U}) = (\rho u_{1}, \rho u_{1}^2 + p, \rho u_{1} u_{2}, (E+p)u_{1} )^\top \,,
\]
\[
{\bf F}_{2} ({\bf U}) = (\rho u_{2}, \rho u_{1} u_{2}, \rho u_{2}^2 + p, (E+p)u_{2} )^\top \,,
\]
denote the fluxes, and ${\bf S}({\bf U},\nabla\phi) = (0, -\rho\phi_{x}, -\rho\phi_{y}, 
-{\bf m}\cdot\nabla\phi)^\top$ is the source term with $ {\bf m} = (\rho u_{1}, \rho u_{2})$ 
being the momentum vector.

Let $\mathcal{T}_{h}^{C} = \big\{ I_{i,j}\,, \forall i,j \big\}$ and $\mathcal{T}_{h}^{D} = \big\{ I_{i+\frac{1}{2}, j+\frac{1}{2}}\,, \forall i,j \big\}$
respectively denote two overlapping uniform  meshes for the rectangular computational domain
$\Omega = [x_{\min},x_{\max}] \times [y_{\min},y_{\max}]$ with $ I_{i,j} = (x_{i-\frac{1}{2}}, x_{i+\frac{1}{2}}) \times (y_{j-\frac{1}{2}}, y_{j+\frac{1}{2}}) $
and $ I_{i+\frac{1}{2}, j+\frac{1}{2}} = ( x_{i}, x_{i+1} ) \times ( y_{j}, y_{j+1} ) $.
The spatial {stepsizes} are $\Delta x = x_{i+\frac{1}{2}} - x_{i-\frac{1}{2}} $ in the $x$-direction and $\Delta y = y_{j+\frac{1}{2}} - y_{j-\frac{1}{2}} $ in the $y$-direction. We define two discrete function spaces associated with the overlapping meshes
{$\{I_{i,j}\}$ and $\{ I_{i+\frac{1}{2}, j+\frac{1}{2}}\}$}
\[
\mathbb{V}_{h}^{C,k} = \Big\{v:v \big|_{I_{i,j}} \in \mathbb{P}^{k}(I_{i,j}), \forall i,j \Big\} \,, \quad
\mathbb{V}_{h}^{D,k} = \Big\{v:v \big|_{I_{i+\frac{1}{2}, j+\frac{1}{2}}} \in \mathbb{P}^{k}(I_{i+\frac{1}{2}, j+\frac{1}{2}}),
\forall i,j \Big\} \,,
\]
where $\mathbb{P}^{k}(I_{i,j})$ and $\mathbb{P}^{k}(I_{i+\frac{1}{2}, j+\frac{1}{2}})$ denote the space of two-dimensional polynomials in
the cells 
$I_{i,j}$ and $I_{i+\frac{1}{2}, j+\frac{1}{2}}$ with degree of at most $k$, respectively. To solve the system (\ref{eq:euler-source-2d}), the standard CDG method in the semi-discrete form looks for two numerical solutions ${\bf U}_{h}^{C} \in [\mathbb{V}_{h}^{C,k}]^4$ and  ${\bf U}_{h}^{D} \in [\mathbb{V}_{h}^{D,k}]^4$ such that
\begin{align}\label{eq:CDG-primal-2d} \notag
	\int_{ I_{i,j} } \frac{\partial {\bf U}_{h}^{C}}{\partial t} v {\rm d}x {\rm d}y
	& = \frac{1}{\tau_{max}}\int_{ I_{i,j} } ( {\bf U}_{h}^{D} - {\bf U}_{h}^{C}) v {\rm d}x {\rm d}y
	+ \int_{I_{i,j}}  {\bf F}({\bf U}_{h}^{D}) \cdot \nabla v  {\rm d}x {\rm d}y  \\ \notag
	& - \int_{y_{j-\frac{1}{2}}}^{y_{j+\frac{1}{2}}}
	\Big({\bf F}_{1}({\bf U}_{h}^{D}(x_{i+\frac{1}{2}},y))v(x_{i+\frac{1}{2}}^{-},y)
	-      {\bf F}_{1}({\bf U}_{h}^{D}(x_{i-\frac{1}{2}},y))v(x_{i-\frac{1}{2}}^{+},y) \Big)  {\rm d}y \\ \notag
	& - \int_{x_{i-\frac{1}{2}}}^{x_{i+\frac{1}{2}}}
	\Big({\bf F}_{2}({\bf U}_{h}^{D}(x, y_{j+\frac{1}{2}})) v(x, y_{j+\frac{1}{2}}^{-})
	- {\bf F}_{2}({\bf U}_{h}^{D}(x, y_{j-\frac{1}{2}})) v(x, y_{j-\frac{1}{2}}^{+}) \Big) {\rm d}x \\
	& +  \int_{ I_{i,j} } {\bf S}({\bf U}_{h}^{D}, \nabla\phi_{h}^{D}) v {\rm d}x{\rm d}y
	\qquad \forall v \in \mathbb{V}_{h}^{C,k}  \,,
\end{align}
\begin{align}\label{eq:CDG-dual-2d} \notag
	\int_{ I_{i+\frac{1}{2},j+\frac{1}{2}} } \frac{\partial {\bf U}_{h}^{D}}{\partial t} w {\rm d}x {\rm d}y
	& = \frac{1}{\tau_{max}}\int_{ I_{i+\frac{1}{2},j+\frac{1}{2}} } ( {\bf U}_{h}^{C} - {\bf U}_{h}^{D}) w {\rm d}x {\rm d}y
	+ \int_{I_{i+\frac{1}{2},j+\frac{1}{2}}}  {\bf F}({\bf U}_{h}^{C}) \cdot \nabla w  {\rm d}x {\rm d}y  \\ \notag
	& - \int_{y_{j}}^{y_{j+1}}
	\Big({\bf F}_{1}({\bf U}_{h}^{C}(x_{i+1},y))w(x_{i+1}^{-},y)
	-    {\bf F}_{1}({\bf U}_{h}^{C}(x_{i},y))w(x_{i}^{+},y) \Big)  {\rm d}y \\ \notag
	& - \int_{x_{i}}^{x_{i+1}}
	\Big({\bf F}_{2}({\bf U}_{h}^{C}(x, y_{j+1})) w(x, y_{j+1}^{-})
	- {\bf F}_{2}({\bf U}_{h}^{C}(x, y_{j} )) w(x, y_{j}^{+}) \Big) {\rm d}x \\
	& +  \int_{ I_{i+\frac{1}{2},j+\frac{1}{2}} } {\bf S}({\bf U}_{h}^{C}, \nabla\phi_{h}^{C}) w {\rm d}x{\rm d}y
	\qquad \forall w \in \mathbb{V}_{h}^{D,k} \,,
\end{align}
where $\tau_{max} = \tau_{max}(t)$ is the maximal time step allowed by the CFL restriction at time $t$. As the one-dimensional case,
%this standard CDG method does not preserve the hydrostatic stationary solution (\ref{eq:steady-state-2d}).
the two-dimensional standard CDG method \eqref{eq:CDG-primal-2d}--\eqref{eq:CDG-dual-2d} is generally not  WB for the stationary hydrostatic  solutions.

\subsection{Novel projection of the stationary hydrostatic solutions}
Assume that the target equilibrium state of the system \eqref{eq:euler-source-2d} is known and denoted by
$$ \big\{ \rho^s(x,y),  u_1^s(x,y),  u_2^s(x,y), p^s(x,y) \big\}, $$
which satisfies
\begin{equation}\label{eq:steady-state-2d}
	 u_1^s(x,y) = 0 \,, \quad  u_2^s(x,y) = 0 \,,  \quad \nabla p^s = -\rho^s \nabla \phi \,.
\end{equation}

Let $ {\bf U}^s(x,y) = \big( \rho^s(x,y), 0,  0, p^s(x,y)/(\gamma-1) \big)^\top $,
{and define}
\begin{align*}
	& I_{i,j}^0 = (x_{i-\frac{1}{2}}, x_{i})\times(y_{j-\frac{1}{2}}, y_{j}),  \quad
	I_{i,j}^1 = (x_{i}, x_{i+\frac{1}{2}})\times(y_{j-\frac{1}{2}}, y_{j}),\\
	&I_{i,j}^2 = (x_{i-\frac{1}{2}}, x_{i})\times(y_{j}, y_{j+\frac{1}{2}}), \quad
	I_{i,j}^3 = (x_{i}, x_{i+\frac{1}{2}})\times(y_{j}, y_{j+\frac{1}{2}}).
\end{align*}
Following the ideas in the one-dimensional case, we first introduce the novel projection of the stationary solution $ {\bf U}^s(x,y)$ on the primal mesh. Define {the} operator  $\mathcal{P}_{h}^{C} : L^2(\Omega) \longrightarrow \mathbb{V}_{h}^{C,k}$, such that for any function $f \in L^2(\Omega)$,
\begin{equation} \label{eq:projection-special-2d}
	\begin{aligned}
	\int_{I_{i,j}^m}\mathcal{P}_{h}^{C}(f) {\rm d}x {\rm d}y  &= \int_{I_{i,j}^m} f {\rm d}x {\rm d}y \,, \quad m \in \{ 0,1,2 \} \,, \\
	\int_{I_{i,j}} \mathcal{P}_{h}^{C} (f) v {\rm d}x {\rm d}y  &= \int_{I_{i,j}} f v {\rm d}x {\rm d}y \,,  \quad
	\forall v \in \mathrm{span} \{ \Phi_{0}\,, \Phi_{4}\,, \cdots \,, \Phi_{K} \},
	%\mathbb{P}^{k}(I_{i,j}) \backslash \mathrm{span} \{ \Phi_{1}\,, \Phi_{2}\,, \Phi_{3} \} \,.
\end{aligned}
\end{equation}
%Here $\mathbb{P}^{k}(I_{i,j}) \backslash \mathrm{span} \{ \Phi_{1}\,, \Phi_{2}\,, \Phi_{3} \} := ,
where
$K = k(k+3)/2$, and $\{\Phi_{l} \}_{l=0}^{K}$ is an orthogonal basis of $\mathbb{P}^{k}(I_{i,j})$ and taken as the scaled Legendre polynomials
\begin{align*}
	&\Phi_{0}{(\xi,\eta)} = 1 \,, 
\  \Phi_{1}{(\xi,\eta)} = \xi \,,  \ \Phi_{2}{(\xi,\eta)} =  \eta \,,   \ \Phi_{3}{(\xi,\eta)} =  \xi\eta \,,
\\  &
	\ \Phi_{4}{(\xi,\eta)} = \xi^2 - \frac{1}{3} \,,   \
 \Phi_{5}{(\xi,\eta)} = \eta^2 - \frac{1}{3} \,,  \  \cdots{,}
\end{align*}
with $\xi = 2(x-x_{i})/\Delta x$ and $\eta = 2(y-y_{j})/\Delta y$. 
It follows from \eqref{eq:projection-special-2d} that the operator $\mathcal{P}_{h}^{C} $ satisfies
\begin{align} \label{eq:projection-primal-2d}
	\int_{I_{i,j}^m}\mathcal{P}_{h}^{C}(f) {\rm d}x {\rm d}y = \int_{I_{i,j}^m} f {\rm d}x {\rm d}y{,}
	\quad \forall m \in \{ 0,1,2,3 \} \,,  \quad \forall  i,j \,,
\end{align}
for any function $f \in L^2(\Omega)$.
As the one-dimensional case, the operator $\mathcal{P}_{h}^{C} $ defined by \eqref{eq:projection-special-2d} can be explicitly expressed.
In fact, the piecewise polynomial $\mathcal{P}_{h}^{C} (f)$ on each cell $I_{i,j}$ takes the form of
\begin{align*}
	\mathcal{P}_{h}^{C} (f) = \sum_{l=0}^{K} N_{l}(f) \Phi_{l}{,}
\end{align*}
with the polynomial coefficients $\{ N_{l}(f) \}_{l=0}^K$ given by
\begin{align*}
	&N_{l}(f) = \frac{ \int_{I_{i,j}} f \Phi_{l} {\rm d}x {\rm d}y }{\int_{I_{i,j}} \Phi_{l}^{2} {\rm d}x {\rm d}y} \,,
	\quad  l \neq 1,2,3 \,, \\
	&N_{1}(f) = \frac{-4(b_{0} + b_{2})}{\Delta x \Delta y} \,, \quad
	N_{2}(f) = \frac{-4(b_{0} + b_{1})}{\Delta x \Delta y} \,, \quad
	N_{3}(f) = \frac{-8(b_{1} + b_{2})}{\Delta x \Delta y} \,,
\end{align*}
where {the} parameter $b_{m} = \int_{ I_{i,j}^m } \big( f - \sum\limits_{l\notin \{1,2,3\} } N_{l}(f) \Phi_{l} \big) {\rm d}x {\rm d}y $,
$m \in \{ 0,1,2 \} $.

Similarly, we can define the projection $\mathcal{P}_{h}^{D} : L^2(\Omega) \longrightarrow \mathbb{V}_{h}^{D,k}$
on the dual mesh  such that
\begin{align} \label{eq:projection-dual-2d}
	\int_{I_{i+\frac{1}{2},j+\frac{1}{2}}^m}\mathcal{P}_{h}^{D}(f) {\rm d}x {\rm d}y =
	\int_{I_{i+\frac{1}{2},j+\frac{1}{2}}^m} f {\rm d}x {\rm d}y{,}
	\quad \forall m \in \{ 0,1,2,3 \} \,,  \quad \forall i,j \,,
\end{align}
where $I_{i+\frac{1}{2},j+\frac{1}{2}}^m$ is a shift of $I_{i,j}^m$ with $\frac{\Delta x}{2}$ in the  $x$-direction and
$\frac{\Delta y}{2}$ in the $y$-direction. Combining \eqref{eq:projection-primal-2d} {with} \eqref{eq:projection-dual-2d} leads to the following crucial {identities}
\begin{align}\label{eq:ID001}
\int_{I_{i,j}}  \mathcal{P}_{h}^{C} (f) {\rm d}x{\rm d}y =  \int_{I_{i,j}}  \mathcal{P}_{h}^{D} (f) {\rm d}x{\rm d}y  \,, \quad
\int_{I_{i+\frac{1}{2},j+\frac{1}{2}}}  \mathcal{P}_{h}^{D} (f) {\rm d}x{\rm d}y =
\int_{I_{i+\frac{1}{2},j+\frac{1}{2}}}  \mathcal{P}_{h}^{C} (f) {\rm d}x{\rm d}y \,.
\end{align}
{If let} ${\bf U}_{h}^{s,C}$ and ${\bf U}_{h}^{s,D}$ denote the above novel {projections  of} the steady state solutions ${\bf U}^s(x,y)$ onto the space  $[\mathbb{V}_{h}^{C,k}]^4$ and $[\mathbb{V}_{h}^{D,k}]^4$, respectively{, then} the {identities in} \eqref{eq:ID001} {imply}
\begin{align}\label{eq:average-identity-2d}
\int_{I_{i,j}}  {\bf U}_{h}^{s,C} {\rm d}x{\rm d}y =  \int_{I_{i,j}}  {\bf U}_{h}^{s,D} {\rm d}x{\rm d}y  \,, \quad
\int_{I_{i+\frac{1}{2},j+\frac{1}{2}}}  {\bf U}_{h}^{s,D}  {\rm d}x{\rm d}y =
\int_{I_{i+\frac{1}{2},j+\frac{1}{2}}}  {\bf U}_{h}^{s,C}  {\rm d}x{\rm d}y \,,  \quad \forall i,j \,.
\end{align}

\subsection{WB CDG schemes}

The design of our two-dimensional WB CDG method on the rectangular mesh is similar to the procedure described in the one-dimensional case.

The numerical dissipation term in the semi-discrete CDG method (\ref{eq:CDG-primal-2d})  is modified as
\begin{equation*}%\label{eq:wb-dissipation-2d}
	\tilde{\rm d}_{ij}^{C} ({\bf U}_{h}^{C}, {\bf U}_{h}^{D}, v) =
	\frac{1}{\tau_{max}} \int_{I_{i,j}} ({\bf U}_{h}^{D}   - {\bf U}_{h}^{C}   ) v {\rm d}x {\rm d}y +
	\frac{1}{\tau_{max}} \int_{I_{i,j}} ({\bf U}_{h}^{s,C} - {\bf U}_{h}^{s,D} ) v {\rm d}x {\rm d}y \,,
\end{equation*}
where $\tilde{\rm d}_{ij}^{C}$ satisfies the WB property $\tilde{\rm d}_{ij}^{C} ({\bf U}_{h}^{s,C}, {\bf U}_{h}^{s,D}, v) = {\bf 0}$. Such modification does not affect the spatial accuracy.

In order to discretize the flux and source term integrals, we need to introduce the two-dimensional numerical quadrature.
The notations for the quadrature points in the $x$-direction is the same as the one-dimensional case. For the $y$-direction,
let $ \mathbb{Q}_{j}^{1,y} = \big\{ y_{j}^{1,\mu} \big\}_{\mu = 1}^{N} $
and $ \mathbb{Q}_{j}^{2,y} = \big\{ y_{j}^{2,\mu} \big\}_{\mu = 1}^{N} $
denote the $N$-point Gauss quadrature nodes transformed into the interval $\big[y_{j-\frac{1}{2}}, y_{j}\big]$
and  $\big[y_{j}, y_{j+\frac{1}{2}}\big]$, respectively, and $\big\{\omega_{\mu}\big\}_{\mu = 1}^{N}$
are the associated weights satisfying $\sum_{\mu = 1}^{N} \omega_{\mu} = 1$, with $N \geq k+1$ for the
CDG accuracy requirement.
Then the flux integrals can be approximated by the numerical quadrature
\[
\int_{I_{i,j}}  {\bf F}({\bf U}_{h}^{D}) \nabla v {\rm d}x{\rm d}y \approx
\frac{\Delta x \Delta y}{4} \sum_{\kappa,\alpha}  \sum_{\sigma, \mu} \omega_{\alpha} \omega_{\mu}
{\bf F}({\bf U}_{h}^{D} (x_{i}^{\kappa,\alpha}, y_{j}^{\sigma, \mu})) \nabla v(x_{i}^{\kappa,\alpha}, y_{j}^{\sigma, \mu}) \,,
\]
\[
\int_{\partial I_{i,j}}  \big( {\bf F}({\bf U}_{h}^{D}) \cdot {\bf n} \big)v {\rm d}s \approx
\Delta y  \delta^x_{ij} { \big( {\bf F}_{1}({\bf U}_{h}^{D}) v \big)} +
\Delta x \delta^y_{ij} { \big( {\bf F}_{2}({\bf U}_{h}^{D}) v \big)}   \,,
\]
where $\sigma,\kappa \in \{1,2\}$ and $\mu,\alpha \in \{1,\cdots,N\}$, $ {\bf n} $ is the outward unit normal vector of the cell $I_{i,j}$, and the operators
$\delta^x_{ij}, \delta^y_{ij}$ are defined by
\begin{align}\notag
  \delta^x_{ij} { ( f)}
	&:= \sum_{\kappa,\alpha} \frac{\omega_{\alpha}}{2}
	\Big(  f (x_{i+\frac{1}{2}}, y_{j}^{\kappa,\alpha})
	- f (x_{i-\frac{1}{2}}, y_{j}^{\kappa,\alpha})  \Big)   \,,  \\ \notag
	\delta^y_{ij} { (f)}
	&:= \sum_{\kappa,\alpha} \frac{\omega_{\alpha}}{2}
	\Big(  f (x_{i}^{\kappa,\alpha}, y_{j+\frac{1}{2}})
	- f (x_{i}^{\kappa,\alpha}, y_{j-\frac{1}{2}})  \Big)  \,.
\end{align}
The last step for designing our WB CDG spatial discretization is to suitably discretize the source term integral.  The source term integrals in the momentum equations are reformulated into
\begin{align*}
	\int_{I_{i,j}} (S_{2},S_{3})^\top v {\rm d}x{\rm d}y
	& = - \int_{I_{i,j}} \rho \nabla \phi v {\rm d}x{\rm d}y
	= \int_{I_{i,j}} \frac{\rho}{\rho^s} \nabla p^s v {\rm d}x{\rm d}y  \\  \notag
	& =   \int_{I_{i,j}} \Bigg( \frac{\rho}{\rho^s} -  \frac{ \overline{\rho}_{ij} }{ \overline{\rho}^s_{ij} }
	+   \frac{ \overline{\rho}_{ij} }{ \overline{\rho}^s_{ij} }  \Bigg) \nabla p^s v {\rm d}x{\rm d}y    \,,
\end{align*}
where $\overline{\rho}_{ij} = \frac{1}{\Delta x \Delta y}\int_{I_{ij}} \rho {\rm d}x{\rm d}y $ is the cell average. Following the one-dimensional design, we observe that
\begin{align*} \notag
	\int_{I_{i,j}} (S_{2},S_{3})^\top v {\rm d}x{\rm d}y   &\approx  \int_{ I_{i,j} }
	\Bigg( \frac{ \rho_{h}^{D} }{ \rho_{h}^{s,D} }
	- \frac{\overline{ (\rho_{h}^{D}) }_{ij} }{ \overline{ (\rho_{h}^{s,D}) }_{ij}}  \Bigg)
	\nabla p_{h}^{s,D} v {\rm d}x{\rm d}y
	+ \frac{\overline{ (\rho_{h}^{D}) }_{ij} }{ \overline{ (\rho_{h}^{s,D}) }_{ij} }
	\int_{I_{i,j}} \nabla p_{h}^{s,D} v {\rm d}x{\rm d}y    \\  \notag
	& \approx  \int_{ I_{i,j} }
	\Bigg( \frac{ \rho_{h}^{D} }{ \rho_{h}^{s,D} }
	- \frac{\overline{ (\rho_{h}^{D}) }_{ij} }{ \overline{ (\rho_{h}^{s,D}) }_{ij}}  \Bigg)
	\nabla p_{h}^{s,D} v {\rm d}x{\rm d}y    -
	\frac{\overline{ (\rho_{h}^{D}) }_{ij} }{ \overline{ (\rho_{h}^{s,D}) }_{ij} }
	\int_{ I_{i,j} } p_{h}^{s,D} \nabla v {\rm d}x{\rm d}y   \\
	&+\frac{\overline{(\rho_{h}^{D})}_{ij} }{ \overline{(\rho_{h}^{s,D})}_{ij}}
	\int_{\partial I_{i,j}} p_{h}^{s,D}v {\bf n} {\rm d}s \,.
\end{align*}
Therefore, the source term integrals $\int_{I_{i,j}} (S_{2},S_{3})^\top v {\rm d}x{\rm d}y$  can be approximated by
\begin{align}\notag
	& \big\langle \big(S_{2,h}^{D},S_{3,h}^{D}\big)^\top, v \big\rangle_{ij}
	:= \frac{\Delta x \Delta y}{4}
	\sum_{\kappa,\alpha}  \sum_{\sigma, \mu}
	\omega_{\alpha} \omega_{\mu}
	\Bigg( \frac{ \rho_{h}^{D}(x_{i}^{\kappa,\alpha}, y_{j}^{\sigma, \mu}) }
	{ \rho_{h}^{s,D}(x_{i}^{\kappa,\alpha}, y_{j}^{\sigma, \mu}) }
	-  \frac{\overline{(\rho_{h}^{D})}_{ij} }{ \overline{(\rho_{h}^{s,D})}_{ij}}  \Bigg)
	(\nabla p_{h}^{s,D} v)(x_{i}^{\kappa,\alpha}, y_{j}^{\sigma, \mu})  \\ \notag
	& - \frac{\overline{(\rho_{h}^{D})}_{ij} }{ \overline{(\rho_{h}^{s,D})}_{ij}}
	\frac{\Delta x \Delta y}{4}  \sum_{\kappa,\alpha}  \sum_{\sigma, \mu} \omega_{\alpha} \omega_{\mu}
	(p_{h}^{s,D} \nabla v)(x_{i}^{\kappa,\alpha}, y_{j}^{\sigma, \mu})
	+ \frac{\overline{(\rho_{h}^{D})}_{ij} }{ \overline{(\rho_{h}^{s,D})}_{ij}}
	\Big(\Delta y \delta^{x}_{ij} (p_{h}^{s,D}v), \Delta x \delta^{y}_{ij} (p_{h}^{s,D}v)  \Big)^\top \,.
\end{align}
Similarly, the source term integral $\int_{I_{i,j}} S_4 v {\rm d}x{\rm d}y$ in the energy equation can be approximated by
\begin{align}\notag
	& \big\langle S_{4,h}^{D}, v \big\rangle_{ij}
	:= \frac{\Delta x \Delta y}{4} \sum_{\kappa,\alpha}  \sum_{\sigma, \mu} \omega_{\alpha} \omega_{\mu}
	\Bigg( \frac{ {\bf m}_{h}^{D}(x_{i}^{\kappa,\alpha}, y_{j}^{\sigma, \mu}) }
	{ \rho_{h}^{s,D}(x_{i}^{\kappa,\alpha}, y_{j}^{\sigma, \mu}) }
	-  \frac{\overline{( {\bf m}_{h}^{D})}_{ij} }{ \overline{(\rho_{h}^{s,D})}_{ij}}  \Bigg)
	(\nabla p_{h}^{s,D} v)(x_{i}^{\kappa,\alpha}, y_{j}^{\sigma, \mu})  \\ \notag
	& - \frac{\overline{( {\bf m}_{h}^{D})}_{ij} }{ \overline{(\rho_{h}^{s,D})}_{ij}}
	\frac{\Delta x \Delta y}{4}  \sum_{\kappa,\alpha}  \sum_{\sigma, \mu} \omega_{\alpha} \omega_{\mu}
	(p_{h}^{s,D} \nabla v)(x_{i}^{\kappa,\alpha}, y_{j}^{\sigma, \mu})
	+ \frac{\overline{( {\bf m}_{h}^{D})}_{ij} }{ \overline{(\rho_{h}^{s,D})}_{ij}}
	\Big(\Delta y \delta^{x}_{ij} (p_{h}^{s,D}v),  \Delta x \delta^{y}_{ij} (p_{h}^{s,D}v) \Big)^\top \,.
\end{align}
Combining those leads to the following WB CDG discretization for the two-dimensional Euler equations with gravity
on the primal mesh
\begin{align}\nonumber
	&\int_{ I_{i,j} } \frac{\partial {\bf U}_{h}^{C}}{\partial t} v {\rm d}x {\rm d}y
	= \frac{1}{\tau_{max}} \int_{ I_{i,j} } ({\bf U}_{h}^{D}   - {\bf U}_{h}^{C}   )  v {\rm d}x {\rm d}y
	+ \frac{1}{\tau_{max}} \int_{ I_{i,j} } ({\bf U}_{h}^{s,C} - {\bf U}_{h}^{s,D} )  v {\rm d}x {\rm d}y
	+ \big\langle {\bf S}_{h}^{D}, v \big\rangle_{ij} \\  
	& + \frac{\Delta x \Delta y}{4} \sum_{\kappa,\alpha}  \sum_{\sigma, \mu} \omega_{\alpha} \omega_{\mu}
	{\bf F}({\bf U}_{h}^{D} (x_{i}^{\kappa,\alpha}, y_{j}^{\sigma, \mu}))
	\nabla v(x_{i}^{\kappa,\alpha}, y_{j}^{\sigma, \mu})
	 - \Delta y \delta^{x}_{ij} {\big( {\bf F}_{1}({\bf U}_{h}^{D}) v \big)}
	- \Delta x \delta^{y}_{ij} \big( {\bf F}_{2}({\bf U}_{h}^{D})v \big)  \,,\label{eq:CDG-modify-primal-2d}
\end{align}
where $\big\langle {\bf S}_{h}^{D}, v \big\rangle_{ij} = \big (0, \big\langle S_{h,2}^{D}, v \big\rangle_{ij},
\big\langle S_{h,3}^{D}, v \big\rangle_{ij}, \big\langle S_{h,4}^{D}, v \big\rangle_{ij} \big)^\top $.
The WB CDG  spatial discretization on the dual mesh is very similar. Denote ${\bf Q} = ({\bf Q}_{1}, {\bf Q}_{2})$, with
${\bf Q}_{1} = (0, \rho, 0, \rho u_{1} )^\top$, ${\bf Q}_{2} = (0, 0, \rho, \rho u_{2})^\top$, one has
${\bf S}({\bf U}) = - {\bf Q} \cdot \nabla \phi$,  and the source term integrals on the dual mesh  are approximated by
\begin{align}\notag
	\big\langle {\bf S}_{h}^{C}, w \big\rangle_{i+\frac{1}{2}, j+\frac{1}{2}}
	& = \frac{\Delta x \Delta y}{4} \sum_{\kappa,\alpha}  \sum_{\sigma, \mu} \omega_{\alpha} \omega_{\mu}
	\Bigg( \frac{ {\bf Q}_{h}^{C}(x_{i+\frac{1}{2}}^{\kappa,\alpha}, y_{j+\frac{1}{2}}^{\sigma, \mu}) }
	{ \rho_{h}^{s,C}(x_{i+\frac{1}{2}}^{\kappa,\alpha}, y_{j+\frac{1}{2}}^{\sigma, \mu}) }
	-  \frac{\overline{( {\bf Q}_{h}^{C})}_{i+\frac{1}{2},j+\frac{1}{2}} }{ \overline{(\rho_{h}^{s,C})}_{i+\frac{1}{2},j+\frac{1}{2}}}  \Bigg)
	(\nabla p_{h}^{s,C} w)(x_{i+\frac{1}{2}}^{\kappa,\alpha}, y_{j+\frac{1}{2}}^{\sigma, \mu})  \\ \notag
	& - \frac{\overline{( {\bf Q}_{h}^{C})}_{i+\frac{1}{2},j+\frac{1}{2}} }{ \overline{(\rho_{h}^{s,C})}_{i+\frac{1}{2},j+\frac{1}{2}}}
	\frac{\Delta x \Delta y}{4}  \sum_{\kappa,\alpha}  \sum_{\sigma, \mu} \omega_{\alpha} \omega_{\mu}
	(p_{h}^{s,C} \nabla w)(x_{i+\frac{1}{2}}^{\kappa,\alpha}, y_{j+\frac{1}{2}}^{\sigma, \mu})   \\ \notag
	&+ \frac{\overline{( {\bf Q}_{h}^{C})}_{i+\frac{1}{2},j+\frac{1}{2}} }{ \overline{(\rho_{h}^{s,C})}_{i+\frac{1}{2},j+\frac{1}{2}}}
	\Big(\Delta y \delta^{x}_{i+\frac{1}{2},j+\frac{1}{2}} (p_{h}^{s,C} w),  \Delta x \delta^{y}_{i+\frac{1}{2},j+\frac{1}{2}} (p_{h}^{s,C} w)  \Big)^\top \,,
\end{align}
where  $\big\{ y_{j+\frac{1}{2}}^{1,\mu} \big\}_{\mu = 1}^{N} = \mathbb{Q}_{j}^{2,y}$,  
      $\big\{ y_{j+\frac{1}{2}}^{2,\mu} \big\}_{\mu = 1}^{N} = \mathbb{Q}_{j+1}^{1,y}$,
and  operators $\delta^x_{i+\frac{1}{2},j+\frac{1}{2}}, \delta^y_{i+\frac{1}{2},j+\frac{1}{2}}$ 
are defined {by}
\begin{align}\notag
	\delta^x_{i+\frac{1}{2},j+\frac{1}{2}} { ( f)}
	&:= \sum_{\kappa,\alpha} \frac{\omega_{\alpha}}{2}
	\Big(  f (x_{i+1}, y_{j+\frac{1}{2}}^{\kappa,\alpha})
	- f (x_{i}, y_{j+\frac{1}{2}}^{\kappa,\alpha})  \Big)   \,,  \\ \notag
	\delta^y_{i+\frac{1}{2},j+\frac{1}{2}} { (f)}
	&:= \sum_{\kappa,\alpha} \frac{\omega_{\alpha}}{2}
	\Big(  f (x_{i+\frac{1}{2}}^{\kappa,\alpha}, y_{j+1})
	- f (x_{i+\frac{1}{2}}^{\kappa,\alpha}, y_{j})  \Big)  \,.
\end{align}
Then the WB CDG discretization on the dual mesh is given by
\begin{align}\notag
	\int_{ I_{i+\frac{1}{2},j+\frac{1}{2}} } &\frac{\partial {\bf U}_{h}^{D}}{\partial t} w {\rm d}x {\rm d}y
	= \frac{1}{\tau_{max}} \int_{ I_{i+\frac{1}{2},j+\frac{1}{2}} } ({\bf U}_{h}^{C}   - {\bf U}_{h}^{D})   w {\rm d}x {\rm d}y
	+ \frac{1}{\tau_{max}} \int_{ I_{i+\frac{1}{2},j+\frac{1}{2}} } ({\bf U}_{h}^{s,D} - {\bf U}_{h}^{s,C}) w {\rm d}x {\rm d}y  
\\ \notag
	& + \frac{\Delta x \Delta y}{4} \sum_{\kappa,\alpha}  \sum_{\sigma, \mu} \omega_{\alpha} \omega_{\mu}
	{\bf F}({\bf U}_{h}^{C} (x_{i+\frac{1}{2}}^{\kappa,\alpha}, y_{j+\frac{1}{2}}^{\sigma, \mu}))
	\nabla w(x_{i+\frac{1}{2}}^{\kappa,\alpha}, y_{j+\frac{1}{2}}^{\sigma, \mu}) 
\\  
	& - \Delta y \delta^{x}_{i+\frac{1}{2},j+\frac{1}{2}} \big( {\bf F}_{1}({\bf U}_{h}^{C}) w \big)
	- \Delta x \delta^{y}_{i+\frac{1}{2},j+\frac{1}{2}} \big( {\bf F}_{2}({\bf U}_{h}^{C}) w \big)
	+ \big\langle {\bf S}_{h}^{C}, w \big\rangle_{i+\frac{1}{2},j+\frac{1}{2}} \,.\label{eq:CDG-modify-dual-2d}
\end{align}

\begin{theorem}
	For the two-dimensional Euler equations (\ref{eq:euler-source-2d}) with gravity, our semi-discrete CDG  schemes
	(\ref{eq:CDG-modify-primal-2d})--(\ref{eq:CDG-modify-dual-2d}) are WB for the stationary  hydrostatic solution
	$(\ref{eq:steady-state-2d})$.
\end{theorem}
\begin{proof}
	The proof is similar to that of Theorem \ref{theorem:wb-property} and thus is omitted here.
\end{proof}

\begin{remark}[WB Implementation of Boundary Conditions]
A suitable implementation of boundary conditions is also essential for preserving the WB property. 
	Take the solid wall boundary as example, which may appear in the bottom of atmosphere as
	in weather modeling (see, e.g., the rising thermal bubble problem in Section \ref{sec:RTB}). 
	Suppose the spatial domain $\Omega = [x_{\min},x_{\max}]\times [y_{\min}, y_{\max}]$ is  divided into $N_x \times N_y$ uniform cells, with  
	\[
	x_{\min} =  x_{\frac{1}{2}} < x_{\frac{3}{2}} < \cdots < x_{N_x+\frac{1}{2}} = x_{\max} \,, \quad
	y_{\min} =  y_{\frac{1}{2}} < y_{\frac{3}{2}} < \cdots < y_{N_y+\frac{1}{2}} = y_{\max} \,. 
	\]
	We implement the reflective boundary conditions on the solid wall as follows:
	\begin{itemize}
		\item For the right boundary, we set 
		\[
		\rho^{C}_{N_x+1,j}(x,y) = \rho^{C}_{N_x,j}(2x_{\max} - x, y)\,, \quad  
		\rho^{D}_{N_x+\frac{1}{2},j+\frac{1}{2}}(x,y) = 
		\rho^{D}_{N_x+\frac{1}{2},j+\frac{1}{2}}(2x_{\max} - x, y) \,,
		\]
		\[
		m^{C}_{1,N_x+1,j}(x,y) = - m^{C}_{1,N_x,j}(2x_{\max} - x, y)  \,, \quad  
		m^{D}_{1,N_x+\frac{1}{2},j+\frac{1}{2}}(x,y) = - m^{D}_{1,N_x+\frac{1}{2},j+\frac{1}{2}}(2x_{\max} - x, y) \,.
		\]
		The boundary conditions for $m_{2,h}^{C}(x,y), E_{h}^{C}(x,y)$ and 
		$m_{2,h}^{D}(x,y), E_{h}^{D}(x,y)$ 
		are same as the density.  The left boundary condition is similar to the right. 
		\item Analogously, for the top boundary, we set  
		\[
		\rho^{C}_{i,N_y+1}(x,y) = \rho^{C}_{i,N_y}(x,2y_{\max} - y)\,, \quad  
		\rho^{D}_{i+\frac{1}{2},N_y+\frac{1}{2}}(x,y) = 
		\rho^{D}_{i+\frac{1}{2},N_y+\frac{1}{2}}(x,2y_{\max} - y) \,,
		\]
		\[
		m^{C}_{2,i,N_y+1}(x,y) = - m^{C}_{2,i,N_y}(x,2y_{\max} - y)\,, \quad  
		m^{D}_{2,i+\frac{1}{2},N_y+\frac{1}{2}}(x,y) = -
		m^{D}_{2,i+\frac{1}{2},N_y+\frac{1}{2}}(x,2y_{\max} - y) \,.
		\]
		The boundary conditions for $m_{1,h}^{C}(x,y), E_{h}^{C}(x,y)$ and 
		$m_{1,h}^{D}(x,y), E_{h}^{D}(x,y)$ are same as the density.  The bottom boundary condition is similar to the top. 
	\end{itemize}
		It is worth noting that, in order to preserve the WB property, we should also apply the same reflective boundary conditions to 
the projected hydrostatic  solutions 
$\rho_{h}^{s,C}, p_{h}^{s,C}$ and $\rho_{h}^{s,D}, p_{h}^{s,D}$ for consistency. The implementations for other boundary conditions are similar and omitted here. 
\end{remark}

\subsection{Positivity-preserving WB CDG schemes}

%In general, the high-order WB CDG schemes (\ref{eq:CDG-modify-primal-2d})--(\ref{eq:CDG-modify-dual-2d}) is not always positivity-preserving. As the one-dimensional case, we can prove that our WB schemes (\ref{eq:CDG-modify-primal-2d})--(\ref{eq:CDG-modify-dual-2d}) satisfy a weak positivity-preserving property, which means a simple limiter are used to enforce the positivity-preserving property  without losing high-order accuracy.

\subsubsection{Properties of admissible states}

The set of admissible states of the two-dimensional Euler equations \eqref{eq:euler-source-2d} is defined by
\begin{equation*}
	G = \left\{ {\bf U} = (\rho, {\bf m}, E)^{\top}:~ \rho > 0, ~ p({\bf U}) = (\gamma -1) \left(
	E - \frac{|{\bf m}|^2}{2\rho} \right) > 0 \right\}  \,,
\end{equation*}
which is a convex set \cite{Zhang2010PP}.

\begin{lemma}\label{lemmma:pp-source-2d}
	For any $ {\bf U} \in G $  and ${\bf b} \in \mathbb{R}^2$, one has $ {\bf U} + \lambda {\bf S}({\bf U}, {\bf b}) \in G$ under the condition
	\[
	|\lambda| < \frac{1}{|{\bf b}|} \sqrt{\frac{2p}{(\gamma-1)\rho}}  \, .
	\]
\end{lemma}

\begin{proof}
	The proof can be found in, for example, \cite[Page A476]{wu2021uniformly}.
\end{proof}

\begin{lemma}\label{lemmma:pp-flux-2d}
	For any $ {\bf U} \in G $  and $\lambda \in \mathbb{R}$,  the states ${\bf U} - \lambda {\bf F}_{1}({\bf U}) \in G$
	and ${\bf U} - \lambda {\bf F}_{2} ({\bf U}) \in G$ under the conditions $|\lambda| a_{x} ({\bf U}) \le 1$
	and $|\lambda| a_{y} ({\bf U}) \le 1$, respectively. Here $a_{x} ({\bf U}) := |u_1| + \sqrt{\gamma p/\rho}$
	and $a_{y} ({\bf U}) := |u_2| + \sqrt{\gamma p/\rho}$.
\end{lemma}

\begin{proof}
	The proof is similar to that of Lemma \ref{lemmma:pp-flux} and can be found in, for example, \cite{Zhang2010PP,wu2021geometric}. 
\end{proof}

\subsubsection{Positivity-preserving analysis}
Let us derive the semi-discrete scheme satisfied by the cell averages of the WB CDG method  (\ref{eq:CDG-modify-primal-2d})--(\ref{eq:CDG-modify-dual-2d}).
Denote %the cell average of ${\bf U}_{h}^{C}$ over cell $I_{ij}$ and ${\bf U}_{h}^{D}$ over cell $I_{i+\frac{1}{2},j+\frac{1}{2}}$ by
\[
\overline{{\bf U}}_{ij}^{C}(t) = \frac{1}{\Delta x \Delta y}  \int_{I_{i,j}} {\bf U}_{h}^{C}(x,y,t) \mathrm{d}x\mathrm{d}y \,, \quad
\overline{{\bf U}}_{i+\frac{1}{2},j+\frac{1}{2}}^{D}(t) = \frac{1}{\Delta x\Delta y}  \int_{I_{i+\frac{1}{2},j+\frac{1}{2}}}
{\bf U}_{h}^{D}(x,y,t) \mathrm{d}x \mathrm{d}y \,.
\]
Taking the test function $v = 1$ in (\ref{eq:CDG-modify-primal-2d}) and $w = 1$ in (\ref{eq:CDG-modify-dual-2d})
{and using} the crucial identities in (\ref{eq:average-identity-2d})
 gives
\begin{align}\label{eq:CDG-modify-average-primal-2d}
	\frac{{{\rm d}} \overline{ {\bf U} }_{ij}^{C} }{{{\rm d}} t}
	= \frac{\Big( \overline{ {\bf U} }_{ij}^{D} - \overline{ {\bf U} }_{ij}^{C} \Big)}{\tau_{max}}
	- \frac{ \delta^{x}_{ij} \big( {\bf F}_{1}({\bf U}_{h}^{D}) \big) }{\Delta x}
	- \frac{ \delta^{y}_{ij} \big( {\bf F}_{2}({\bf U}_{h}^{D}) \big) }{\Delta y}
	+ \frac{\big\langle {\bf S}_{h}^{D}, 1 \big\rangle_{ij}}{\Delta x \Delta y}  \,,
\end{align}
\begin{align}\label{eq:CDG-modify-average-dual-2d}\notag
	\frac{{{\rm d}} \overline{ {\bf U} }_{i+\frac{1}{2},j+\frac{1}{2}}^{D} }{{{\rm d}} t}
	&= \frac{\Big( \overline{ {\bf U} }_{i+\frac{1}{2},j+\frac{1}{2}}^{C}
		- \overline{ {\bf U} }_{i+\frac{1}{2},j+\frac{1}{2}}^{D} \Big)}{\tau_{max}}
	- \frac{ \delta^{x}_{i+\frac{1}{2},j+\frac{1}{2}} \big( {\bf F}_{1}({\bf U}_{h}^{C}) \big) }{\Delta x}  \\
	&- \frac{ \delta^{y}_{i+\frac{1}{2},j+\frac{1}{2}} \big( {\bf F}_{2}({\bf U}_{h}^{C}) \big) }{\Delta y}
	+ \frac{\big\langle {\bf S}_{h}^{C}, 1 \big\rangle_{i+\frac{1}{2},j+\frac{1}{2}}}{\Delta x \Delta y}  \,{.}
\end{align}

%\begin{align}
%	\frac{1}{\Delta x \Delta y} \int_{I_{i,j}} \Big( {\bf U}_{h}^{s,C} - {\bf U}_{h}^{s,D} \Big) {\rm d}x{\rm d}y = 0 \,, \quad
%	\frac{1}{\Delta x \Delta y} \int_{I_{i+\frac{1}{2},j+\frac{1}{2}}} \Big( {\bf U}_{h}^{s,D} - {\bf U}_{h}^{s,C} \Big) {\rm d}x{\rm d}y
%	= 0 \,.
%\end{align}
Denote the right-hand {sides of}  (\ref{eq:CDG-modify-average-primal-2d}) and (\ref{eq:CDG-modify-average-dual-2d}) by
${\bf L}_{ij}^{C} \big( {\bf U}_{h}^{C}, {\bf U}_{h}^{D} \big)$ and ${\bf L}_{i+\frac{1}{2},j+\frac{1}{2}}^{D} \big( {\bf U}_{h}^{D}, {\bf U}_{h}^{C} \big)$, respectively. In other words, we have
\begin{align}\label{eq:CDG-modify-average-2d}
	\frac{ {{\rm d}} \overline{ {\bf U} }_{ij}^{C} }{{{\rm d}} t} =  {\bf L}_{ij}^{C} \big( {\bf U}_{h}^{C}, {\bf U}_{h}^{D} \big) \,, \quad
	\frac{ {{\rm d}} \overline{ {\bf U} }_{i+\frac{1}{2},j+\frac{1}{2}}^{D} }{ {{\rm d}} t} = {\bf L}_{i+\frac{1}{2},j+\frac{1}{2}}^{D} \big( {\bf U}_{h}^{D}, {\bf U}_{h}^{C} \big) \,.
\end{align}

On each primal cell $I_{i,j}$, we define the point set
\[
\mathbb{S}_{ij} =  \mathbb{S}_{ij}^{1,1} \cup \mathbb{S}_{ij}^{1,2} \cup \mathbb{S}_{ij}^{2,1}  \cup \mathbb{S}_{ij}^{2,2} \,,\quad
\mathbb{Q}_{ij} =  \mathbb{Q}_{ij}^{1,1} \cup \mathbb{Q}_{ij}^{1,2} \cup \mathbb{Q}_{ij}^{2,1}
\cup \mathbb{Q}_{ij}^{2,2} \,,
\]
\[
\mathbb{S}_{ij}^{\kappa, \tau} = (\mathbb{Q}_{i}^{\kappa,x}\otimes\mathbb{L}_{j}^{\tau,y}) \cup
(\mathbb{L}_{i}^{\kappa,x}\otimes\mathbb{Q}_{j}^{\tau,y}) \cup (\mathbb{Q}_{i}^{\kappa,x}\otimes\mathbb{Q}_{j}^{\tau,y}) \,, \quad
\mathbb{Q}_{ij}^{\kappa, \tau} =  \mathbb{Q}_{i}^{\kappa,x}\otimes\mathbb{Q}_{j}^{\tau,y}  \,,  \quad
\kappa, \tau \in \{1,2\} \,{,}
\]
{and} the parameters $\tilde{\alpha}_{x,ij}^D$ and $\tilde{\alpha}_{y,ij}^D$  by
%Next, define the parameters $\tilde{\alpha}_{x,ij}^D$ and $\tilde{\alpha}_{y,ij}^D$ as
\[
\tilde{\alpha}_{x,ij}^D = \tilde{\alpha}_{x,1}^D + \tilde{\alpha}_{x,2}^D           \,, \quad
\tilde{\alpha}_{x,1}^D  = \max_{(x,y) \in \mathbb{S}_{ij} }  a_{x}(U_{h}^{D})       \,, \quad
\tilde{\alpha}_{x,2}^D  = \frac{\hat{\omega}_{1} \Delta x}{4} \tilde{\alpha}_{s}^D  \,,
\]
\[
\tilde{\alpha}_{y,ij}^D = \tilde{\alpha}_{y,1}^D + \tilde{\alpha}_{y,2}^D           \,, \quad
\tilde{\alpha}_{y,1}^D = \max_{(x,y) \in \mathbb{S}_{ij} }  a_{y}(U_{h}^{D})        \,, \quad
\tilde{\alpha}_{y,2}^D = \frac{\hat{\omega}_{1} \Delta y}{4} \tilde{\alpha}_{s}^D   \,{,}
\]
where
\[
\tilde{\alpha}_{s}^D = \max_{(x,y) \in \mathbb{Q}_{ij}}
\Bigg\{|\nabla\hat{\phi}_{h}^{D}| \sqrt{\frac{(\gamma-1)\rho_{h}^{D}}{ 2 p_{h}^{D} } } \Bigg\} \,, \quad
\nabla \hat{\phi}_{h}^{D} = -\frac{\nabla p_{h}^{s,D}}{\rho_{h}^{s,D}} -
\frac{ \Big(
	\jump{ p_{h}^{s,D} }_{ij}^x , \jump{ p_{h}^{s,D} }_{ij}^y \Big)^\top }
{ \overline{(\rho_{h}^{s,D})}_{ij} \Delta x \Delta y} \,,
\]
\[
\jump{ p_{h}^{s,D} }_{ij}^x := \int_{y_{j-\frac{1}{2}}}^{y_{j+\frac{1}{2}}}
\Big( p_{h}^{s,D}(x_{i}^{+}, y) - p_{h}^{s,D}(x_{i}^{-}, y) \Big) {\rm d}y, \
\jump{ p_{h}^{s,D} }_{ij}^y := \int_{x_{i-\frac{1}{2}}}^{x_{i+\frac{1}{2}}}
\Big( p_{h}^{s,D}(x, y_{j}^{+}) - p_{h}^{s,D}(x, y_{j}^{-}) \Big) {\rm d}x  \,.
\]
Similarly, on each dual cell $I_{i+\frac{1}{2},j+\frac{1}{2}}$, we define the point set $\mathbb{S}_{i+\frac{1}{2},j+\frac{1}{2}}$
and $\mathbb{Q}_{i+\frac{1}{2},j+\frac{1}{2}}$, which are respectively the shifts of the point sets  $\mathbb{S}_{i,j}$ and $\mathbb{Q}_{i,j}$ with $\frac{\Delta x}{2}$ in the $x$-direction and $\frac{\Delta y}{2}$ in the $y$-direction{, and} the parameters $\tilde{\alpha}_{x,i+\frac{1}{2},j+\frac{1}{2}}^C$ and $\tilde{\alpha}_{y,i+\frac{1}{2},j+\frac{1}{2}}^C$ by
\[
\tilde{\alpha}_{x,i+\frac{1}{2},j+\frac{1}{2}}^C = \tilde{\alpha}_{x,1}^C + \tilde{\alpha}_{x,2}^C           \,, \quad
\tilde{\alpha}_{x,1}^C  = \max_{(x,y) \in \mathbb{S}_{i+\frac{1}{2},j+\frac{1}{2}} }  a_{x}({\bf U}_{h}^{C})       \,, \quad
\tilde{\alpha}_{x,2}^C  = \frac{\hat{\omega}_{1} \Delta x}{4} \tilde{\alpha}_{s}^C  \,,
\]
\[
\tilde{\alpha}_{y,i+\frac{1}{2},j+\frac{1}{2}}^C = \tilde{\alpha}_{y,1}^C + \tilde{\alpha}_{y,2}^C           \,, \quad
\tilde{\alpha}_{y,1}^C = \max_{(x,y) \in \mathbb{S}_{i+\frac{1}{2},j+\frac{1}{2}} }  a_{y}({\bf U}_{h}^{C})        \,, \quad
\tilde{\alpha}_{y,2}^C = \frac{\hat{\omega}_{1} \Delta y}{4} \tilde{\alpha}_{s}^C   \,{,}
\]
where
\[
\tilde{\alpha}_{s}^C = \max_{(x,y) \in \mathbb{Q}_{i+\frac{1}{2},j+\frac{1}{2}}}
\Bigg\{|\nabla\hat{\phi}_{h}^{C}| \sqrt{\frac{(\gamma-1)\rho_{h}^{C}}{ 2 p_{h}^{C} } } \Bigg\} \,, \
\nabla \hat{\phi}_{h}^{C} = -\frac{\nabla p_{h}^{s,C}}{\rho_{h}^{s,C}} -
\frac{ \Big( \jump{ p_{h}^{s,C} }_{i+\frac12,j+\frac12}^x,
\jump{ p_{h}^{s,C} }_{i+\frac12,j+\frac12}^y \Big)^\top }
{ \overline{(\rho_{h}^{s,C})}_{i+\frac{1}{2},j+\frac{1}{2}} \Delta x \Delta y}{,} %\,,
\]
{with}
\[
\jump{ p_{h}^{s,C} }_{i+\frac12,j+\frac12}^x := \int_{y_{j}}^{y_{j+1}}
\Big( p_{h}^{s,C}(x_{i+\frac{1}{2}}^{+}, y) - p_{h}^{s,C}(x_{i+\frac{1}{2}}^{-}, y) \Big) {\rm d} y  \,,
\]
\[
\jump{ p_{h}^{s,C} }_{i+\frac12,j+\frac12}^y := \int_{x_{i}}^{x_{i+1}}
\Big( p_{h}^{s,C}(x, y_{j+\frac{1}{2}}^{+}) - p_{h}^{s,C}(x, y_{j+\frac{1}{2}}^{-}) \Big) {\rm d} x  \,.
\]

\begin{theorem}\label{theorem:PP-WB-CDG-2d}
	Assume that the numerical solutions ${\bf U}_{h}^{C}(x,y), {\bf U}_{h}^{D}(x,y)$  satisfy
	\begin{align}\label{eq:pp-condition-2d}
		{\bf U}_{h}^{C}(x,y)   \in G ~~ \forall (x,y) \in \mathbb{S}_{ij},  \quad {\bf U}_{h}^{D}(x,y)   \in G ~~ \forall  (x,y) \in \mathbb{S}_{i+\frac{1}{2},j+\frac{1}{2}} \,,  \quad \forall i,j \,,  \end{align}
{and the projected stationary hydrostatic  solutions ${\bf U}_{h}^{s,C}(x,y), {\bf U}_{h}^{s,D}(x,y)$ satisfy}
\begin{align*}
		{\bf U}_{h}^{s,C}(x,y) \in G ~~ \forall (x,y) \in \mathbb{S}_{ij}, \quad {\bf U}_{h}^{s,D}(x,y) \in G ~~ \forall  (x,y) \in  \mathbb{S}_{i+\frac{1}{2},j+\frac{1}{2}} \,,  \quad \forall i,j.
	\end{align*}
If $\overline{ {\bf U} }_{ij}^{C}, \overline{ {\bf U} }_{i+\frac{1}{2},j+\frac{1}{2}}^{D} \in G$,
	then the positivity-preserving property
	\begin{align*}
		\overline{ {\bf U} }_{ij}^{C} + \Delta t {\bf L}_{ij}^{C} \big( {\bf U}_{h}^{C},{\bf U}_{h}^{D} \big) \in G \,,  \quad
		\overline{ {\bf U} }_{i+\frac{1}{2},j+\frac{1}{2}}^{D} + \Delta t {\bf L}_{i+\frac{1}{2},j+\frac{1}{2}}^{D}
		\big( {\bf U}_{h}^{D},{\bf U}_{h}^{C} \big) \in G \,,  \quad \forall  i,j \,,
	\end{align*}
	holds under the CFL-type condition
	\begin{align*}%\label{eq:CFL-2d}
		\frac{\Delta t}{\Delta x} \tilde{\alpha}_{x} + \frac{\Delta t}{\Delta y} \tilde{\alpha}_{y}
		<  \frac{\theta \hat{\omega}_{1}}{2}\,, \quad \theta = \frac{\Delta t}{\tau_{max}} \in (0,1] \,,
	\end{align*}
	with
	\begin{align*}
		\tilde{\alpha}_{x} =  \max\limits_{i,j} \max \{\tilde{\alpha}_{x,ij}^D, \tilde{\alpha}_{x,i+\frac{1}{2},j+\frac{1}{2}}^C \} \,, \quad
		\tilde{\alpha}_{y} =  \max\limits_{i,j} \max \{\tilde{\alpha}_{y,ij}^D, \tilde{\alpha}_{y,i+\frac{1}{2},j+\frac{1}{2}}^C \} \,.
	\end{align*}
\end{theorem}

\begin{proof}
	Using (\ref{eq:CDG-modify-average-2d}) gives
	\begin{align*}
		\overline{ {\bf U} }_{ij}^{C} + \Delta t {\bf L}_{ij}^{C} \big( {\bf U}_{h}^{C},{\bf U}_{h}^{D} \big)
		& = (1-\theta)\overline{{\bf U}}_{ij}^{C} + \theta \overline{{\bf U}}_{ij}^{D}
		+ \frac{\Delta t }{\Delta x \Delta y} \big\langle {\bf S}_{h}^{D}, 1 \big\rangle_{ij}         \\ \notag
		& - \lambda_{x} \delta^{x}_{ij} \big( {\bf F}_{1}({\bf U}_{h}^{D}) \big)
		- \lambda_{y} \delta^{y}_{ij}  \big( {\bf F}_{2}({\bf U}_{h}^{D}) \big)     \\ \notag
		& = (1-\theta)\overline{ {\bf U} }_{ij}^{C} + \Big( \eta\theta \overline{ {\bf U} }_{ij}^{D}
		+ \frac{\Delta t }{\Delta x \Delta y} \big\langle {\bf S}_{h}^{D}, 1 \big\rangle_{ij} \Big)   \\ \notag
		& + \Big( (1-\eta)\theta \overline{ {\bf U} }_{ij}^{D}
		- \lambda_{x} \delta^{x}_{ij} \big( {\bf F}_{1}({\bf U}_{h}^{D}) \big)
		- \lambda_{y} \delta^{y}_{ij} \big( {\bf F}_{2}({\bf U}_{h}^{D}) \big) \Big)  \\ \notag
		& = (1-\theta)\overline{ {\bf U} }_{ij}^{C} + \eta\theta {\bf L}_{h,{\bf S}}
		+ (1-\eta)\theta {\bf L}_{h,{\bf F}} \,,
	\end{align*}
	where $\lambda_{x} = \frac{\Delta t}{\Delta x}, \lambda_{y} = \frac{\Delta t}{\Delta y}, \theta = \frac{\Delta t}{\tau_{max}}, \eta \in (0,1)$ is a constant, and ${\bf L}_{h,{\bf F}}$ and $ {\bf L}_{h,{\bf S}} $ are defined by
	\begin{align*}
		{\bf L}_{h,{\bf F}} =& \overline{  {\bf U} }_{ij}^{D} - \frac{\lambda_{x}}{(1-\eta)\theta } \delta^{x}_{ij}
		{ \big( {\bf F}_{1}({\bf U}_{h}^{D})  \big) }
		- \frac{\lambda_{y}}{(1-\eta)\theta } \delta^{y}_{ij}
		\big( {\bf F}_{2} ({\bf U}_{h}^{D}) \big)  \,,
\\
		{\bf L}_{h,{\bf S}} =& \overline{ {\bf U} }_{ij}^{D} + \frac{\Delta t}{\eta\theta\Delta x\Delta y }
		\big\langle {\bf S}_{h}^{D}, 1 \big\rangle_{ij} \,.
	\end{align*}
	Using the convexity of set $G$ and the exactness of the quadrature rule can derive that
	\begin{align*}
		& \overline{{\bf U}}_{ij}^{D}  = \frac{1}{\Delta x \Delta y}  \int_{I_{ij}}  {\bf U}_{h}^{D}(x,y) 
		{{\rm d}}x {{\rm d}}y
		= \frac{1}{\Delta x} \int_{x_{i-\frac{1}{2}}}^{x_{i+\frac{1}{2}}}
		\bigg( \frac{1}{2} \sum_{\kappa = 1 }^{2} \sum_{\alpha = 1}^{N} \omega_{\alpha}
		{\bf U}_{h}^{D}(x,y_{j}^{\kappa,\alpha}) \bigg) {{\rm d}}x  \\ \notag
		= & \frac{1}{2} \sum_{\kappa = 1 }^{2} \sum_{\alpha = 1}^{N} \omega_{\alpha}
		\bigg( \frac{1}{\Delta x} \int_{x_{i-\frac{1}{2}}}^{x_{i+\frac{1}{2}}}
		{\bf U}_{h}^{D}(x,y_{j}^{\kappa,\alpha})  {{\rm d}}x \bigg)
		= \frac{1}{2} \sum_{\kappa = 1 }^{2} \sum_{\alpha = 1}^{N} \omega_{\alpha}
		\bigg( \frac{1}{2} \sum_{\tau = 1 }^{2} \sum_{\beta = 1}^{L} \hat{\omega}_{\beta}
		{\bf U}_{h}^{D}(\hat{x}_{i}^{\tau,\beta},y_{j}^{\kappa,\alpha}) \bigg)  \\ \notag
		= & \frac{1}{2} \sum_{\kappa = 1 }^{2} \sum_{\alpha = 1}^{N} \omega_{\alpha}
		\bigg( \frac{\hat{\omega}_{1}}{2} {\bf U}_{h}^{D}(x_{i-\frac{1}{2}},y_{j}^{\kappa,\alpha}) +
		\frac{\hat{\omega}_{1}}{2} {\bf U}_{h}^{D}(x_{i+\frac{1}{2}},y_{j}^{\kappa,\alpha}) +
		(1-\hat{\omega}_{1} )\mathbb{E}^{\alpha}_{1,ij} \bigg) \,,
	\end{align*}
	where $\hat{\omega}_{1} = \hat{\omega}_{L}$ has been used, and
	\begin{align*}
		\mathbb{E}^{\alpha}_{1,ij} = \frac{1}{(1-\hat{\omega}_{1} )}
		\bigg( \frac{1}{2}  \sum_{\beta = 2}^{L} \hat{\omega}_{\beta}
		{\bf U}_{h}^{D}(\hat{x}_{i}^{1,\beta},y_{j}^{\kappa,\alpha})
		+ \frac{1}{2}  \sum_{\beta = 1}^{L-1} \hat{\omega}_{\beta}
		{\bf U}_{h}^{D}(\hat{x}_{i}^{2,\beta},y_{j}^{\kappa,\alpha}) \bigg) \in G \,.
	\end{align*}
	Similarly, one has
	\begin{align*}
		\overline{{\bf U}}_{ij}^{D}  = \frac{1}{2} \sum_{\kappa = 1 }^{2} \sum_{\alpha = 1}^{N} \omega_{\alpha}
		\bigg( \frac{\hat{\omega}_{1}}{2} {\bf U}_{h}^{D}(x_{i}^{\kappa,\alpha},y_{j-\frac{1}{2}}) +
		\frac{\hat{\omega}_{1}}{2} {\bf U}_{h}^{D}(x_{i}^{\kappa,\alpha},y_{j+\frac{1}{2}}) +
		(1-\hat{\omega}_{1} )\mathbb{E}^{\alpha}_{2,ij} \bigg) \,,
	\end{align*}
	with
	\begin{align*}
		\mathbb{E}^{\alpha}_{2,ij} = \frac{1}{(1-\hat{\omega}_{1} )}
		\bigg( \frac{1}{2}  \sum_{\beta = 2}^{L} \hat{\omega}_{\beta}
		{\bf U}_{h}^{D}( x_{i}^{\kappa,\alpha}, \hat{y}_{j}^{1,\beta} )
		+ \frac{1}{2}  \sum_{\beta = 1}^{L-1} \hat{\omega}_{\beta}
		{\bf U}_{h}^{D}( x_{i}^{\kappa,\alpha}, \hat{y}_{j}^{2,\beta} ) \bigg) \in G \,.
	\end{align*}
	Note that ${\bf L}_{h,{\bf F}}$ can be reformulated as
	\begin{align*}  \notag
		{\bf L}_{h,{\bf F}} = \frac{a_{x} \lambda_{x}}{\lambda}
		\bigg( \overline{{\bf U}}_{ij}^{D} - \frac{\lambda}{(1-\eta)\theta a_{x} }
		\delta^{x}_{ij} \big( {\bf F}_{1}({\bf U}_{h}^{D}) \big) \bigg)
		+ \frac{a_{y} \lambda_{y}}{\lambda}
		\bigg( \overline{{\bf U}}_{ij}^{D}-  \frac{\lambda}{(1-\eta)\theta a_{y} }
		\delta^{y}_{ij} \big( {\bf F}_{2} ({\bf U}_{h}^{D}) \big) \bigg) \,,
	\end{align*}
	where $\lambda = a_{x}\lambda_{x}  +  a_{y}\lambda_{y}$, $a_{x} = \tilde{\alpha}_{x,1}^D$,
	$a_{y} = \tilde{\alpha}_{y,1}^D$, and
	\begin{align*}  \notag
		\overline{{\bf U}}_{ij}^{D} - \frac{\lambda}{(1-\eta)\theta a_{x} }
		\delta^{x}_{ij} \big( {\bf F}_{1}({\bf U}_{h}^{D}) \big)
		&= \frac{1}{2} \sum_{\kappa = 1 }^{2} \sum_{\alpha = 1}^{N} \omega_{\alpha}
		\bigg( \frac{\hat{\omega}_{1}}{2} \mathbb{E}^{+}_{1,i-\frac{1}{2}, j} +
		\frac{\hat{\omega}_{1}}{2} \mathbb{E}^{-}_{1,i+\frac{1}{2},j} +
		(1-\hat{\omega}_{1} )\mathbb{E}^{\alpha}_{1,ij} \bigg) \,,  \\ \notag
		\overline{{\bf U}}_{ij}^{D}-  \frac{\lambda}{(1-\eta)\theta a_{y} }
		\delta^{y}_{ij}  \big( {\bf F}_{2} ({\bf U}_{h}^{D}) \big)
		&= \frac{1}{2} \sum_{\kappa = 1 }^{2} \sum_{\alpha = 1}^{N} \omega_{\alpha}
		\bigg( \frac{\hat{\omega}_{1}}{2} \mathbb{E}^{+}_{2,i,j-\frac{1}{2}} +
		\frac{\hat{\omega}_{1}}{2} \mathbb{E}^{-}_{2,i,j+\frac{1}{2}} +
		(1-\hat{\omega}_{1} )\mathbb{E}^{\alpha}_{2,ij} \bigg) \,,
	\end{align*}
	with
	\begin{align*}  \notag
		\mathbb{E}^{\pm}_{1,i\mp\frac{1}{2}, j} &= {\bf U}_{h}^{D}(x_{i\mp\frac{1}{2}},y_{j}^{\kappa,\alpha}) \pm
		\frac{2 \lambda}{\hat{\omega}_{1} (1-\eta) \theta a_{x} }
		{\bf F}_{1}({\bf U}_{h}^{D}(x_{i\mp\frac{1}{2}},y_{j}^{\kappa,\alpha}))   \,, \\ \notag
		\mathbb{E}^{\pm}_{2,i,j\mp\frac{1}{2}}  &= {\bf U}_{h}^{D}(x_{i}^{\kappa,\alpha},y_{j\mp\frac{1}{2}}) \pm
		\frac{2 \lambda}{\hat{\omega}_{1} (1-\eta) \theta a_{y} }
		{\bf F}_{2} ({\bf U}_{h}^{D}(x_{i}^{\kappa,\alpha},y_{j\mp\frac{1}{2}}))    \,.
	\end{align*}
	Thanks to Lemma \ref{lemmma:pp-flux-2d},  we have $\mathbb{E}^{\pm}_{1,i\mp\frac{1}{2}, j} \in G$ and $\mathbb{E}^{\pm}_{2,i,j\mp\frac{1}{2}} \in G$, as long as
	\begin{align}\label{9883}
		\lambda_{x} \tilde{\alpha}_{x,1}^D  + \lambda_{y} \tilde{\alpha}_{y,1}^D < \frac{(1-\eta)\theta \hat{\omega}_{1}}{2}  \,.
	\end{align}
	Using the convexity of $G$, we further obtain ${\bf L}_{h,{\bf F}} \in G$ under \eqref{9883}. Next, we discuss the term ${\bf L}_{h,{\bf S}}$.
	The source term approximations $\big\langle \big(S_{2,h}^{D},S_{3,h}^{D}\big)^\top, 1 \big\rangle_{ij}$ in the momentum equations are
	\begin{align*} \notag
		\big\langle \big(S_{2,h}^{D},S_{3,h}^{D}\big)^\top, 1 \big\rangle_{ij}
		&= \frac{\Delta x \Delta y}{4}  \sum_{\kappa,\alpha}  \sum_{\tau,\beta}  \omega_{\alpha} \omega_{\beta}
		\left( \frac{ \rho_{h}^{D}(x_{i}^{\kappa,\alpha}, y_{j}^{\tau,\beta}) }
		{ \rho_{h}^{s,D}(x_{i}^{\kappa,\alpha}, y_{j}^{\tau,\beta}) }
		-  \frac{\overline{(\rho_{h}^{D})}_{ij} }{ \overline{(\rho_{h}^{s,D})}_{ij}}  \right)
		(\nabla p_{h}^{s,D} )(x_{i}^{\kappa,\alpha}, y_{j}^{\tau,\beta})  \\
		& + \frac{\overline{(\rho_{h}^{D})}_{ij} }{ \overline{(\rho_{h}^{s,D})}_{ij}}
		\Big(\Delta y \delta^{x}_{ij} (p_{h}^{s,D}),  \Delta x \delta^{y}_{ij} (p_{h}^{s,D})  \Big)^\top \,.
	\end{align*}
	Based on
	\begin{align} \notag
		\int_{I_{ij}} \nabla p_{h}^{s,D} {\rm d}x{\rm d}y   &=  \frac{\Delta x \Delta y}{4}
		\sum_{\kappa,\alpha}  \sum_{\tau,\beta}  \omega_{\alpha} \omega_{\beta}
		\nabla p_{h}^{s,D}(x_{i}^{\kappa,\alpha}, y_{j}^{\tau,\beta}) \,,
		\\ \notag
		\int_{\partial I_{ij}} p_{h}^{s,D} {\bf n}{\rm d}s &=  \Big(\Delta y \delta^{x}_{ij} (p_{h}^{s,D}),
		\Delta x \delta^{y}_{ij} (p_{h}^{s,D}) \Big)^\top  \,,
	\end{align}
	and the identity
	\begin{align*}
		\int_{\partial I_{ij}} p_{h}^{s,D} {\bf n}{\rm d}s - \int_{I_{ij}} \nabla p_{h}^{s,D} {\rm d}x{\rm d}y
		= \Big( \jump{ p_{h}^{s,D} }_{ij}^x , \jump{ p_{h}^{s,D} }_{ij}^y \Big)^\top \,,
	\end{align*}
	we reformulate $\big\langle \big(S_{2,h}^{D},S_{3,h}^{D}\big)^\top, 1 \big\rangle_{ij}$ as
	\begin{align*} \notag
		&\big\langle \big(S_{2,h}^{D},S_{3,h}^{D}\big)^\top, 1 \big\rangle_{ij}
		= \frac{\Delta x \Delta y}{4}  \sum_{\kappa,\alpha}  \sum_{\tau,\beta}  \omega_{\alpha} \omega_{\beta}
		\frac{ \rho_{h}^{D}(x_{i}^{\kappa,\alpha}, y_{j}^{\tau,\beta}) }
		{ \rho_{h}^{s,D}(x_{i}^{\kappa,\alpha}, y_{j}^{\tau,\beta}) }
		(\nabla p_{h}^{s,D} )(x_{i}^{\kappa,\alpha}, y_{j}^{\tau,\beta})  \\
		& + \frac{\overline{(\rho_{h}^{D})}_{ij} }{ \overline{(\rho_{h}^{s,D})}_{ij}}
		\Big( \jump{ p_{h}^{s,D} }_{ij}^x , \jump{ p_{h}^{s,D} }_{ij}^y \Big)^\top
		= -\frac{\Delta x \Delta y}{4}  \sum_{\kappa,\alpha}  \sum_{\tau,\beta}  \omega_{\alpha} \omega_{\beta}
		\rho_{h}^{D}(x_{i}^{\kappa,\alpha}, y_{j}^{\tau,\beta})
		\nabla \hat{\phi}_{h}^{D}(x_{i}^{\kappa,\alpha}, y_{j}^{\tau,\beta}) \,,
	\end{align*}
	where
	\[
	-\nabla \hat{\phi}_{h}^{D}(x_{i}^{\kappa,\alpha}, y_{j}^{\tau,\beta})
	= \frac{ \nabla p_{h}^{s,D}(x_{i}^{\kappa,\alpha}, y_{j}^{\tau,\beta})}
	{ \rho_{h}^{s,D}(x_{i}^{\kappa,\alpha}, y_{j}^{\tau,\beta}) } +
	\frac{ \Big( \jump{ p_{h}^{s,D} }_{ij}^x , \jump{ p_{h}^{s,D} }_{ij}^y \Big)^\top }
	{ \overline{(\rho_{h}^{s,D})}_{ij} \Delta x \Delta y} \,.
	\]
	Similarly, we can rewrite $\big\langle S_{4,h}^{D}, 1 \big\rangle_{ij}$ in the energy equation as
	\begin{align*}
		\big\langle S_{4,h}^{D}, 1 \big\rangle_{ij}
		= -\frac{\Delta x \Delta y}{4}  \sum_{\kappa,\alpha}  \sum_{\tau,\beta}  \omega_{\alpha} \omega_{\beta}
		{\bf m}_{h}^{D}(x_{i}^{\kappa,\alpha}, y_{j}^{\tau,\beta}) \cdot
		\nabla \hat{\phi}_{h}^{D}(x_{i}^{\kappa,\alpha}, y_{j}^{\tau,\beta})  \,.
	\end{align*}
	Thus, ${\bf L}_{h,{\bf S}} = \overline{ {\bf U} }_{ij}^{D} + \frac{\Delta t}{\eta\theta\Delta x\Delta y } \big\langle {\bf S}_{h}^{D}, 1 \big\rangle_{ij}$ can be reformulated as
	\begin{align} \notag
		{\bf L}_{h,{\bf S}} = \sum_{\kappa,\alpha}  \sum_{\tau,\beta}  \frac{\omega_{\alpha} \omega_{\beta}}{4}
		\Big( {\bf U}_{h}^{D}(x_{i}^{\kappa,\alpha}, y_{j}^{\tau,\beta}) +
		\frac{\Delta t}{\eta\theta} \hat{ {\bf S}}_{h}^{D}(x_{i}^{\kappa,\alpha}, y_{j}^{\tau,\beta}) \Big){,}
	\end{align}
	with $\hat{{\bf S}}_{h}^{D}(x_{i}^{\kappa,\alpha}, y_{j}^{\tau,\beta}) :=
	\Big(0, - (\rho_{h}^{D}\nabla \hat{\phi}_{h}^{D}) (x_{i}^{\kappa,\alpha}, y_{j}^{\tau,\beta}),
	- ({\bf m}_{h}^{D}\cdot\nabla \hat{\phi}_{h}^{D}) (x_{i}^{\kappa,\alpha}, y_{j}^{\tau,\beta}) \Big)^\top$.
	Thanks to  Lemma \ref{lemmma:pp-source-2d}, we have ${\bf L}_{h,{\bf S}} \in G$  under the condition
	\begin{equation*}
		\Delta t < \eta\theta \min_{(x,y) \in \mathbb{Q}_{ij}}
		\Bigg\{ \frac{1}{|\nabla\hat{\phi}_{h}^{D}|} \sqrt{\frac{2 p_{h}^{D}}{(\gamma-1)\rho_{h}^{D}}} \Bigg\} \,,
	\end{equation*}
	or equivalently
	\begin{align*}
		\lambda_{x} \tilde{\alpha}_{x,2}^D + \lambda_{y} \tilde{\alpha}_{y,2}^D < \frac{\eta \theta \hat{\omega}_{1}}{2} \,.
	\end{align*}
	Combining {those} results, we conclude that if %$\lambda_{x}, \lambda_{y}$ satisfies
	\begin{align}\label{condition8393}
		(\lambda_{x}, \lambda_{y}) \in  \Big\{ (\lambda_{x} ,\lambda_{y})\in (\mathbb{R}^{+})^2:
		\lambda_{x} \tilde{\alpha}_{x,1}^D + \lambda_{y} \tilde{\alpha}_{y,1}^D < \frac{(1-\eta)\theta \hat{\omega}_{1}}{2}, ~
		\lambda_{x} \tilde{\alpha}_{x,2}^D + \lambda_{y} \tilde{\alpha}_{y,2}^D < \frac{\eta \theta \hat{\omega}_{1}}{2}  \Big\} \,,
	\end{align}
	then $ \overline{ {\bf U} }_{ij}^{C} + \Delta t {\bf L}_{ij}^{C}
	\big( {\bf U}_{h}^{C},{\bf U}_{h}^{D} \big) \big) \in G$. Since the
	parameters $\eta$ can be chosen arbitrarily in this proof, we specify
	$$
	\eta = \frac{ \lambda_{x} \tilde{\alpha}_{x,2}^D + \lambda_{y} \tilde{\alpha}_{y,2}^D } {  \lambda_{x} \tilde{\alpha}_{x,ij}^D + \lambda_{y} \tilde{\alpha}_{y,ij}^D } = \frac{ \lambda_{x} \tilde{\alpha}_{x,2}^D + \lambda_{y} \tilde{\alpha}_{y,2}^D } { \lambda_{x} (\tilde{\alpha}_{x,1}^D + \tilde{\alpha}_{x,2}^D) +
		\lambda_{y} (\tilde{\alpha}_{y,1}^D + \tilde{\alpha}_{y,2}^D) },
	$$
	so that the condition \eqref{condition8393} becomes
	\[
	\lambda_{x} (\tilde{\alpha}_{x,1}^D + \tilde{\alpha}_{x,2}^D) +
	\lambda_{y} (\tilde{\alpha}_{y,1}^D + \tilde{\alpha}_{y,2}^D) =
	\lambda_{x} \tilde{\alpha}_{x,ij}^D + \lambda_{y} \tilde{\alpha}_{y,ij}^D
	<  \frac{\theta \hat{\omega}_{1}}{2} \,.
	\]
	Similar arguments can show that $\overline{ {\bf U} }_{i+\frac{1}{2},j+\frac{1}{2}}^{D} + \Delta t  {\bf L}_{i+\frac{1}{2},j+\frac{1}{2}}^{D} \big( {\bf U}_{h}^{D}, {\bf U}_{h}^{C} \big) \in G$. The proof is completed.
\end{proof}

Theorem \ref{theorem:PP-WB-CDG-2d} provides a sufficient condition for the proposed high-order WB
CDG schemes (\ref{eq:CDG-modify-primal-2d}) and (\ref{eq:CDG-modify-dual-2d}) to be positivity-preserving,
when {the} SSP time discretization is used. The condition (\ref{eq:pp-condition-2d}) can again be enforced
by a simple positivity-preserving limiter similar to the one-dimensional case; see Section \ref{sec:2Dlimiter}. 
With the positivity-preserving limiter applied at each stage of
the SSP Runge-Kutta method, the resulting fully discrete CDG schemes are positivity-preserving.
%under some CFL-type conditions.

\subsection{Positivity-preserving limiting operators}\label{sec:2Dlimiter}
Introduce the following two sets
\[
\overline{\mathbb{G}}_{h}^{C,k} := \Big \{ {\bf v} \in [\mathbb{V}_{h}^{C,k}]^4 : ~
\frac{1}{\Delta x \Delta y} \int_{I_{ij}} {\bf v} (x,y) {\rm d}x {\rm d}y \in G \,, ~ \forall ~ i,j  \Big\}  \,,
\]
\[
\mathbb{G}_{h}^{C,k} := \Big \{ {\bf v} \in \overline{\mathbb{G}}_{h}^{C,k} :
{\bf v}\big|_{I_{ij}} (x,y) \in G \,, ~ \forall ~ (x,y) \in \mathbb{S}_{ij}   \,, ~ \forall ~ i,j  \Big\}  \,.
\]
For any ${\bf U}_{h}^C \in \overline{\mathbb{G}}_{h}^{C,k}$ with ${\bf U}_{h}^C \big|_{I_{ij}} := {\bf U}_{ij}^C(x,y)$,
 {following \cite{Zhang2012robust,Zhang2010PP} we} define the positivity-preserving limiting operator ${\Pi}_{h}^C : \overline{\mathbb{G}}_{h}^{C,k} \longrightarrow \mathbb{G}_{h}^{C,k}$ as follows
\[
{\Pi}_{h}^C {\bf U}_{h}^C \big|_{I_{ij}} =
\theta_{ij}^{(2)} ( \hat{{\bf U}}_{ij}^C(x,y) - \overline{{\bf U}}_{ij}^C ) + \overline{{\bf U}}_{ij}^C \,, \quad
\theta_{ij}^{(2)} = \min \Bigg\{1, ~ \frac{ p(\overline{{\bf U}}_{ij}^C) - \epsilon_{2} }{ p(\overline{{\bf U}}_{ij}^C)
	- \min\limits_{(x,y) \in \mathbb{S}_{ij}} p(\hat{{\bf U}}_{ij}^C(x,y)) } \Bigg\} \,,
\]
where $\hat{{\bf U}}_{ij}^C(x,y) = (\hat{\rho}_{ij}^C(x,y), {\bf m}_{ij}^C(x,y), E_{ij}^C(x,y) )^{\top} $, and
$\hat{\rho}_{ij}^C(x,y)$ is a modification of the density $\rho_{ij}^C(x,y)$ given by
\[
\hat{\rho}_{ij}^C(x,y)  = \theta_{ij}^{(1)} ( \rho_{ij}^C(x,y) - \overline{\rho}_{ij}^C ) + \overline{\rho}_{ij}^C \,, \quad
\theta_{ij}^{(1)} = \min \Bigg\{1, ~ \frac{\overline{\rho}_{ij}^C - \epsilon_{1} }{ \overline{\rho}_{ij}^C
	- \min\limits_{(x,y) \in \mathbb{S}_{ij}} \rho_{ij}^C(x,y) } \Bigg\} \,,
\]
We take
$\epsilon_{1} = \min\{ 10^{-13}, \overline{\rho}_{ij}^C \}$, $\epsilon_{2} = \min\{ 10^{-13}, p(\overline{{\bf U}}_{ij}^C) \}$.
The sets $\overline{\mathbb{G}}_{h}^{D,k}, \mathbb{G}_{h}^{D,k}$ and {the} positivity-preserving limiting operator
${\Pi}_{h}^D : \overline{\mathbb{G}}_{h}^{D,k} \longrightarrow \mathbb{G}_{h}^{D,k}$ defined on the dual mesh
$I_{i+\frac{1}{2},j+\frac{1}{2}}$ are very similar, and the details are omitted here.

\section{Numerical examples} \label{section:numerical-example}

{This section presents} several one- and two-dimensional tests to demonstrate the WB and positivity-preserving
properties of the proposed CDG methods on uniform Cartesian meshes.
The explicit third-order SSP Runge-Kutta method is employed for {the} time discretization.
For comparison, we will also show the numerical
results of the standard non-WB CDG schemes with the straightforward source term approximation
and the original numerical dissipation term.
Unless explained specifically, we use the ideal EOS (\ref{eq:eos}) with $\gamma  = 1.4$,
the CFL numbers for the third-order and fourth-order
 CDG methods are taken as $0.25$ and $0.15$, respectively, and the parameter $\theta =\frac{\Delta t}{\tau_{max}} = 1$. In all the numerical examples, the schemes are implemented by using C/C++ language with double precision.

\subsection{One-dimensional tests}

\subsubsection{Example 1: Accuracy test} \label{example:accuracy-test-1d}

We start with a one-dimensional example \cite{Xing2013FDWB} to demonstrate the accuracy of the proposed WB CDG schemes
for the Euler equations under a linear gravitational field $\phi_x = g = 1$. The  time-dependent exact solution of this example
is given by
\begin{align} \notag
	\rho(x,t) & = 1 + 0.2\sin(\pi(x-u_{0}t)) \,, \quad  u(x,t)  = u_{0} \,, \\ \notag
	p(x,t) & = p_{0} + u_{0} t -x + 0.2\cos(\pi(x-u_{0}t))/\pi \,,
\end{align}
where the constants $u_{0} = 1.0$, $p_{0} = 4.5$. The computational domain $\Omega =  [0,2]$
is divided into $N$ uniform cells and  the boundary condition is set as the exact solution  on $\partial \Omega$.
To match the temporal and spatial accuracy, we use $\Delta t = 0.15 (\Delta x)^\frac{4}{3} /\tilde{\alpha}_{x}$ for the fourth-order
WB CDG scheme. The $L^{1}$ errors and corresponding convergence rates at $t = 0.1$ are shown in Tables  \ref{tab:accuracy-test-1d-order3} and \ref{tab:accuracy-test-1d-order4}. We clearly observe that the expected third-order and fourth-order convergence rates are achieved by the WB CDG schemes. This indicates that our novel projection, modification of the dissipation term and WB source term approximation do not destroy the accuracy of our proposed WB CDG schemes.

\begin{table}[htbp]
	\centering
	\caption{Example 1: $L^{1}$ errors at $t = 0.1$ and corresponding convergence rates
		     for the third-order WB CDG scheme at different grid resolutions.}
	\begin{tabular}{c|c|c|c|c|c|c}
		\hline   N  &\multicolumn{2}{|c|}{$\rho$}  &\multicolumn{2}{|c|}{$\rho u$ }  &\multicolumn{2}{|c}{$E$ }  \\
		\hline   Mesh       &$L^{1}$ error      & Order & $L^{1}$ error     & Order  &$L^{1}$ error      & Order \\
		\hline  $8$ &1.99e-04	&-	&2.01e-04	&-	&1.94e-04	&-	             \\
		\hline  $16$ &2.54e-05	&2.97	&2.54e-05	&2.98	&2.39e-05	&3.02	 \\
		\hline  $32$ &3.16e-06	&3.01	&3.17e-06	&3.00	&2.98e-06	&3.00	 \\
		\hline  $64$ &3.96e-07	&3.00	&3.96e-07	&3.00	&3.72e-07	&3.00	 \\
		\hline  $128$ &4.94e-08	&3.00	&4.95e-08	&3.00	&4.65e-08	&3.00    \\
		\hline
	\end{tabular}
	\label{tab:accuracy-test-1d-order3}
\end{table}

\begin{table}[htbp]
	\centering
	\caption{Same as Table \ref{tab:accuracy-test-1d-order3}, except for
the fourth-order WB CDG scheme.}
	\begin{tabular}{c|c|c|c|c|c|c}
		\hline   N  &\multicolumn{2}{|c|}{$\rho$}  &\multicolumn{2}{|c|}{$\rho u$ }  &\multicolumn{2}{|c}{$E$ }  \\
		\hline  Mesh        &$L^{1}$ error      & Order & $L^{1}$ error     & Order  &$L^{1}$ error      & Order  \\ 		
		\hline  $8$   &3.12e-06	&-	    &2.93e-06	&-	    &2.70e-06	&-	     \\
		\hline  $16$  &1.69e-07	&4.21	&1.67e-07	&4.13	&1.61e-07	&4.07	 \\
		\hline  $32$  &9.62e-09	&4.13	&9.68e-09	&4.11	&9.74e-09	&4.05	 \\
		\hline  $64$  &6.67e-10	&3.85	&6.75e-10	&3.84	&7.65e-10	&3.67	 \\
		\hline  $128$ &4.05e-11	&4.04	&4.12e-11	&4.03	&4.79e-11	&4.00    \\
		\hline
	\end{tabular}
	\label{tab:accuracy-test-1d-order4}
\end{table}

\subsubsection{Example 2: Isothermal equilibrium test} \label{example:isothermal-1d}

We consider the isothermal equilibrium  problem \cite{Zenk2017WBNDG} under a linear gravitational field $\phi_x = g  = 1$.
The initial data is taken as an isothermal steady state solution
\begin{equation}\label{eq:example-isothermal-1d}
	\rho(x) = \exp(-x) \,, \quad u(x) = 0 \,, \quad p(x) = \exp(-x) \,.
\end{equation}
The computational domain is {taken} as $\Omega = [0,1]$, {and} the adiabatic index is $\gamma = 5/3$.  We first use this example to verify the WB property of the proposed CDG method. The numerical solution is computed until $t = 2$ by using our third-order WB CDG scheme with respectively $50$ and $100$ uniform cells. {Table \ref{tab:isothermal-1d} lists the $L^{1}$ errors between} the numerical solution and the projected stationary hydrostatic solution (\ref{eq:example-isothermal-1d}). It is clearly observed that {all} the numerical errors are  at the level of rounding error, demonstrating {that} the proposed CDG method satisfies the WB property.

\begin{table}[htbp]
	\centering
	\caption{$L^{1}$ errors for the isothermal equilibrium test in Example 2.}
	\begin{tabular}{c|c|c|c}
		\hline
		Mesh	       & errors in  $\rho$   & errors in $\rho u$ & errors in $E$ \\
		\hline
		50         & 7.71e-15 & 1.97e-15 & 4.00e-15   \\
		\hline
		100        & 1.63e-14 & 4.50e-15 & 7.27e-15  \\
		\hline
	\end{tabular}  \label{tab:isothermal-1d}
\end{table}

Next, in order to check the effectiveness of our WB CDG method in capturing a small perturbation near the isothermal equilibrium solution (\ref{eq:example-isothermal-1d}), we modify the initial pressure state 
{as}
\[
p(x,0) = \exp(-x) + \eta \exp(-100(x-0.5)^2) \, ,
\]
where $\eta$ is a non-zero perturbation parameter.
We simulate two cases: $\eta = 10^{-2}$ and $\eta = 10^{-3}$.
Outflow boundary conditions are used at $x = 0$ and $x = 1$. The pressure perturbations at $t=0.25$ computed by our third-order WB CDG scheme on the mesh of $50$ uniform cells, compared with a reference solution with $1000$ cells, are displayed in Figure \ref{fig:isothermal-1d-pressure}. For comparison, we also present the results simulated by the third-order non-WB CDG scheme in the same figure.
The initial pressure perturbations are also plotted in {the} dashed curves.
As we can see {that} the results of the WB CDG scheme agree well with the reference solutions for both cases, while the results obtained by the non-WB CDG scheme fail to capture the small perturbations. This demonstrates that our WB method is more accurate and advantageous in resolving small perturbations near {the} equilibrium states.

\begin{figure}[htbp]
	\centering
	\subfigure[{$\eta = 10^{-2}$}]{
		\includegraphics[width=0.475\textwidth,height=0.375\textwidth]
		{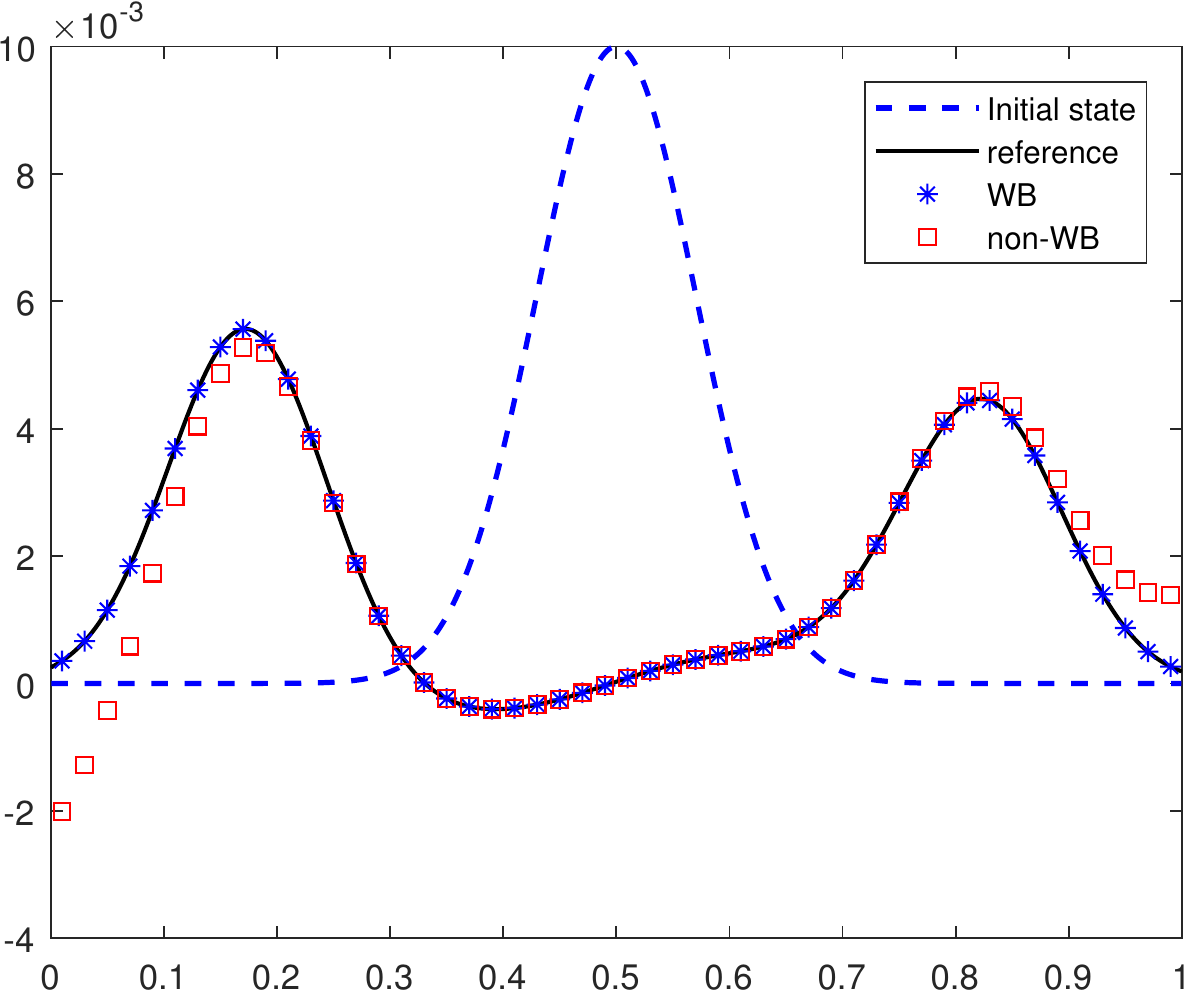}
	}
	\hfill
	\subfigure[{$\eta = 10^{-3}$}]{
		\includegraphics[width=0.475\textwidth,height=0.375\textwidth]
		{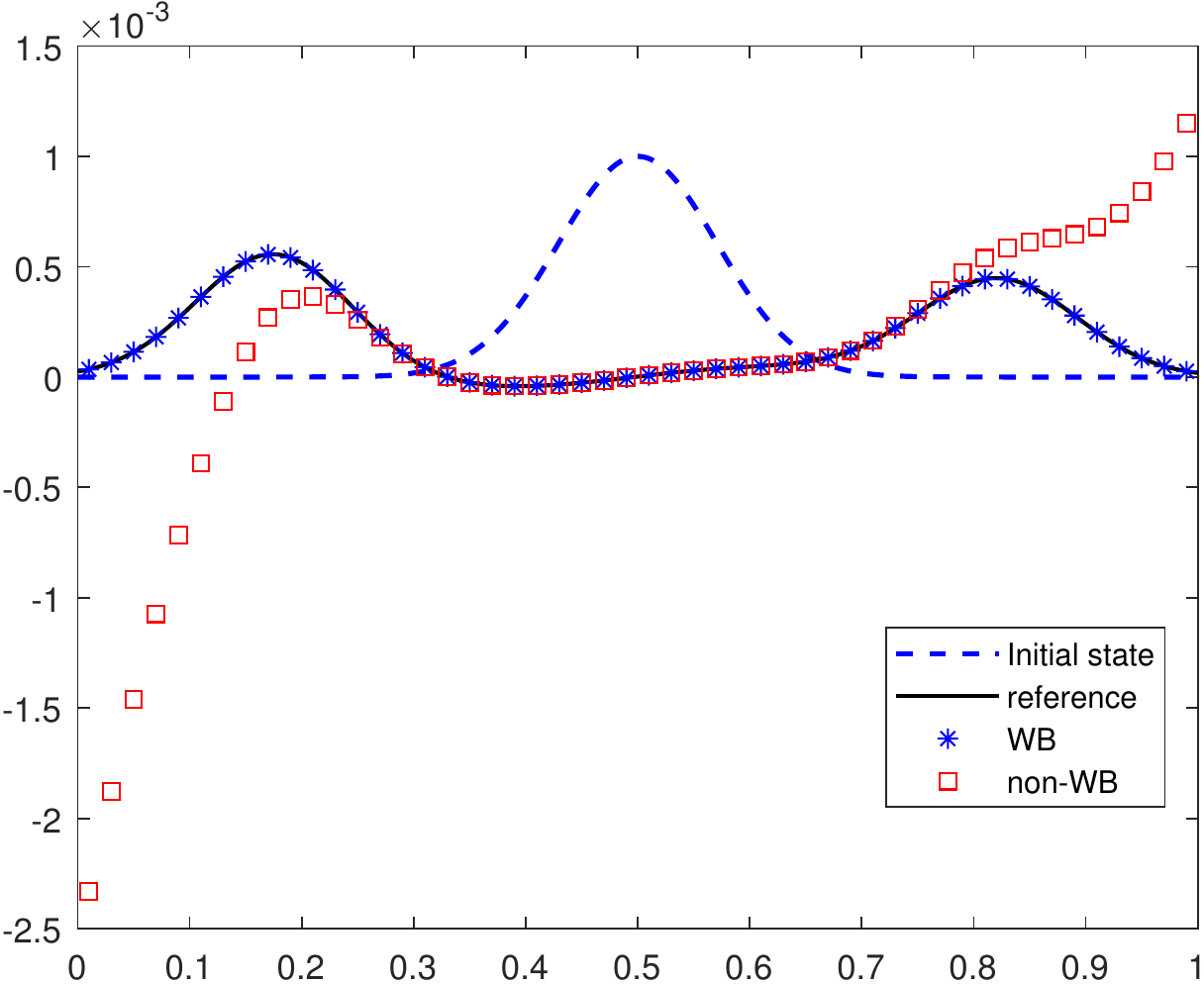}
	}
	\caption{Example $2$: {Pressure perturbations at $t=0.25$ obtained by the} WB scheme and non-WB scheme with $50$  uniform cells. The reference solutions are obtained by the WB scheme using $1000$  cells.}
	\label{fig:isothermal-1d-pressure}
\end{figure}

\subsubsection{Example 3: Rarefaction test with {low density and pressure} }
The purpose of this example is to investigate the positivity-preserving property of our WB CDG method. We consider an extreme rarefaction problem \cite{wu2021uniformly} under a quadratic gravitational potential function $\phi(x) = x^2/2$, and the initial data {are} given by
\begin{equation*}
	\rho(x,0) = 7 , \quad p(x,0) = 0.2 , \quad
	u(x,0) =
	\left\{
	\begin{aligned}
		-1  \,, \quad  x < 0  \,,
		\\  1  \,, \quad  x > 0  \,.
	\end{aligned}
	\right.
\end{equation*}
The computational domain is set as $\Omega = [-1, 1]$ with outflow boundary conditions  at $x = -1$ and $x = 1$.
This test involves extremely {low density and pressure}, so that the positivity-preserving limiting operators are used
in our simulation. 
%The CFL number is taken as $0.08$, which is slightly smaller than $\frac{\hat{\omega}_{1}}{2}  = \frac{1}{12}$. 
{Figure \ref{fig:Rarefaction} shows the} numerical solutions at $t = 0.6$ computed by our third-order positivity-preserving  WB CDG scheme with
$400$ uniform cells, along with a reference solution with $1000$ cells.
As we can see {that} the low density and pressure wave structures are well captured by the proposed method. Our numerical scheme exhibits good robustness, and no negative density or pressure is encountered during the entire simulation.

\begin{figure}[htbp]
	\centering
	\subfigure[{Density}]{
		\includegraphics[width=0.30\textwidth,height=0.25\textwidth]
		{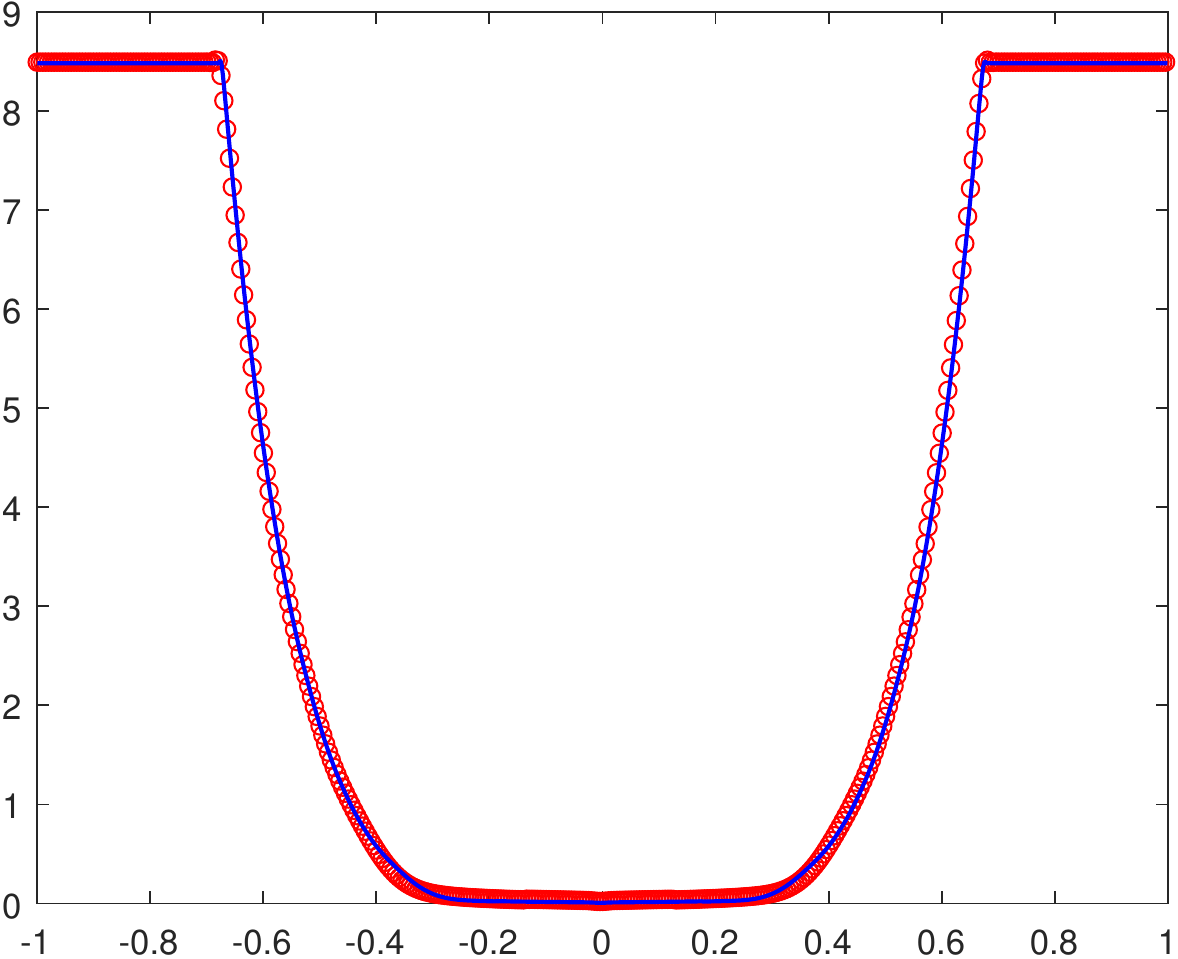}
	}
	\hfill
	\subfigure[{Momentum}]{
		\includegraphics[width=0.30\textwidth,height=0.25\textwidth]
		{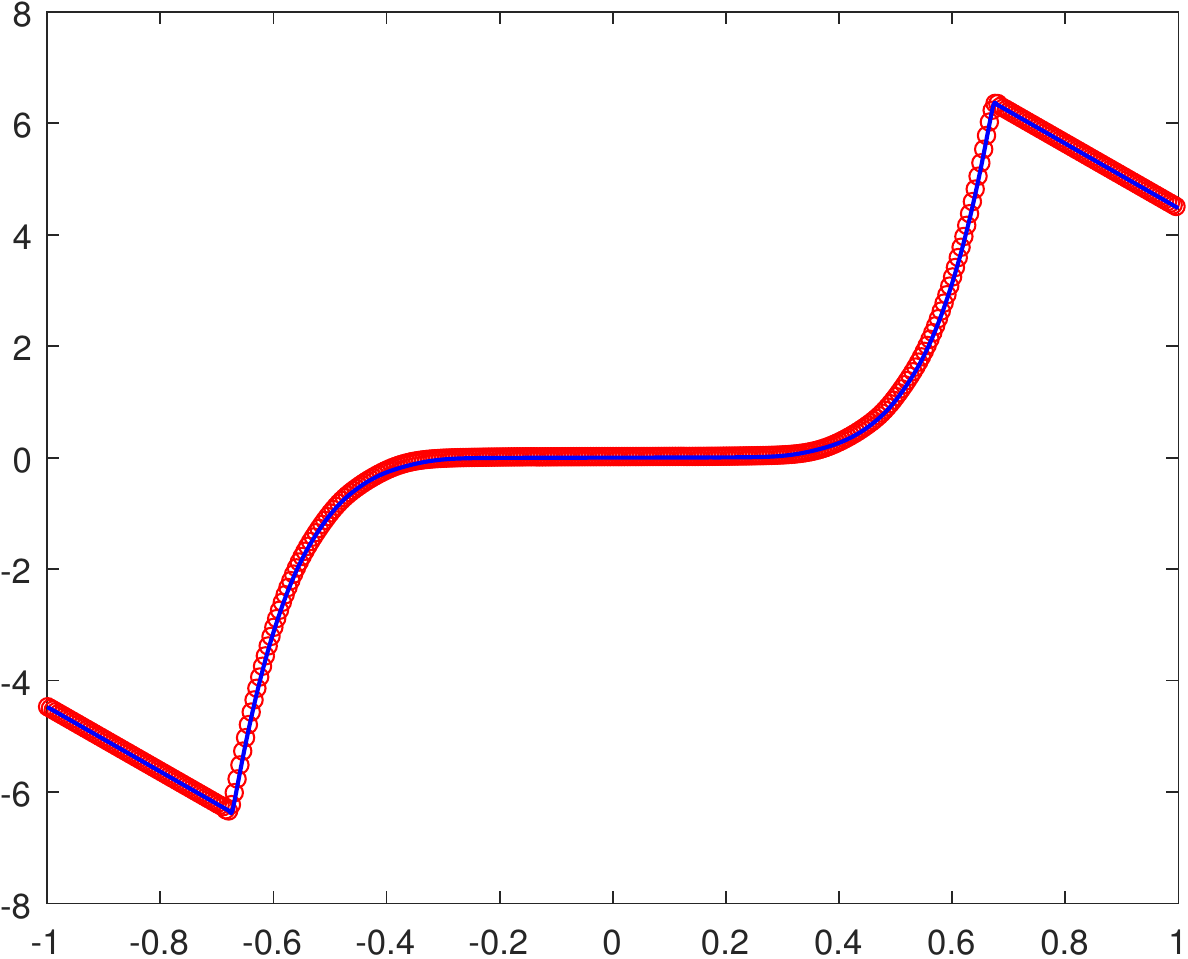}
	}
	\hfill
	\subfigure[{Energy}]{
		\includegraphics[width=0.30\textwidth,height=0.25\textwidth]
		{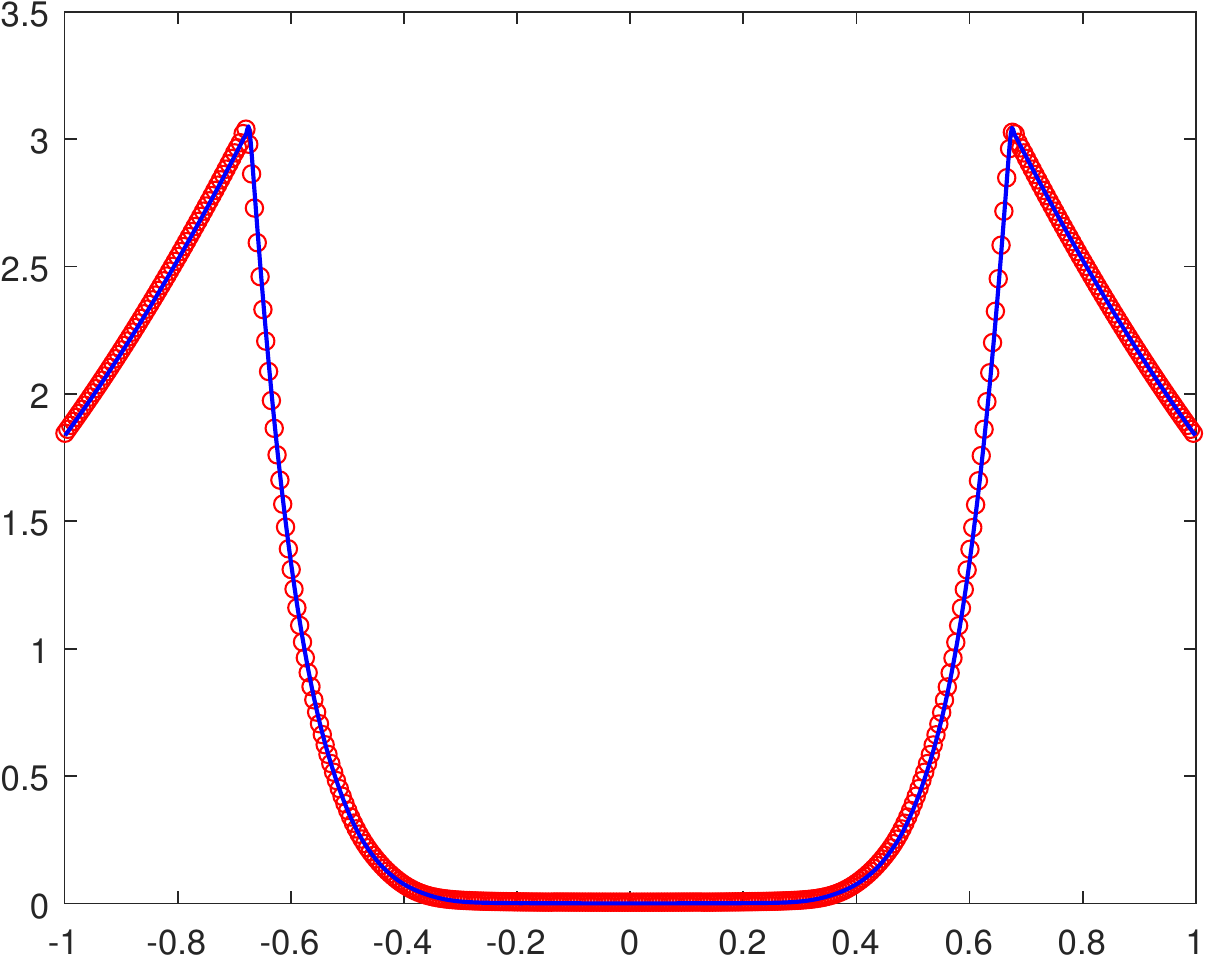}
	}
	\caption{Example $3$: {Numerical} results for the rarefaction test at $t = 0.6$ obtained by the  third-order positivity-preserving WB CDG scheme with $400$ cells (circles) and $1000$ cells (solid lines).}
	\label{fig:Rarefaction}
\end{figure}

\subsubsection{Example 4: Leblanc shock tube problem under gravitational field}

{This example considers} a Leblanc shock tube problem \cite{wu2021uniformly} under a linear gravitational field
$\phi(x) = gx$ with $g = 1$. The initial condition is given by
\begin{equation*}
	{(\rho, u, p)(x,0)=
	\begin{cases}
		  (2, ~ 0, ~ 10^9),    &  x < 0\,,
		\\  (10^{-3}, ~ 0, ~1), &  x > 0  \,.
	\end{cases}}%
\end{equation*}
The computational domain is taken as $\Omega = [-10,10]$ with outflow boundary condition at $x = -10$ and
$x = 10$. %The CFL number is set as $0.08$ in the computation. 
As the initial data {contain a strong discontinuity in the density and the pressure}, the WB implementation (see Remark \ref{rem:WB-WENO}) of the WENO limiter \cite{qiu2005runge} is applied 
to the local characteristic fields in some troubled cells, before the positivity-preserving limiting procedure in our simulation. The troubled cells are  adaptively identified 
with parameters $(M_{1}, M_{2}, M_{3}) = (10^{10}, 2\times10^6, 5\times10^{10})$, where $M_i$ denotes the parameter $M$ in \eqref{minmod-corrected} for the $i$-th component of ${\bf U}$. %In each component, one may use a different TVB parameter $M$ in \eqref{minmod-corrected} (see, for example, \cite{Zhang2010PP}),  and in this case we employ  to distinguish the $M$ in \eqref{minmod-corrected} for the $i$-th component. 
{Figure \ref{fig:Leblanc-shock} displays the density, the velocity, and the pressure} at $t=0.0001$ computed by our third-order positivity-preserving WB CDG scheme on a mesh with $800$ cells, compared with a reference solution with a refined mesh of $1600$ cells.  It is seen that our positivity-preserving scheme is highly robust, and the strong discontinuities are captured with high resolution.

\begin{figure}[htbp]
	\centering
	\subfigure[$\log \rho$]{
		\includegraphics[width=0.30\textwidth,height=0.25\textwidth]
		{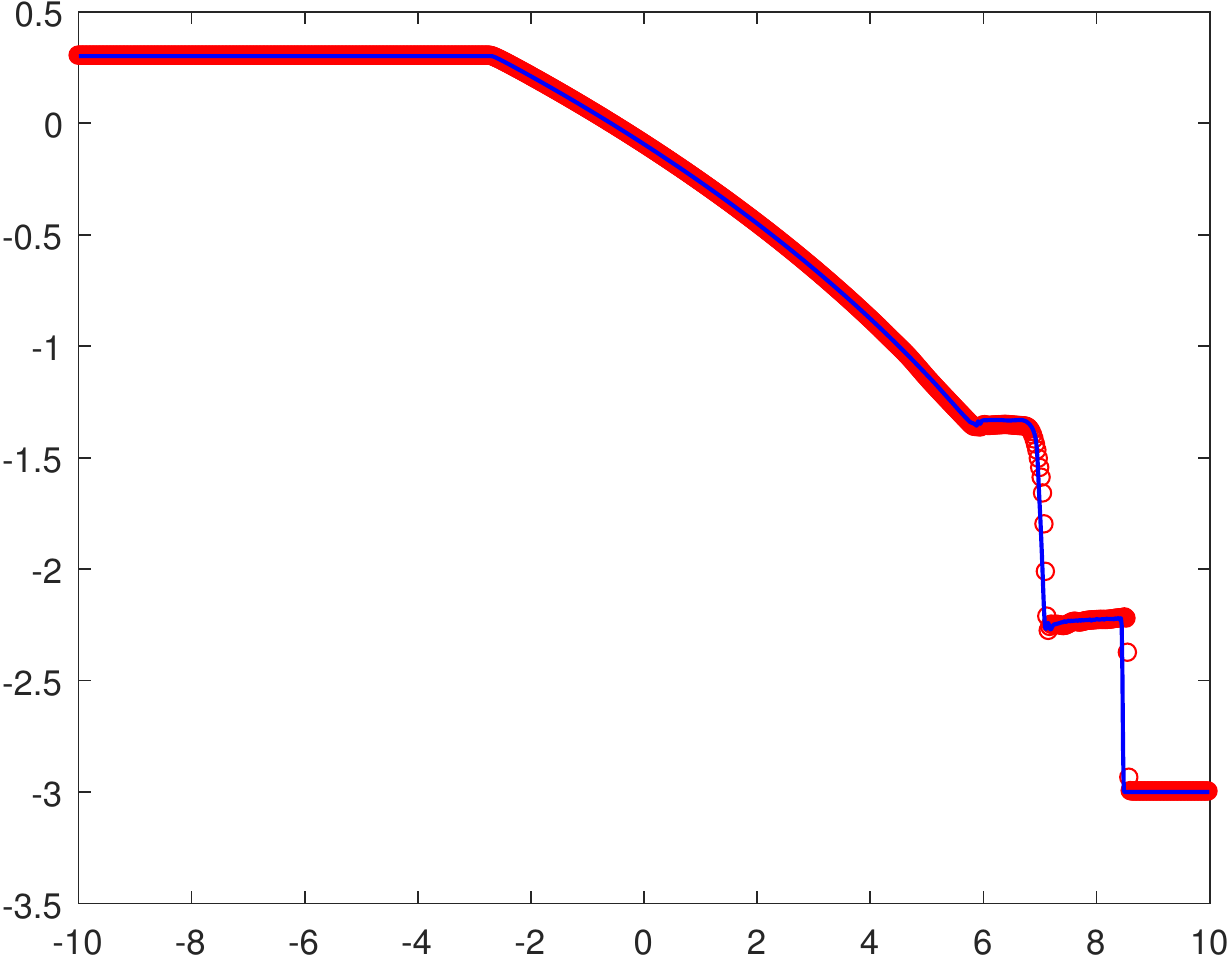}
	}
	\hfill
	\subfigure[$u$]{
		\includegraphics[width=0.30\textwidth,height=0.26\textwidth]
		{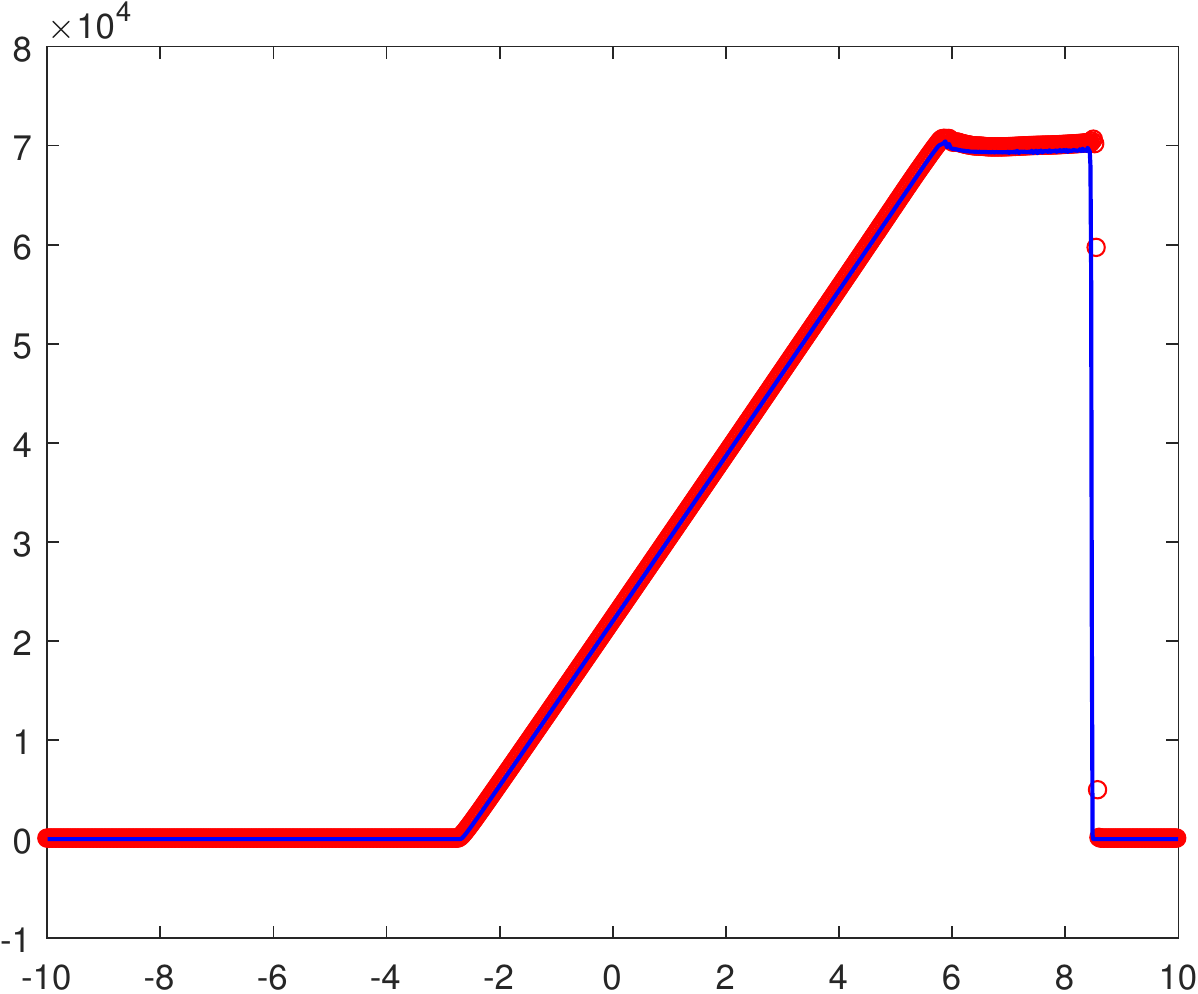}
	}
	\hfill
	\subfigure[$\log p$]{
		\includegraphics[width=0.30\textwidth,height=0.25\textwidth]
		{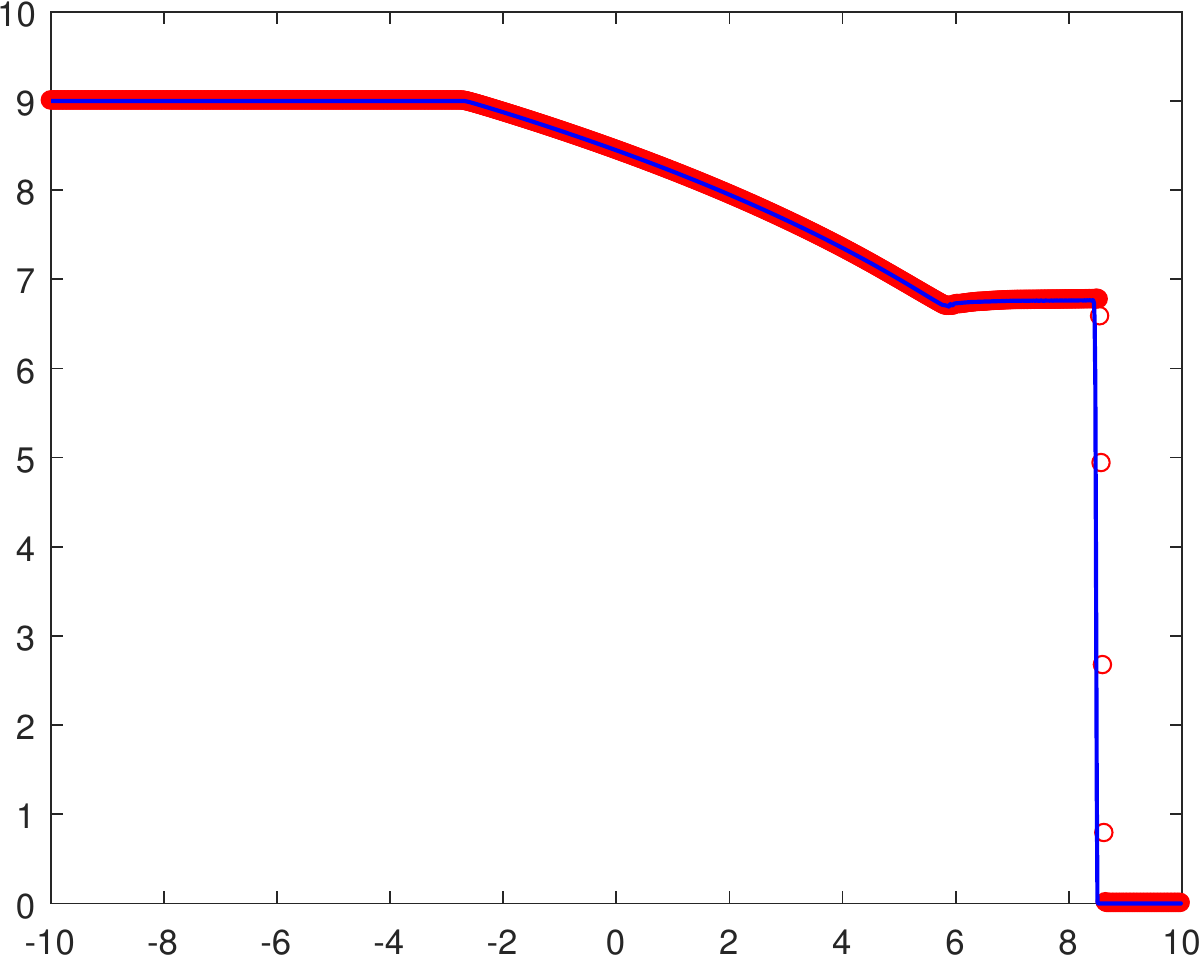}
	}
	\caption{ Example $4$: {Numerical} results for the Leblanc shock tube problem at $t = 0.0001$ obtained by the third-order positivity-preserving WB CDG scheme with $800$ cells (circles) and $1600$ cells (solid lines). }
	\label{fig:Leblanc-shock}
\end{figure}

\subsection{Two-dimensional tests} 

\subsubsection{Example 5: Accuracy test} \label{example:accuracy-test-2d}

{This test checks} the convergence rates of the WB CDG schemes for the two-dimensional Euler equations
under a linear gravitational field $\phi_x = \phi_y = 1$. A time-dependent exact solution \cite{Xing2013FDWB}
takes the form of
\begin{align} \notag
	&\rho(x,y,t)  = 1 + 0.2\sin(\pi(x + y - (u_{0} + v_{0})t )),  \\ \notag
	&u_{1}(x,y,t)  = u_{0}, \quad u_{2}(x,y,t)  = v_{0},                 \\ \notag
	&p(x,y,t)  = p_{0} + (u_{0} + v_{0}) t -x -y + 0.2 \cos(\pi(x+y - (u_{0} + v_{0})t ))/\pi,
\end{align}
where the parameters $u_{0} = v_{0} = 1$ and $p_{0} = 4.5$. The computational domain
$\Omega = [0, 2]\times[0, 2]$ is divided into $N \times N$ uniform cells, and the boundary condition
is set as the exact solution on $\partial \Omega$. The $L^{1}$ errors and {corresponding convergence rates} at $t = 0.1$ are
displayed in Tables \ref{tab:accuracy-test-2d-order3}  and \ref{tab:accuracy-test-2d-order4} . It is seen that
{the theoretical} convergence rates are achieved by our WB CDG schemes, as expected. Our novel projection, modification of the dissipation term, and  WB source term approximation do not affect the accuracy of the CDG method.

\begin{table}[htbp]
	\centering
	\caption
	{Example 5: $L^{1}$ errors at $t = 0.1$ and corresponding convergence rates
	  for the third-order WB CDG scheme at different grid resolutions. }
	\begin{tabular}{c|c|c|c|c|c|c|c|c}
		\hline  $N \times N$    &\multicolumn{2}{|c|}{$\rho$}         &\multicolumn{2}{|c|}{$\rho u_{1}$ }
		&\multicolumn{2}{|c|}{$\rho u_{2}$ }  &\multicolumn{2}{|c}{$E$ }  \\
		\hline   Mesh       &$L^{1}$ error & Order & $L^{1}$ error & Order  &$L^{1}$ error & Order &$L^{1}$ error & Order \\
		\hline  $8\times8$      &7.18e-04	&  -	&7.09e-04	&  -	&7.09e-04	&  -	&8.99e-04	& -	    \\
		\hline  $16\times16$	&8.53e-05	&3.07	&8.48e-05	&3.06	&8.48e-05	&3.06	&1.08e-04	&3.06	\\
		\hline  $32\times32$	&1.05e-05	&3.02	&1.05e-05	&3.01	&1.05e-05	&3.01	&1.34e-05	&3.01	\\
		\hline  $64\times64$	&1.31e-06	&3.00	&1.30e-06	&3.01	&1.30e-06	&3.01	&1.67e-06	&3.00	\\
		\hline  $128\times128$	&1.63e-07	&3.01	&1.63e-07	&3.00	&1.63e-07	&3.00	&2.09e-07	&3.00	\\		
		\hline
	\end{tabular}
	\label{tab:accuracy-test-2d-order3}
\end{table}

\begin{table}[htbp]
	\centering
	\caption
	{Same as Table \ref{tab:accuracy-test-2d-order3}, 
except for the fourth-order WB CDG scheme. }
	\begin{tabular}{c|c|c|c|c|c|c|c|c}
		\hline $N \times N$   &\multicolumn{2}{|c|}{$\rho$}         &\multicolumn{2}{|c|}{$\rho u_{1}$ }
		&\multicolumn{2}{|c|}{$\rho u_{2}$ }  &\multicolumn{2}{|c}{$E$ }  \\
		\hline     Mesh     &$L^{1}$ error & Order & $L^{1}$ error & Order  &$L^{1}$ error & Order &$L^{1}$ error & Order \\
		\hline  $8\times8$	    &9.65e-05	& -	    &9.35e-05	& -	    &9.35e-05	& -	    &1.16e-04	& -	   \\
		\hline  $16\times16$	&5.51e-06	&4.13	&5.43e-06	&4.11	&5.43e-06	&4.11	&6.87e-06	&4.08  \\
		\hline  $32\times32$	&3.33e-07	&4.05	&3.30e-07	&4.04	&3.30e-07	&4.04	&4.22e-07	&4.02  \\	
		\hline  $64\times64$	&2.06e-08	&4.01	&2.05e-08	&4.01	&2.05e-08	&4.01	&2.62e-08	&4.01  \\	
		\hline  $128\times128$	&1.29e-09	&4.00	&1.29e-09	&3.99	&1.29e-09	&3.99	&1.65e-09	&3.99  \\		
		\hline
	\end{tabular}
	\label{tab:accuracy-test-2d-order4}
\end{table}

\subsubsection{Example 6: Isothermal equilibrium solution}\label{example:isothermal-2d}

{This example considers} a two-dimensional isothermal equilibrium problem \cite{Xing2013FDWB} under the linear
gravitational field $\phi_{x} = \phi_{y} = g $. The initial condition 
{is specified as}
\begin{align} \label{eq:example-isothermal-2d}
\begin{aligned}
	&\rho(x,y) = \rho_{0} \exp\Big(-\frac{\rho_{0}g}{p_{0}}(x+y)\Big) \,,  \\  
	&u_{1}(x,y) = u_{2}(x,y) = 0 \,,  \\  
	&p(x,y) = p_{0} \exp\Big(-\frac{\rho_{0}g}{p_{0}}(x+y)\Big) \,,
\end{aligned}\end{align}
where the parameters $\rho_{0} = 1.21$, $p_{0} = 1$ and $g = 1$. The computational domain is a
unit square $\Omega = [0,1]\times[0,1]$. 
We first use this test to demonstrate the WB property of the proposed CDG method. The numerical results at $t = 1$ are obtained by using our third-order WB CDG scheme on two different uniform meshes. {Table \ref{tab:isothermal-2d} lists the $L^1$ errors   between  the numerical solution and
the  isothermal equilibrium solution \eqref{eq:example-isothermal-2d}}. It is clearly observed that {all} the numerical errors are at the level of machine precision, demonstrating
{that} the proposed CDG method satisfies the WB property in two-dimensional case.
\begin{table}[htbp]
	\centering
	\caption{ $L^{1}$ errors for the isothermal equilibrium solution in Example 6. }
	\begin{tabular}{c|c|c|c|c}
		\hline
		Mesh      & errors in $\rho$   & errors in $\rho u_{1}$  & errors in $\rho u_{2}$ & errors in $E$ \\
		\hline
		50$\times$50        & 2.26e-15 & 8.43e-16 & 8.44e-16  & 4.07e-15  \\
		\hline
		80$\times$80        & 3.96e-15 & 1.40e-15 & 1.38e-15  & 6.49e-15 \\
		\hline
	\end{tabular}  \label{tab:isothermal-2d}
\end{table}

Next, we investigate the effectiveness of our WB CDG method in capturing the propagation of {a} small wave perturbation
around the isothermal equilibrium solution. Initially, a small Gaussian perturbation is imposed in the pressure
state as follows
\[
p(x,y) = p_{0} \exp\Big(-\frac{\rho_{0}g}{p_{0}}(x+y)\Big) +
\eta \exp \Big(-100\frac{\rho_{0}g}{p_{0}} \big( (x-0.3)^2+(y-0.3)^2 \big)\Big),
\]
where parameter $\eta = 10^{-3}$, and the density and {the} velocities are given by the equilibrium
state (\ref{eq:example-isothermal-2d}). We evolve the numerical solution until $t = 0.15$ on the mesh of
$50\times50$ cells, with {the} transmissive boundary conditions {specified on} $\partial\Omega$.  The contour plots of the density  and  pressure perturbations obtained {by} the third-order WB and non-WB CDG schemes are {shown} in Figure \ref{fig:isothermal-2d}. One can see that the non-WB CDG scheme cannot capture {those} small perturbations well on the relatively coarse mesh, while the WB method resolves them  accurately.

\begin{figure}[htbp]
	\centering
	\subfigure[{Pressure perturbation of non-WB method}]{
		\includegraphics[width=0.47\textwidth]
		{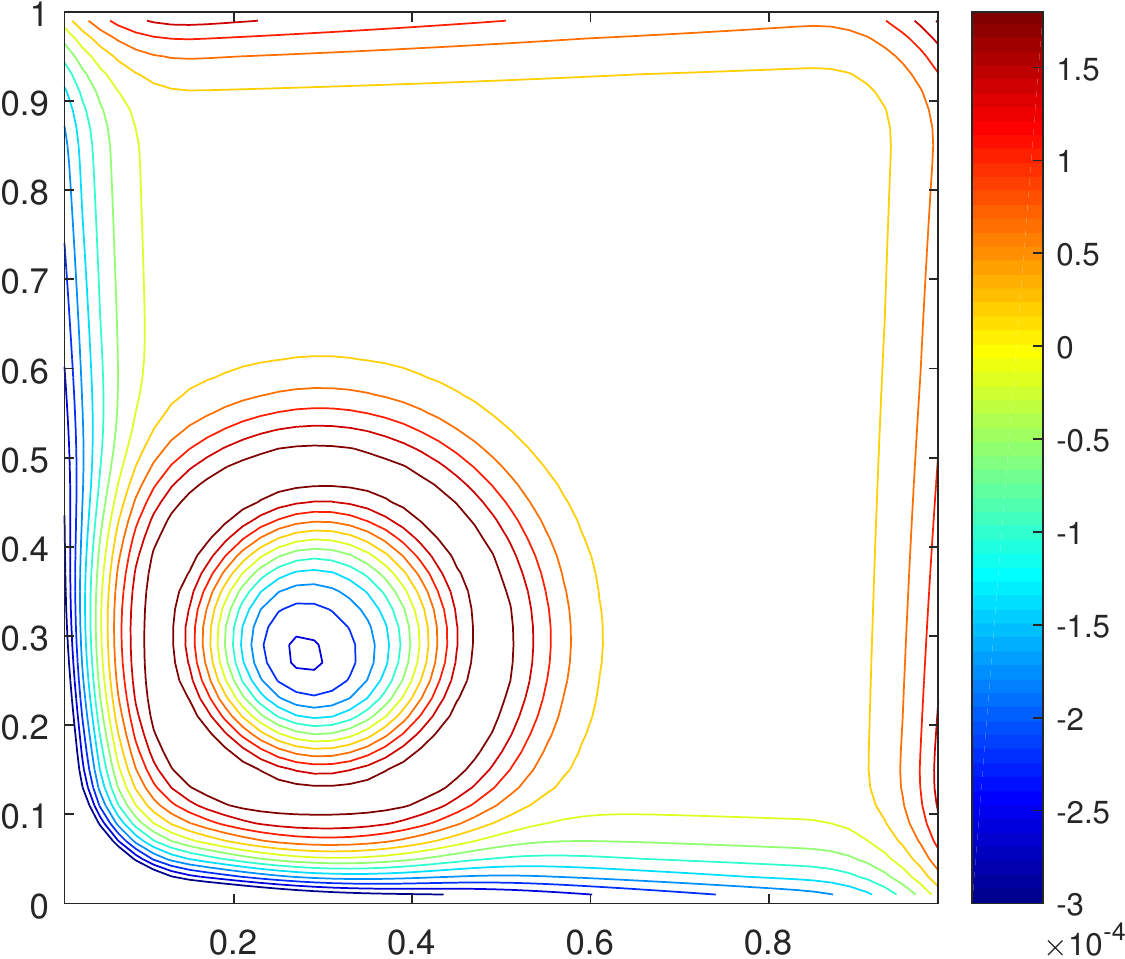}
	}
	\hfill
	\subfigure[{Density perturbation of non-WB method}]{
		\includegraphics[width=0.47\textwidth]
		{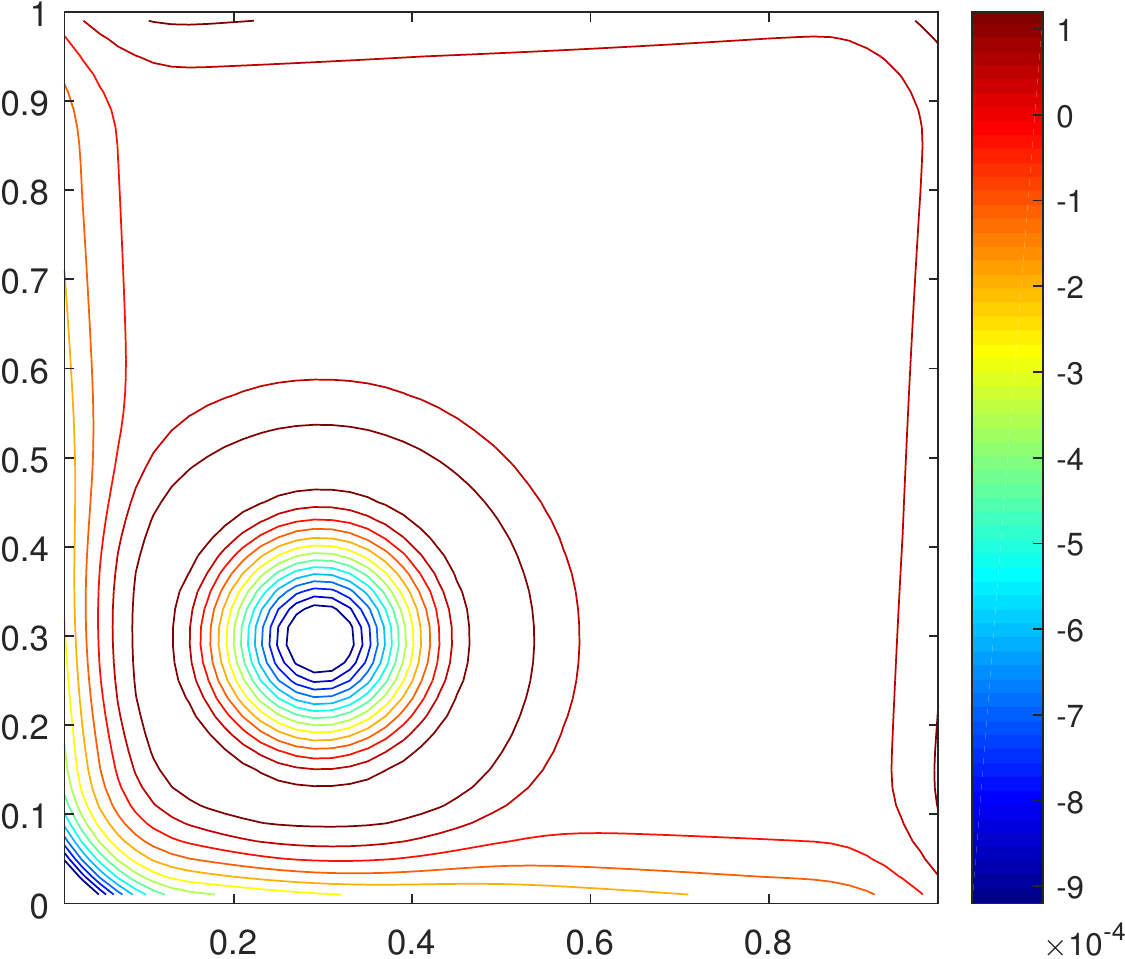}
	}
	\hfill
	\subfigure[{Pressure perturbation of WB method}]{
		\includegraphics[width=0.47\textwidth]
		{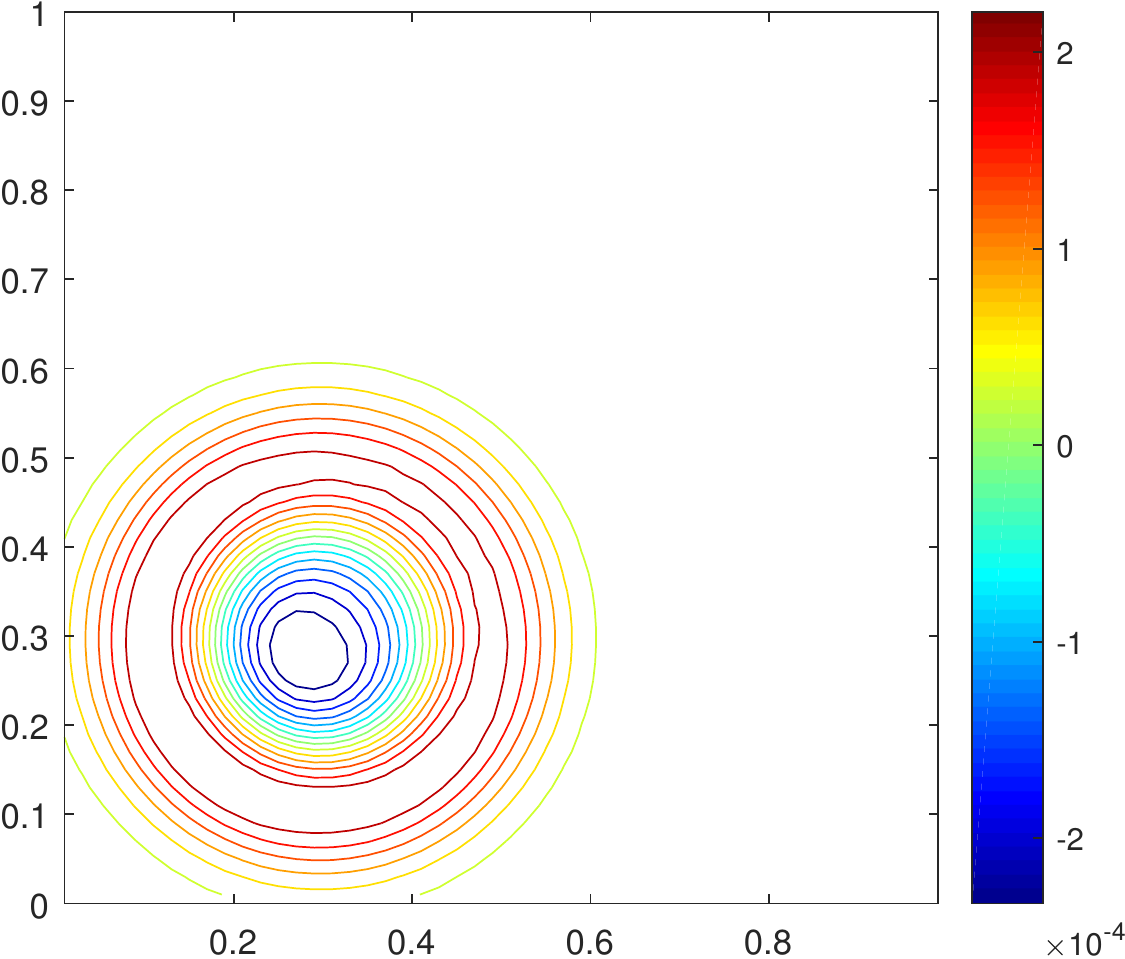}
	}
	\hfill
	\subfigure[{Density perturbation of WB method}]{
		\includegraphics[width=0.47\textwidth]
		{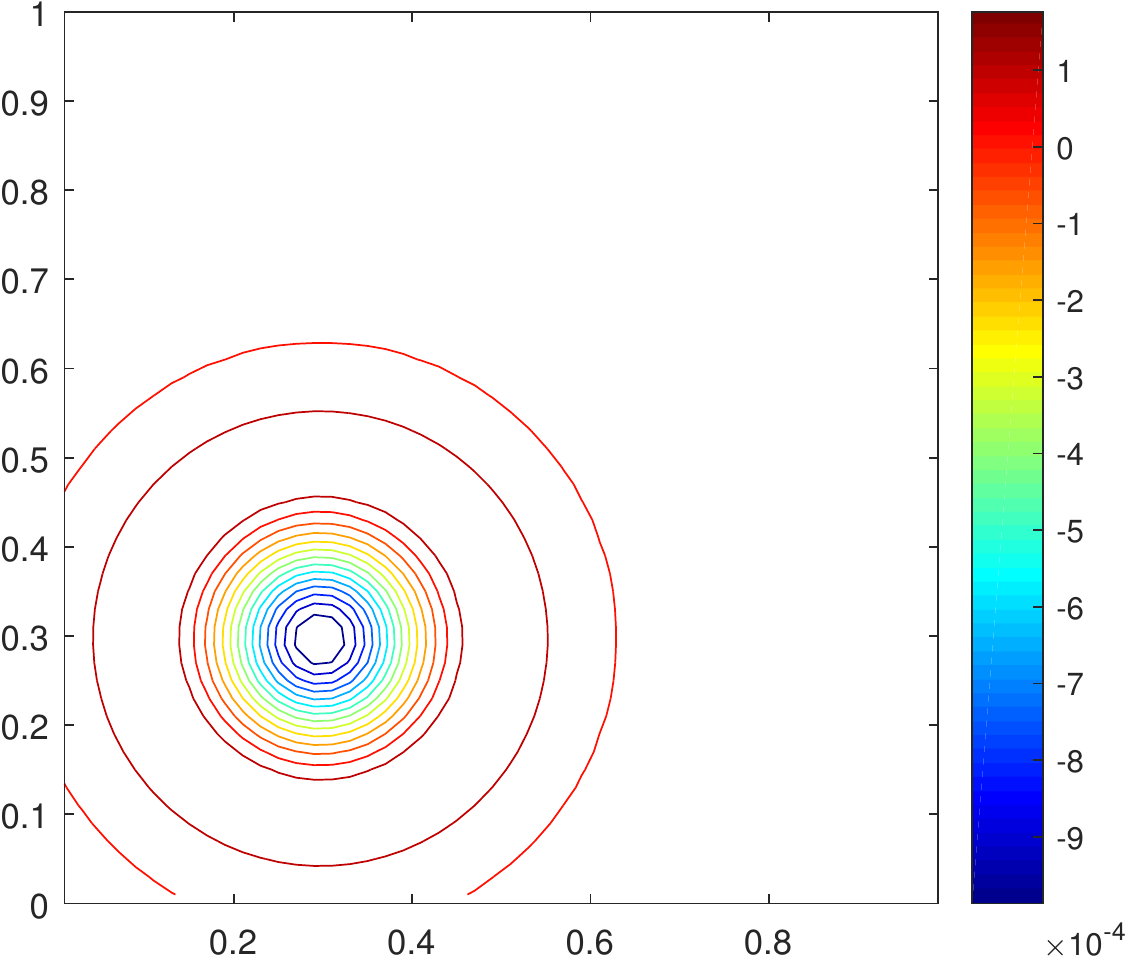}
	}
	\caption{Example 6: {Contour} plots of the pressure and density perturbations for the two-dimensional isothermal equilibrium problem at $t = 0.15$ obtained by the WB and non-WB CDG schemes with $50\times50$ uniform cells. }
	\label{fig:isothermal-2d}
\end{figure}

\subsubsection{ Example 7: {Polytropic} equilibrium solution }\label{example:polytropic-2d}

{This example considers} a two-dimensional polytropic problem arising from astrophysics \cite{Mishra2014well}.
This model can be constructed from the hydrostatic equilibrium in spherical symmetry case
\[
\frac{dp}{dr} = -\rho \frac{d\phi}{dr} \,,
\]
where $r := \sqrt{x^2 + y^2}$ denotes the radial variable, and the adiabatic index is $\gamma = 2$. One equilibrium solution of
this model is given by
\begin{equation}\label{eq:polytropic-radial}
	\rho (r) = \rho_{c}\frac{\sin(\alpha r)}{\alpha r} \,, \quad u_{1}(r) = 0\,, \quad u_{2}(r) = 0\,, \quad p(r) = K\rho(r)^2\,,
\end{equation}
coupled with a gravitational potential function
\begin{equation}\label{eq:polytropic-radial-potential}
	\phi(r) = -2K\rho_{c} \frac{\sin(\alpha r)}{\alpha r} \,,
\end{equation}
where $\alpha = \sqrt{\frac{2\pi g}{K}}$ and  $K = g = \rho_{c} = 1$. The computational domain is taken as $\Omega = [-0.5, 0.5]\times[-0.5, 0.5]$.

We first use this example to verify the WB property of the proposed CDG schemes. The initial
data {are} set as the equilibrium solution (\ref{eq:polytropic-radial}), and we perform the numerical simulations
up to $t = 14.8$ on two different uniform meshes. 
{Table \ref{tab:polytropic-2d} lists the $L^1$ errors between the numerical
solution and the projected  equilibrium solution \eqref{eq:polytropic-radial}}.
It is observed that all the numerical errors are at the level of rounding error, which confirms that the proposed CDG method is WB.
\begin{table}[h]
	\centering
	\caption{ $L^{1}$ errors for the polytropic equilibrium solution in Example 7. }
	\begin{tabular}{c|c|c|c|c}
		\hline
		Mesh      & errors in $\rho$   & errors in $\rho u_{1}$  & errors in $\rho u_{2}$ & errors in $E$ \\
		\hline
		50$\times$50        & 1.31e-13 & 1.48e-14 & 1.55e-14  & 3.68e-14  \\
		\hline
		80$\times$80        & 2.25e-13 & 2.02e-14 & 2.03e-14  & 6.39e-14  \\
		\hline
	\end{tabular}  \label{tab:polytropic-2d}
\end{table}

In order to investigate the capability of our WB CDG method in capturing small perturbations near the polytropic equilibrium solution, we impose a small Gaussian hump perturbation in pressure as follows
\[
p(r) = K\rho(r)^2 + \eta \exp(-100r^2) \,,
\]
where the parameter $\eta = 10^{-3}$. We perform the numerical simulation up to $t = 0.2$ on the mesh of $50\times50$ cells with outflow boundary conditions specified on $\partial\Omega$. The contour plots of the pressure perturbation and the velocity magnitude $\sqrt{u^2+v^2}$ are displayed in Figure \ref{fig:polytropic-2d}. {The results show} that the
non-WB scheme are not capable of capturing {the} small perturbations on the relatively coarse mesh, while our WB scheme can resolve them accurately. In addition, the WB scheme is able to preserve the axial symmetry, but the non-WB scheme cannot well maintain the symmetry.
\begin{figure}[htbp]
	\centering
	\subfigure[{Velocity magnitude of non-WB scheme}]{
		\includegraphics[width=0.47\textwidth]
		{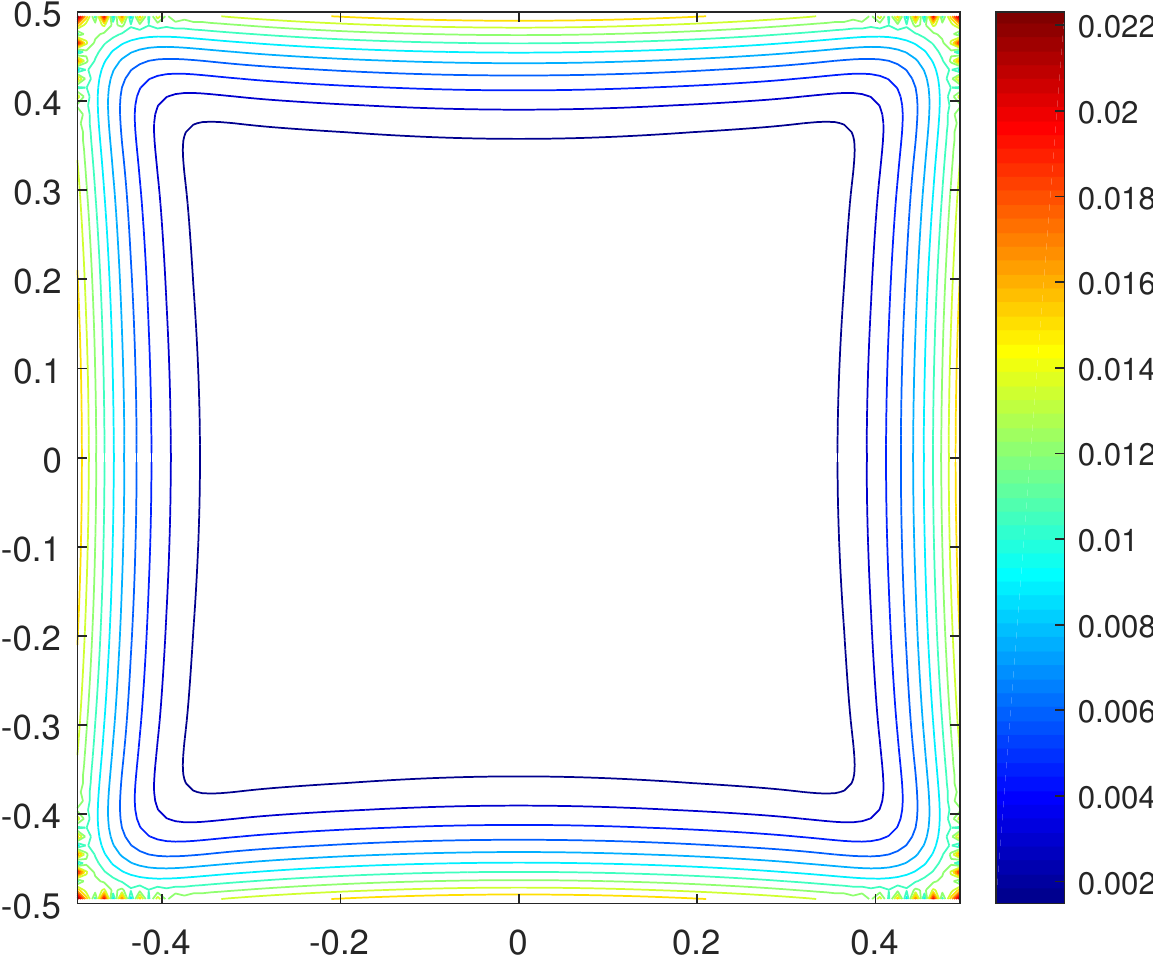}
	}
	\hfill
	\subfigure[{Pressure perturbation of non-WB scheme}]{
		\includegraphics[width=0.47\textwidth]
		{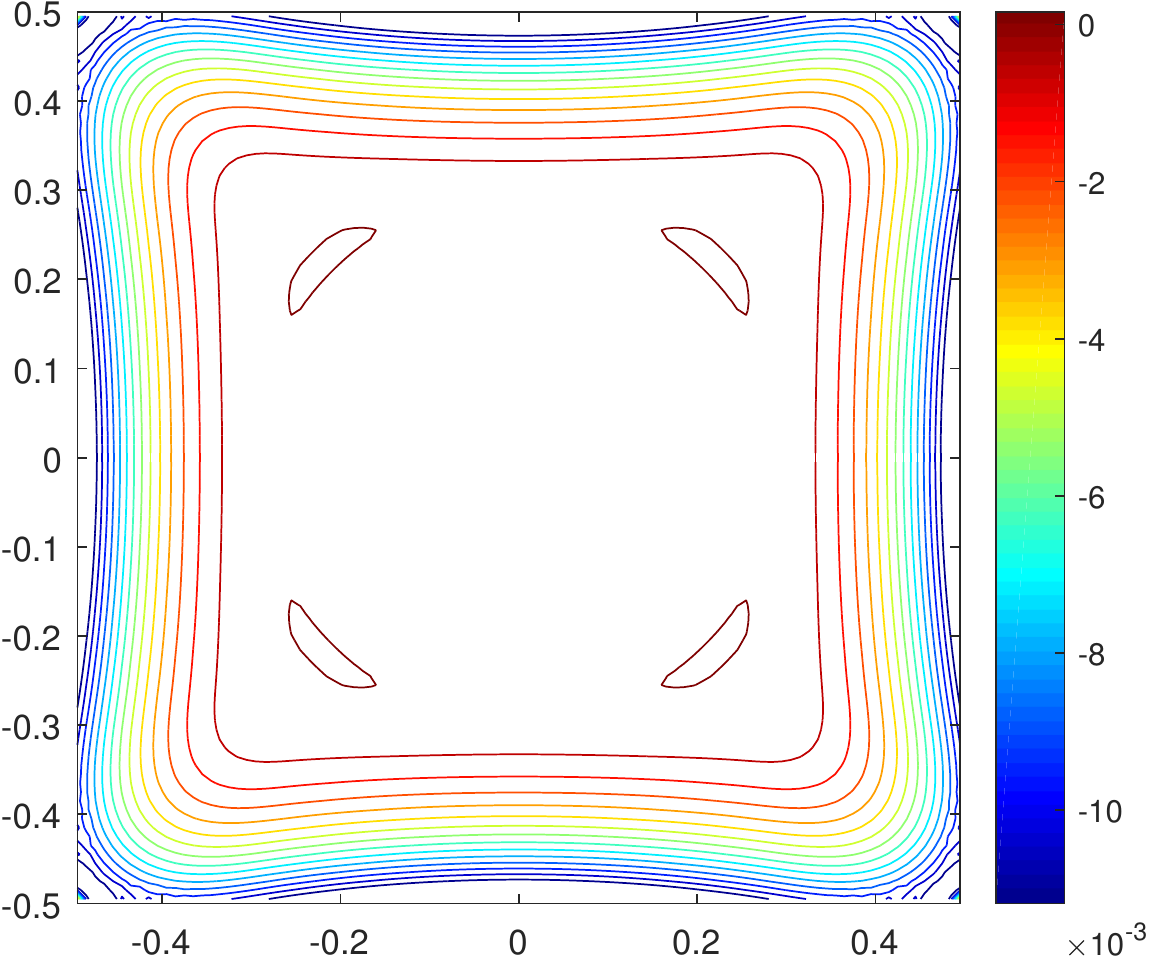}
	}
	\hfill
	\subfigure[{Velocity magnitude of WB scheme}]{
		\includegraphics[width=0.47\textwidth]
		{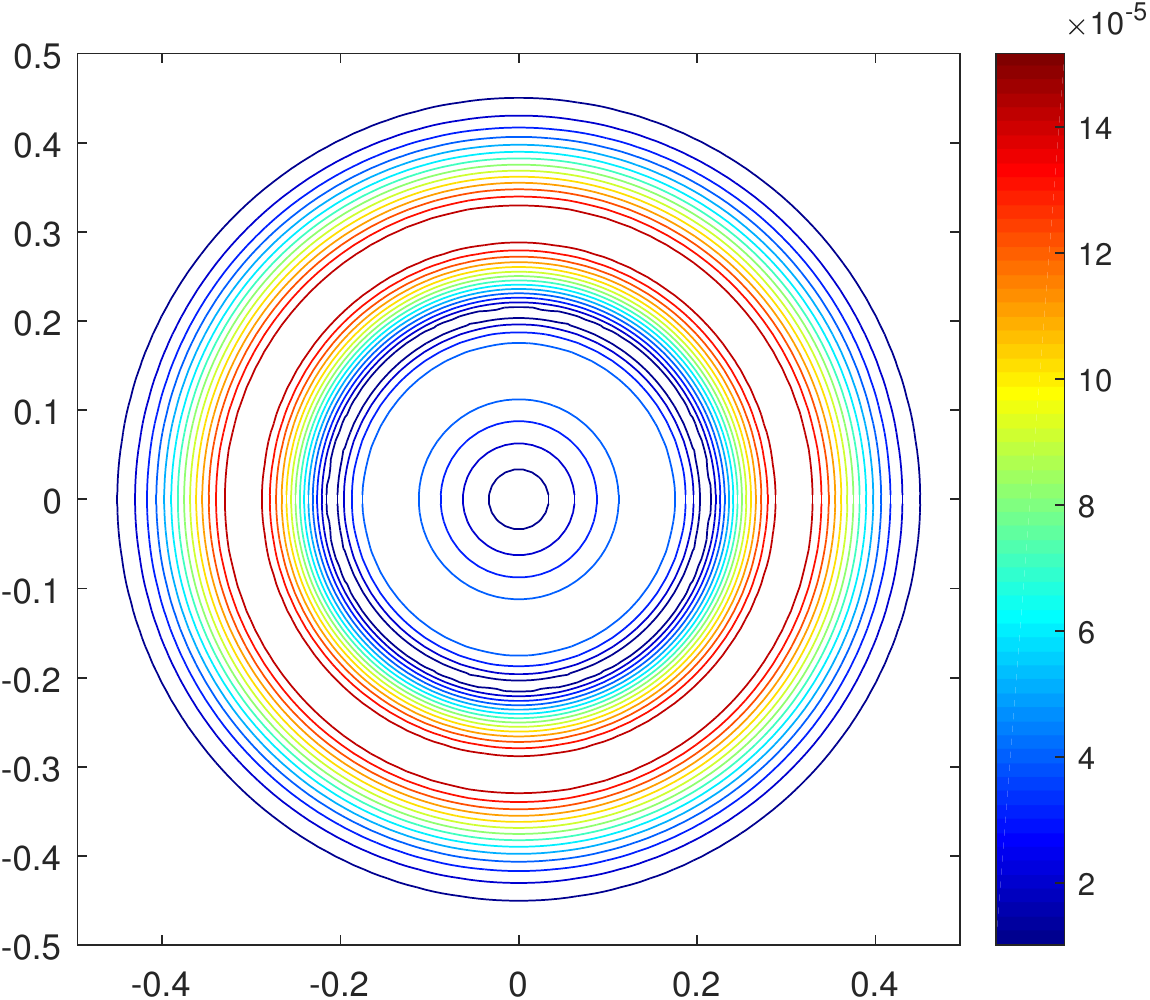}
	}
	\hfill
	\subfigure[{Pressure perturbation of WB scheme}]{
		\includegraphics[width=0.47\textwidth]
		{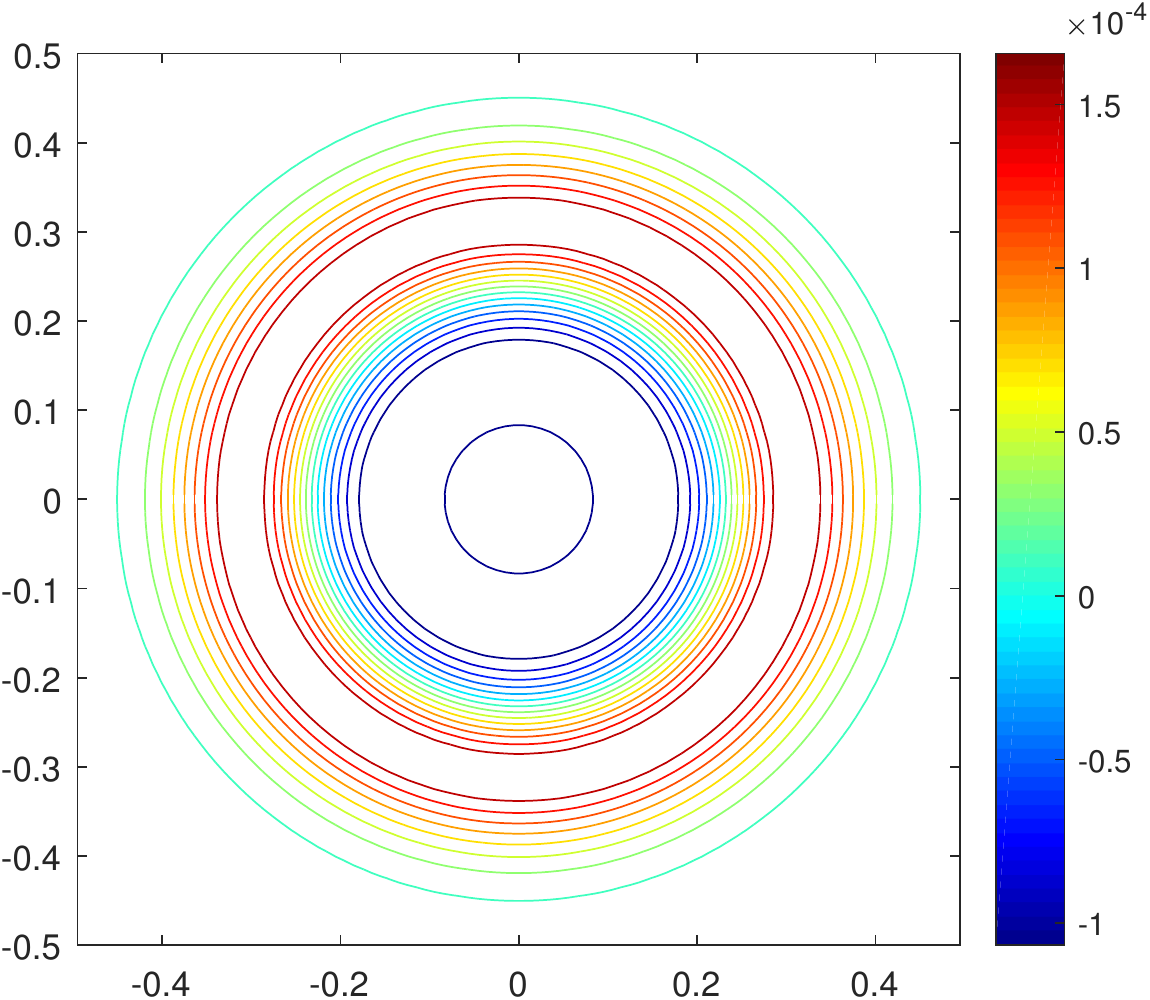}
	}
	\caption{Example 7: {Contour} plots of the velocity magnitude $\sqrt{u_{1}^2 + u_{2}^2}$ and pressure perturbation for the two-dimensional polytropic equilibrium problem at $t = 0.2$ obtained by the third-order WB and non-WB CDG schemes with $100\times100$ uniform cells. $10$ uniformly spaced contour lines are displayed. }
	\label{fig:polytropic-2d}
\end{figure}

\subsubsection{Example 8: Rarefaction test with low density and pressure}
This example is used to demonstrate the positivity-preserving property of the proposed CDG schemes.
The initial condition is given by
\begin{align*}
	& \rho(x,y,0)  = \exp(-\phi(x,y)/0.4) \,, \quad p(x,y,0) = 0.4 \exp(-\phi(x,y)/0.4) \,,   \\ \notag
	& u_{1}(x,y,0) =  \left\{
	\begin{aligned}
		~  -2  \,, \quad  x < 0.5  \,,  \\
		~ ~ 2  \,, \quad  x > 0.5  \,,
	\end{aligned} \right. \quad  u_{2}(x,y,0) = 0 \,,
\end{align*}
with a quadratic gravitational potential $\phi(x,y) = \frac{1}{2} \big[(x-0.5)^2 + (y-0.5)^2 \big]$.
The computational domain $\Omega = [0,1]\times[0,1]$ is divided into $100\times100$ uniform cells with outflow
boundary conditions {on} $\partial\Omega$. 
%The CFL number is taken as $0.08$, which is slightly smaller than $\frac{\hat{\omega}_{1}}{2}  = \frac{1}{12}$. 
Figure \ref{fig:Rarefaction-2d} displays the numerical solutions 
obtained by our third-order positivity-preserving WB CDG method.  We observe that the density and the pressure get close
to zero but remain positive throughout the simulation. It is noticed that the CDG code would blow-up, if the positivity-preserving limiter is not employed.

\begin{figure}[htbp]
	\centering
	\subfigure[{$\rho$}]{
		\includegraphics[width=0.3\textwidth]
		{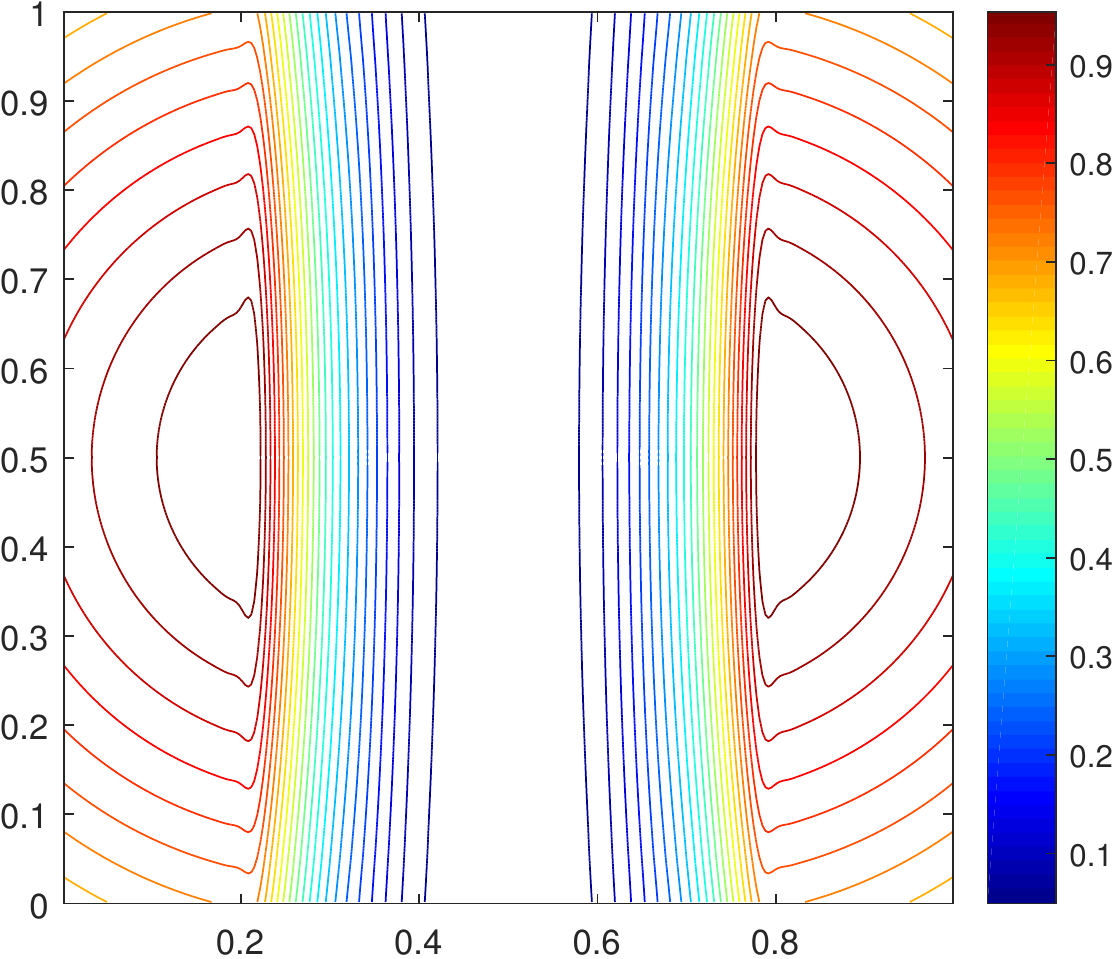}
	}
	\hfill
	\subfigure[{$\rho u_{1}$}]{
		\includegraphics[width=0.3\textwidth]
		{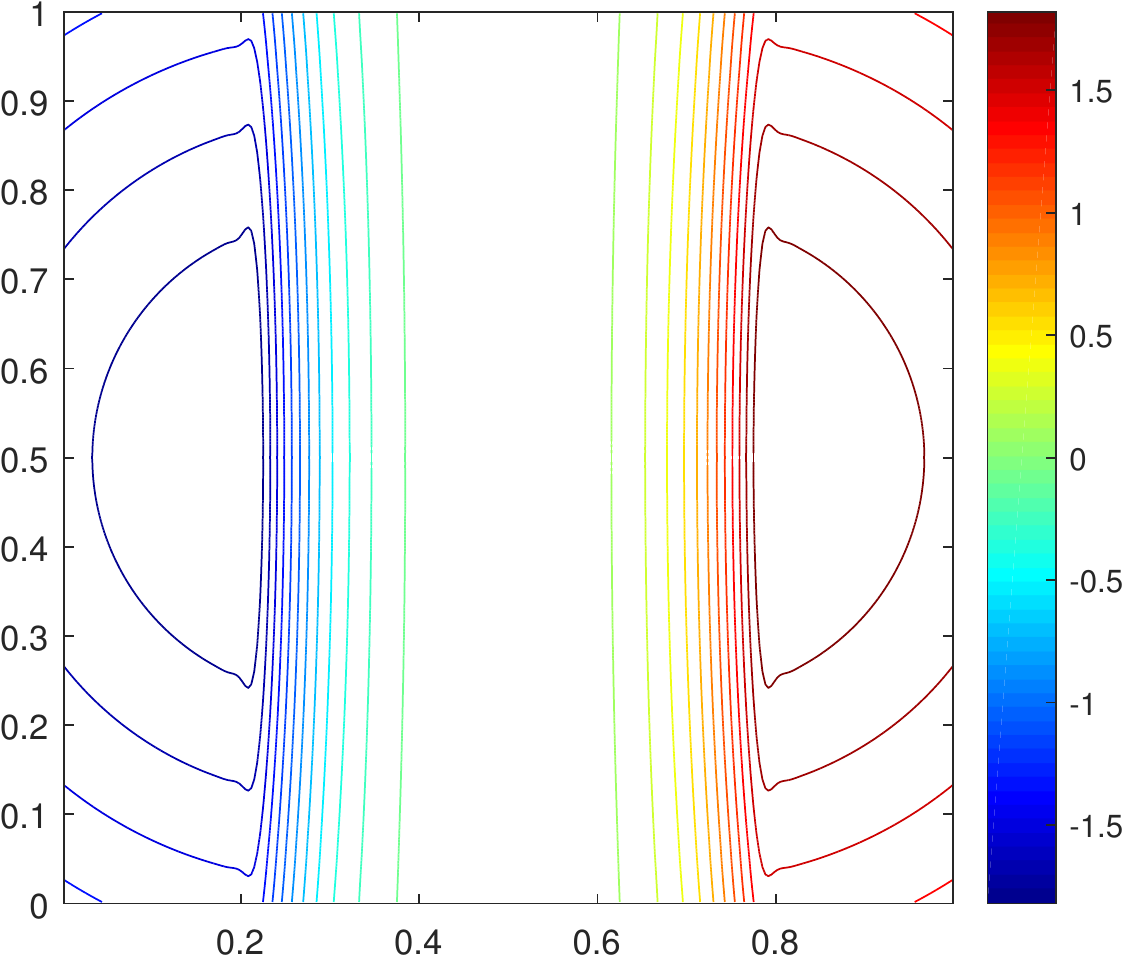}
	}
	\hfill
	\subfigure[{$p$}]{
		\includegraphics[width=0.3\textwidth]
		{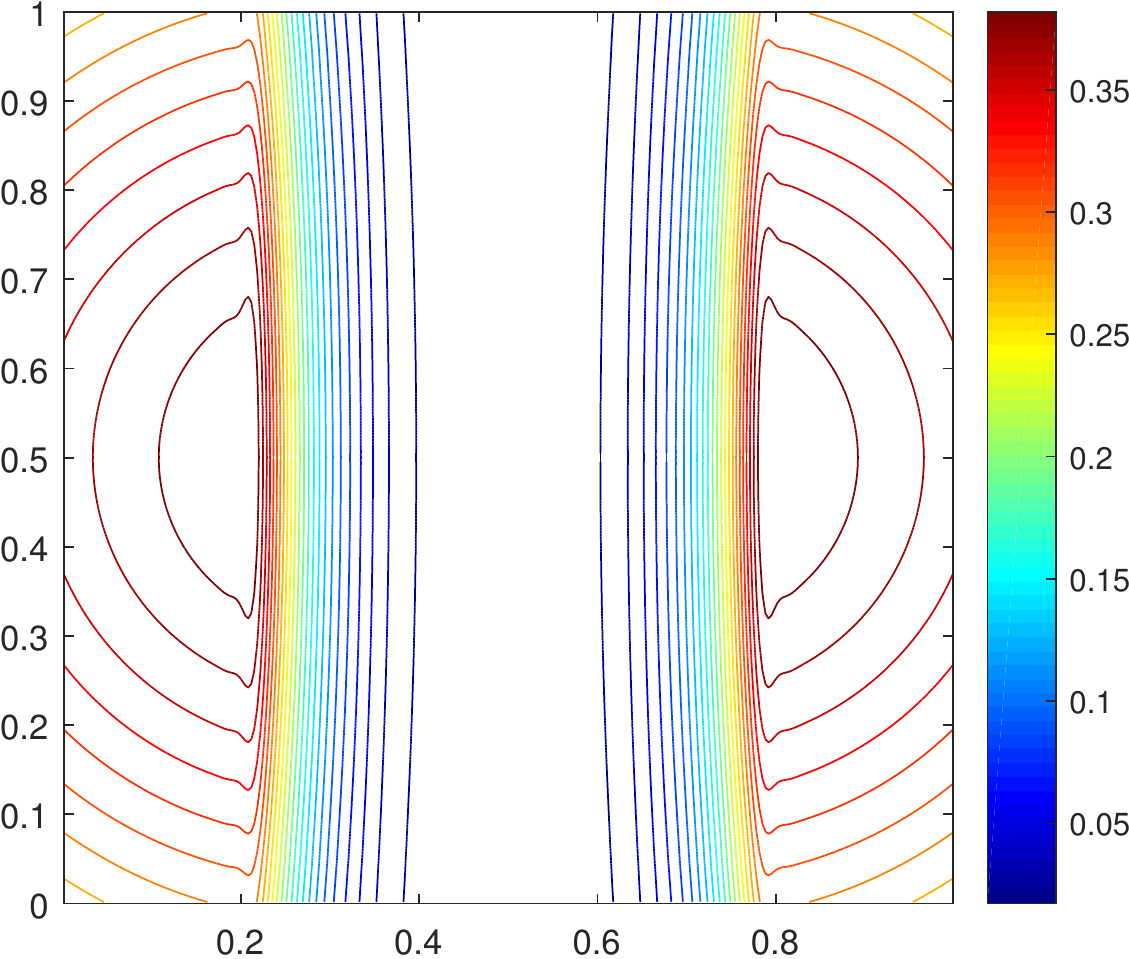}
	}
	\caption{ Example 8: {Contour} plots for the two-dimensional rarefaction test at $t = 0.1$ obtained by the positivity-preserving WB CDG scheme with $100\times100$ uniform cells. }
	\label{fig:Rarefaction-2d}
\end{figure}

\subsubsection{Example 9: Blast problem}

In order to demonstrate the positivity-preserving property and the capability of the proposed WB CDG
method in resolving strong discontinuities, we consider a two-dimensional blast problem \cite{wu2021uniformly}
under the gravitational field (\ref{eq:polytropic-radial-potential}). The initial condition is obtained by adding
a huge jump to the pressure of the polytropic equilibrium  solution (\ref{eq:polytropic-radial}). {Specially,  the initial pressure is given by}
\begin{equation*}
	p(r) = K\rho(r)^2 +
	\left\{
	\begin{aligned}
		100 \,, \quad  r < 0.1  \,,
		\\  0 \,, \quad  r \geq 0.1 \,.
	\end{aligned}
	\right.
\end{equation*}
We set the parameters $K = g = 1 $ and $\rho_{c} = 0.01$, so that {the} low pressure and {the} low density appear in the solution and make this test challenging. The computational domain is set as $\Omega = [-0.5, 0.5]\times[-0.5, 0.5]$, and the adiabatic index is $\gamma = 2$.

In this test, {both} the adaptive WENO limiter \cite{qiu2005runge} (see Remark \ref{rem:WB-WENO} for its WB implementation with the TVB parameter $M = 200$) and {the} positivity-preserving limiter are implemented. 
%The CFL number is set as $0.08$ in the computation. 
{Figure \ref{fig:polytropic-pp} displays the} contour plots of {the} density and {the} pressure at $t = 0.005$ computed
by the third-order positivity-preserving WB CDG method with $200 \times 200$ cells. {Figure \ref{fig:polytropic-pp} also gives} the plots along the line $y = 0$, from which we can clearly observe a strong shock at $|x| \approx 0.4$. It is seen that the discontinuities are captured with high resolution, and the proposed CDG method preserves the positivity of {the} density and {the} pressure as well as the axisymmetric structure of the solution.

\begin{figure}[htbp]
	\centering
	\subfigure{
		\includegraphics[width=0.47\textwidth]
		{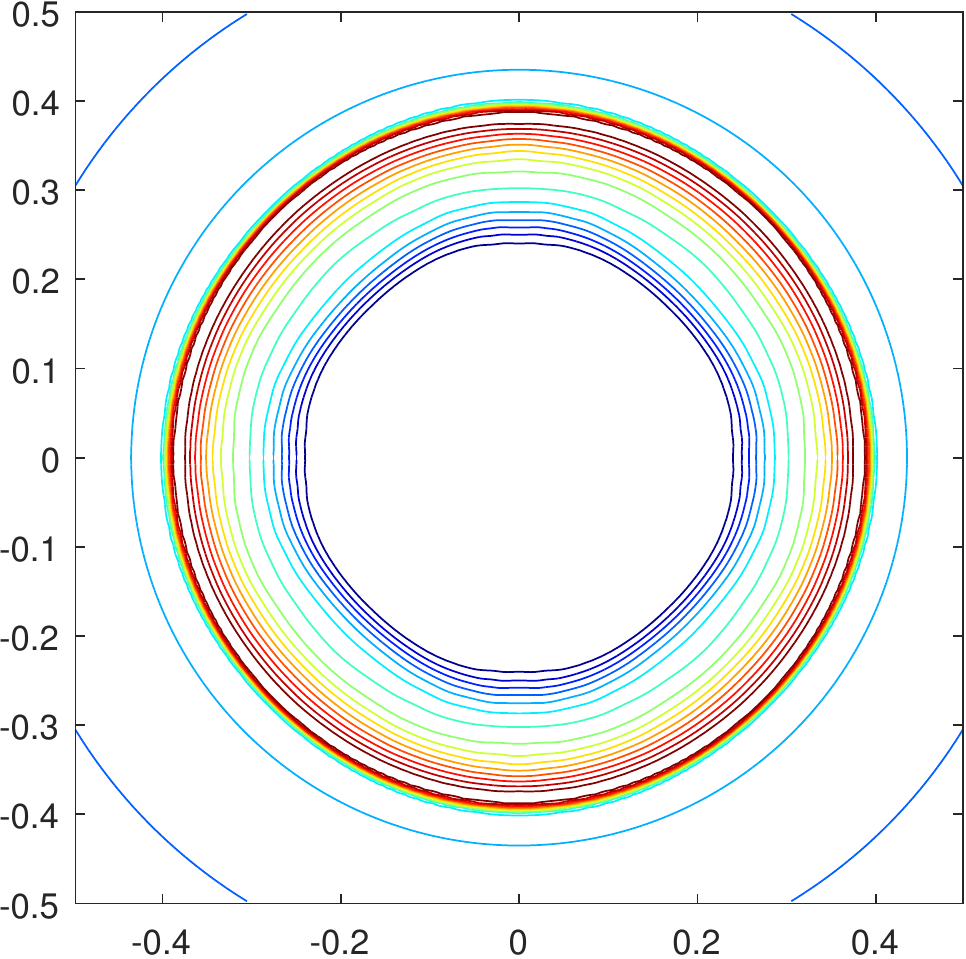}
	}
	\hfill
	\subfigure{
		\includegraphics[width=0.47\textwidth]
		{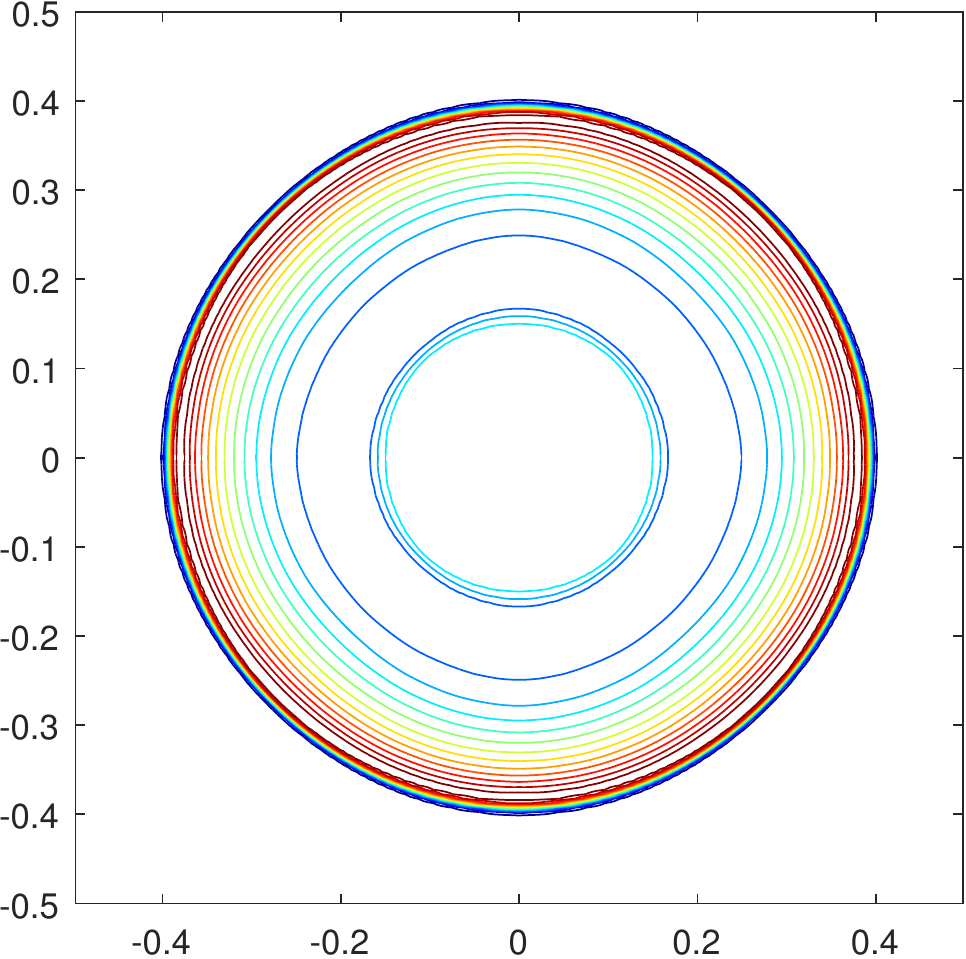}
	}
	\hfill
	\subfigure{
		\includegraphics[width=0.47\textwidth,height=0.43\textwidth]
		{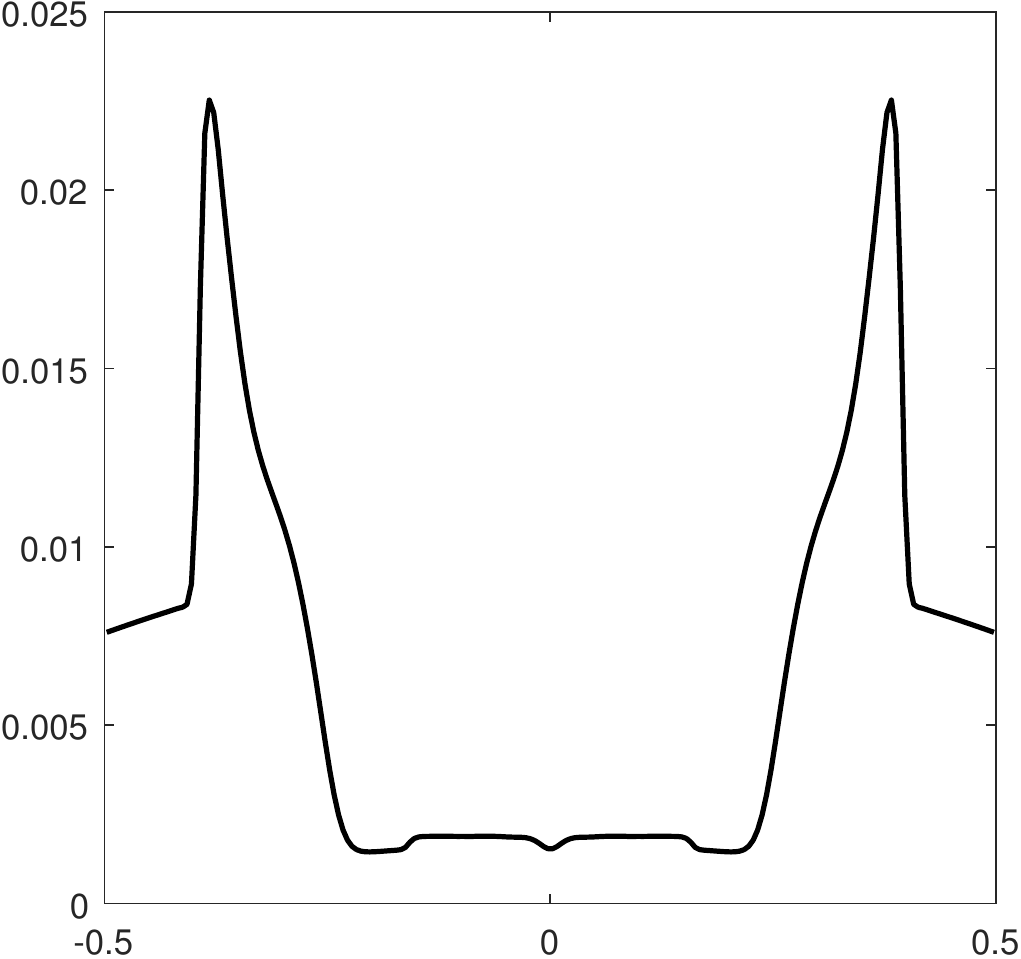}
	}
	\hfill
	\subfigure{
		\includegraphics[width=0.47\textwidth,height=0.43\textwidth]
		{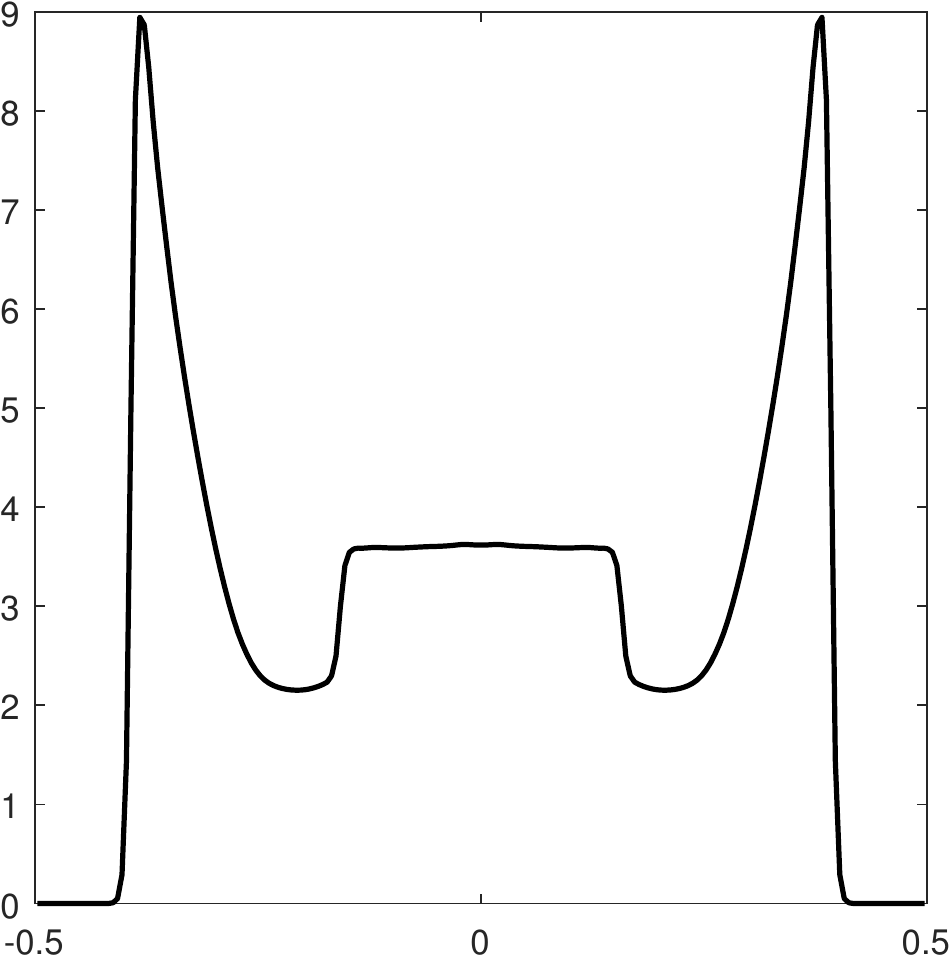}
	}
	\caption{ Example 9:  {Contour plots} of the density (top-left) and the pressure
		(top-right) {and corresponding plots (bottom) along the line $y = 0$ %within the interval $[-0.5, 0.5]$ with 200 cells
}
for the two-dimensional blast problem at $t = 0.005$,
 obtained by the positivity-preserving WB CDG scheme
		on a mesh with $200\times200$ uniform cells.
%{Plots} of the density (bottom-left) and the pressure
%		(bottom-right) along the line $y = 0$ within the interval $[-0.5, 0.5]$ with 200 cells.
}
	\label{fig:polytropic-pp}
\end{figure}

\subsubsection{Example 10: Rising thermal bubble}\label{sec:RTB}

This is a benchmark test problem arising from {the} atmospheric flows \cite{ghosh2016well,giraldo2008study,wu2021uniformly}.
{It shows the evolution of a warm bubble in a constant potential temperature environment.
Because the bubble is warmer than the ambient air, it rises while deforming as a consequence of the shearing motion caused by the velocity field gradients until it forms a mushroom cloud.}
 The computational domain
is set as $\Omega = [0, 1000]\times[0, 1000]{~{\rm m^2}}$.  
%{The unit of length is meters}.
The boundary conditions on all sides are set as the solid walls and the reflective boundary conditions are specified. 
{The initial solution is a stratified atmosphere in hydrostatic balance; see, e.g., the second example  in  the Appendix  of \cite{ghosh2016well}. The constant potential temperature (and thus the reference temperature at $y=0$ m) is 300 K, and the reference pressure is $10^5{\rm N/m}^2$. The ambient flow is at rest (i.e. ${\bf u} = {\bf 0} {~ {\rm m}/{\rm s}}$) and experiences a constant gravitational force per unit mass of $g=9.8~{\rm m/s}^2$, which implies
a linear gravitational field  with $\phi_{x} = 0 {~ {\rm m}/{\rm s}^{2}}$ and $\phi_{y} = g$}.
 {The potential temperature and the Exner pressure of  the ambient air are $\Theta = T_{0} = 300 ~ {\rm K}$ and $\Pi  = 1 - \frac{ (\gamma - 1)gy}{\gamma RT_{0}}$, respectively}, where $R = 287.058 ~ {\rm J} / {({\rm kg} \cdot {\rm K})}$ is the gas constant {for dry air}. Initially, the warm bubble is added as a potential temperature perturbation to the hydrostatic balance:
\begin{equation}\notag
	\Delta \Theta (x,y,t=0) =
	\left\{
	\begin{aligned}
		& 0 \,,      \quad\quad\quad\quad ~                \quad  r > r_{c}     \,,  \\
		& \frac{\theta_{c}}{2} (1 + \cos(\pi r / r_{c})) \,, \quad  r \leq  r_{c}  \,,
	\end{aligned}
	\right.
	\quad r = \sqrt{ (x-x_{c})^2 + (y-y_{c})^2 } \,,
\end{equation}
where $\theta_{c} = 0.5 ~ {\rm K}$, $(x_{c}, y_{c}) = {(500~{\rm m}, 350~{\rm m})}$, and $r_{c} = 250 ~ {\rm m}$.
The pressure and density are computed by $\Theta$  and $\Pi$ via the following formulas:
\begin{equation}
	p = p_{0} \Pi^{\frac{\gamma}{\gamma - 1}} \,, \quad  \rho = \frac{p_{0}}{R \Theta} \Pi^{\frac{1}{\gamma - 1}} \,,
\end{equation}
with the reference pressure $p_{0} = 10^{5} ~ {\rm N}/{\rm m}^2$. Figure \ref{fig:thermal-bubble} shows the evolution of {the} potential temperature perturbation $\Delta \Theta$ obtained by the proposed fourth-order accurate
WB CDG method on the meshes of $100 \times 100$ cells {($10 ~ {\rm m}$} resolution). We observe {clearly} that the initial circular bubble is deformed to a mushroom-like cloud {and  the} flow structures are well resolved.

\begin{figure}[htbp]
	\centering
	\subfigure[$t = 400~{\rm s}$]{
		\includegraphics[width=0.47\textwidth]
		{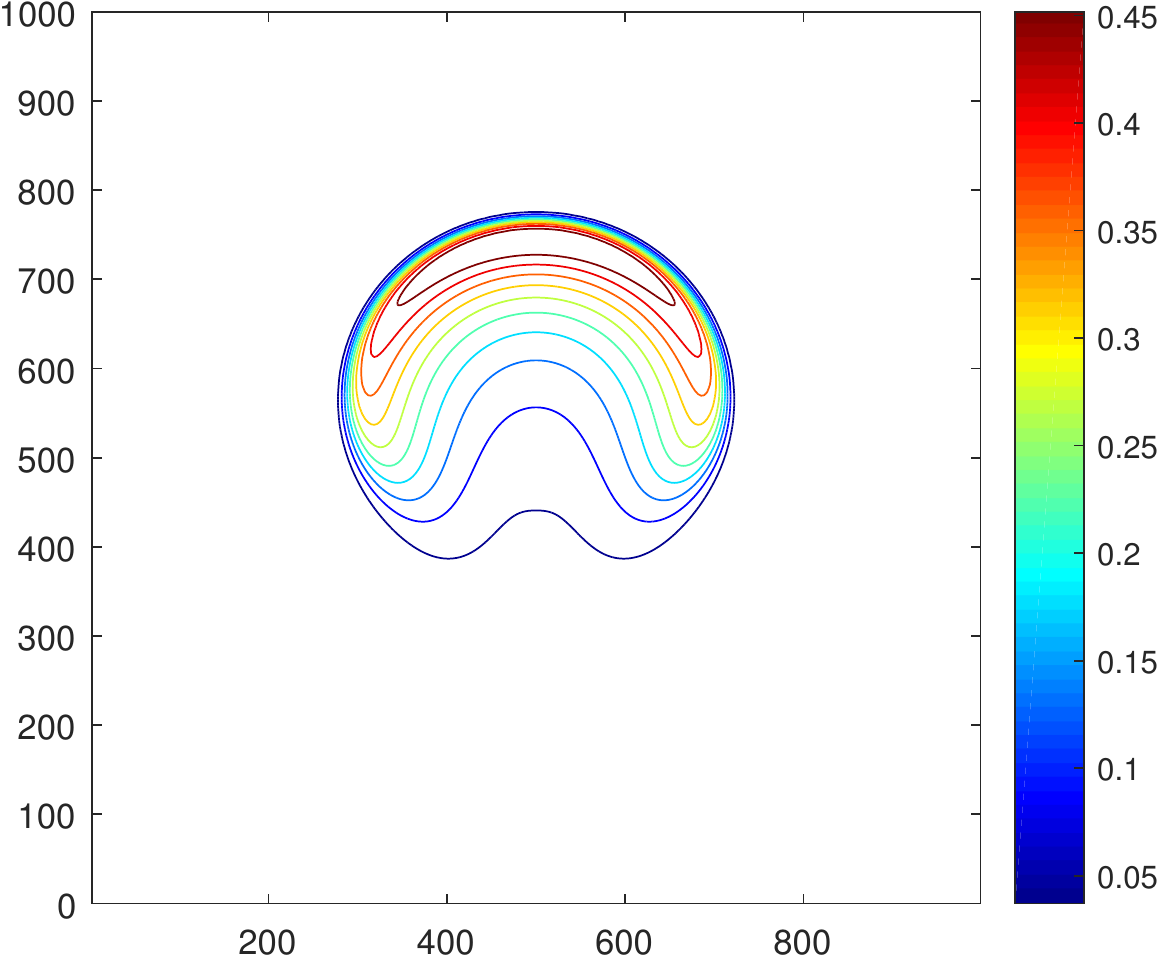}
	}
	\hfill
	\subfigure[$500~{\rm s}$]{
		\includegraphics[width=0.47\textwidth]
		{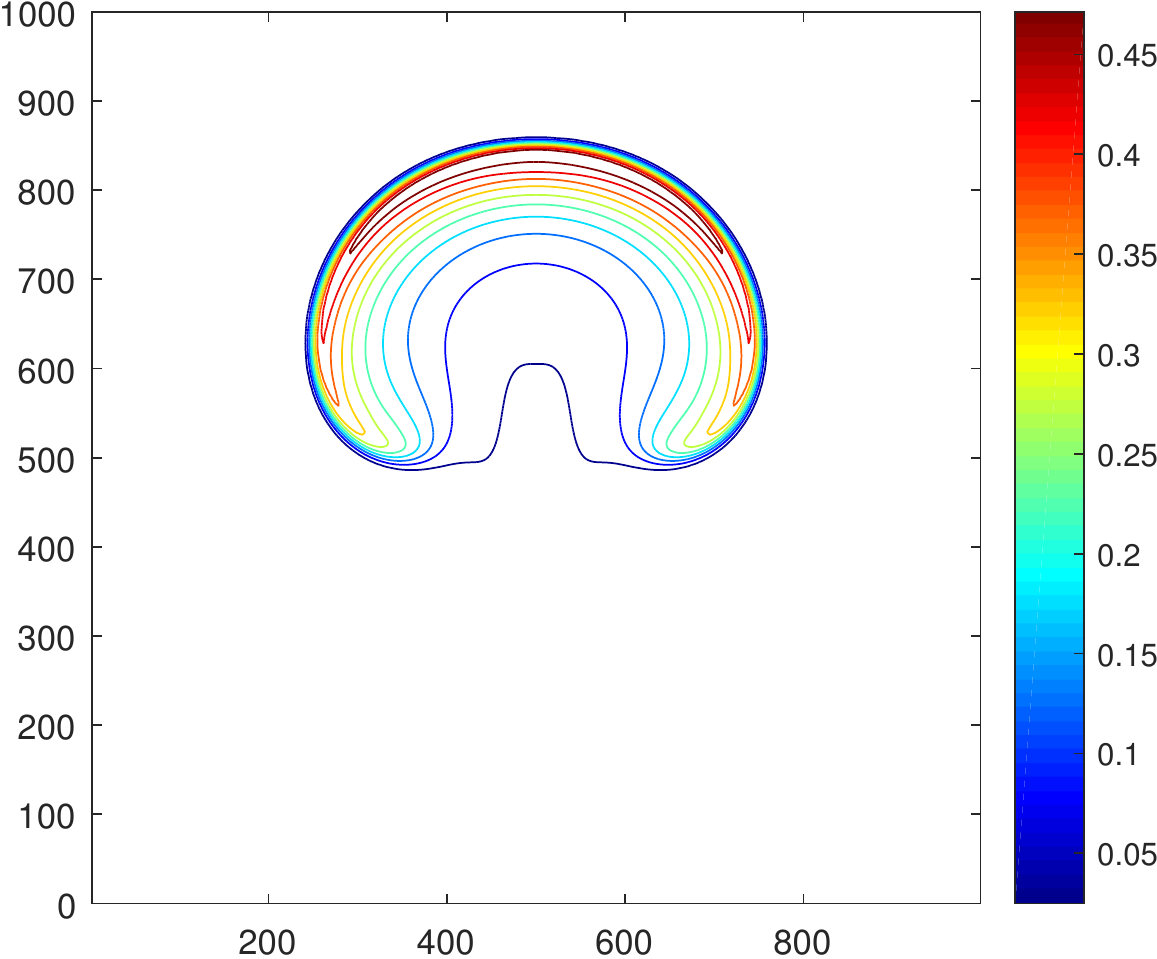}
	}
	\hfill
	\subfigure[$600~{\rm s}$]{
		\includegraphics[width=0.47\textwidth]
		{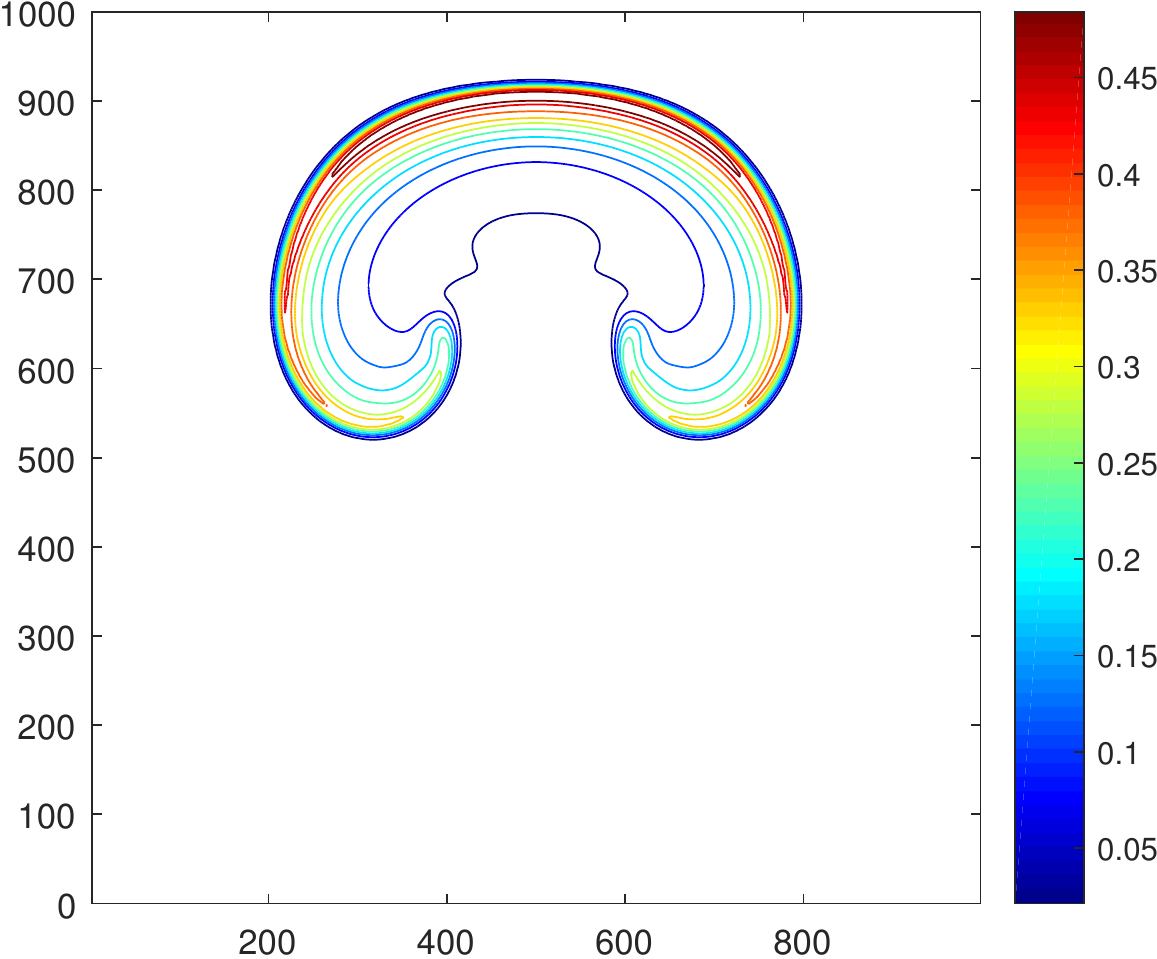}
	}
	\hfill	
	\subfigure[$700~{\rm s}$]{
		\includegraphics[width=0.47\textwidth]
		{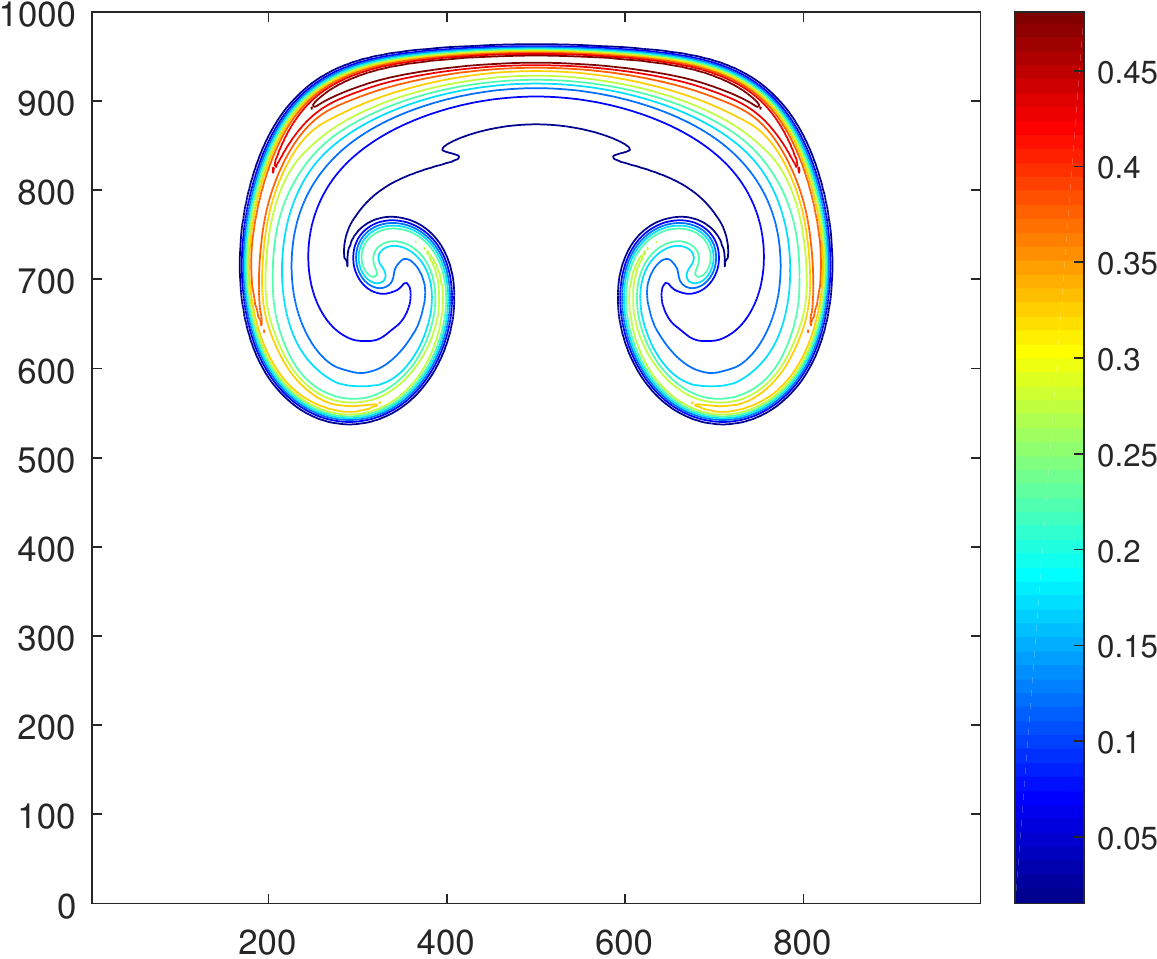}
	}
	\caption{Example 10: Contour plots of the potential temperature perturbation $\Delta \Theta$  at $t = 400~{\rm s}$,
			$500~{\rm s}$,  $600~{\rm s}$, and  $700~{\rm s}$, respectively, obtained by our fourth-order
			WB CDG method. $10$ uniformly spaced contour lines are displayed. }
	\label{fig:thermal-bubble}
\end{figure}

\subsubsection{Example 11: Rayleigh--Taylor (RT) instability tests}
This example simulates three tests, which involve discontinuous stationary hydrostatic solutions. 

For the first two tests, we use the same setups as in \cite{Zenk2017WBNDG}, with 
the gravitational potential function $\phi(x,y) = y$ and the computational domain $ \Omega = [-0.25, 0.25]\times[-1, 1]$. The 
 initial solution is stationary hydrostatic with the pressure and density given by 
\begin{equation}\label{eq:0012}
	p(x,y,0)=
	\begin{cases}
		p_{0}\exp{(-y/T_{l})},    &  y < 0  \,,  \\
		p_{0}\exp{(-y/T_{u})},    &  y > 0  \,,
	\end{cases} \qquad
	\rho(x,y,0)=
	\begin{cases}
		p/T_{l},    &  y < 0  \,,  \\
		p/T_{u},    &  y > 0  \,,
	\end{cases}
\end{equation}
where $p_{0} = 1$, $\{T_{l}, T_{u}\}$ are two different constant temperatures. 
Note that the density $\rho(x,y,0)$ is discontinuous because of 
the jump in temperature at $y=0$, while the pressure is continuous at $y=0$.
We consider two configurations \cite{Zenk2017WBNDG}:
\begin{itemize}
	\item  {\bf RT test 1}: $T_{l} = 1$ and $T_{u} = 2$. This is a stable case, because the light fluid is above the heavy fluid. 
	\item  {\bf RT test 2}: $T_{l} = 2$ and $T_{u} = 1$. This is a physically unstable case, because heavy fluid is above the light fluid.
\end{itemize}
The numerical solutions for both tests are computed until $t=0.1$ by using the third-order WB CDG scheme with respectively 
$25\times100$ and $50\times200$ uniform cells. 
In order to demonstrate our WB implementation of the WENO limiter (see Remark \ref{rem:WB-WENO}), 
we perform the tests with and without the WENO limiter, respectively. 
Tables \ref{tab:Rayleigh-case1} and \ref{tab:Rayleigh-case2} list 
the $L^1$ errors between the numerical solutions and the projection of the initial solution \eqref{eq:0012}. 
We clearly see that all the numerical errors are at the level of rounding error, confirming that the proposed CDG method and our 
implementation of the WENO limiter exactly preserve the WB property. As the solution remains 
at the stationary hydrostatic state, we observe that no cell is flagged as ``troubled cells'' up to $t=0.1$. 
Following \cite{Zenk2017WBNDG}, we also continue the simulation for a very long time. 
As shown in Figure \ref{fig:Rayleigh-WENO-case12}, the steady state solution in the stable case (RT test 1)  
is still exactly preserved in a long-time simulation until $t=10$, thanks to the WB property. 
However, %if we continue the simulation of  RT test 2 to $t=10$,  
because the configuration of RT test 2 is the physically unstable, 
in the long-time simulation the small rounding errors 
will accumulate, as time evolves, and eventually cause the RT 
instability near the
interface $y = 0$ (the WB property is locally preserved away from
the interface); see Figure \ref{fig:Rayleigh-WENO-case12}.

\begin{table}[htbp]
	\caption{$L^{1}$ errors at $t=0.1$ for RT test 1 in Example 11.}  
	\centering  
	\begin{tabular}{cccccc}  
		\toprule  
		 Limiter &  mesh &  errors in $\rho$ &  errors in $\rho u_{1}$ &  errors in $\rho u_{2}$  &  errors in $E$ \\
		\midrule  
		 No limiter  &  $25\times100$  &  1.43e-15 & 	3.41e-16 & 	5.22e-16 & 	5.54e-15 \\
		
		&  $50\times200$  &  2.93e-15 & 	6.39e-16 & 	9.07e-16 & 	1.17e-14 \\ 
		
		 WENO limiter  &  $25\times100$  &  1.43e-15 & 	3.41e-16 & 	5.22e-16 &	 5.54e-15 \\
		
		&  $50\times200$  &  2.93e-15 & 	6.39e-16 & 	9.07e-16 & 	1.17e-14 \\		
		\bottomrule  	
	\end{tabular} \label{tab:Rayleigh-case1}
\end{table}

\begin{table}[htbp]
	\caption{$L^{1}$ errors  $t=0.1$ for RT test 2 in Example 11.} 
	\centering  
	\begin{tabular}{cccccc}  
		\toprule  
		 Limiter &  mesh &  errors in $\rho$ &  errors in $\rho u_{1}$ &  errors in $\rho u_{2}$  &  errors in $E$ \\
		\midrule  
		 No limiter  &  $25\times100$  &  9.48e-16 & 	2.92e-16 & 	3.20e-16 &	 5.01e-15 \\
		&  $50\times200$  &  2.02e-15 &	5.34e-16 & 	6.05e-16 & 	1.01e-14 \\ 
		 WENO limiter  &  $25\times100$  &  9.48e-16 & 	2.92e-16 & 	3.20e-16 &	 5.01e-15 \\
		&   $50\times200$  &  2.02e-15 &	5.34e-16 &	 6.05e-16 &	 1.01e-14 \\ 	
		\bottomrule  	
	\end{tabular} \label{tab:Rayleigh-case2}
\end{table}

\begin{figure}[htbp]
	\centering
	\subfigure[t = 5]{
	\includegraphics[width=0.2\textwidth]
	{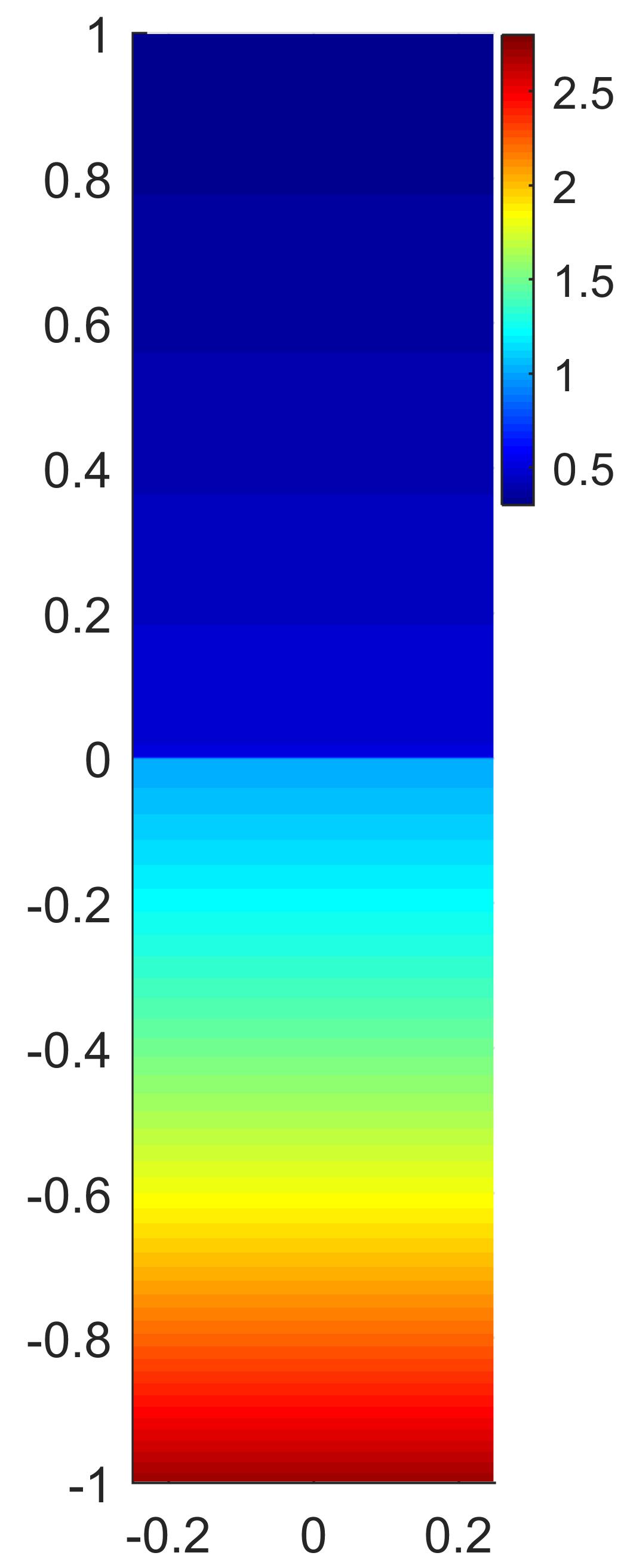}
}
\quad 
\subfigure[t = 7]{
	\includegraphics[width=0.2\textwidth]
	{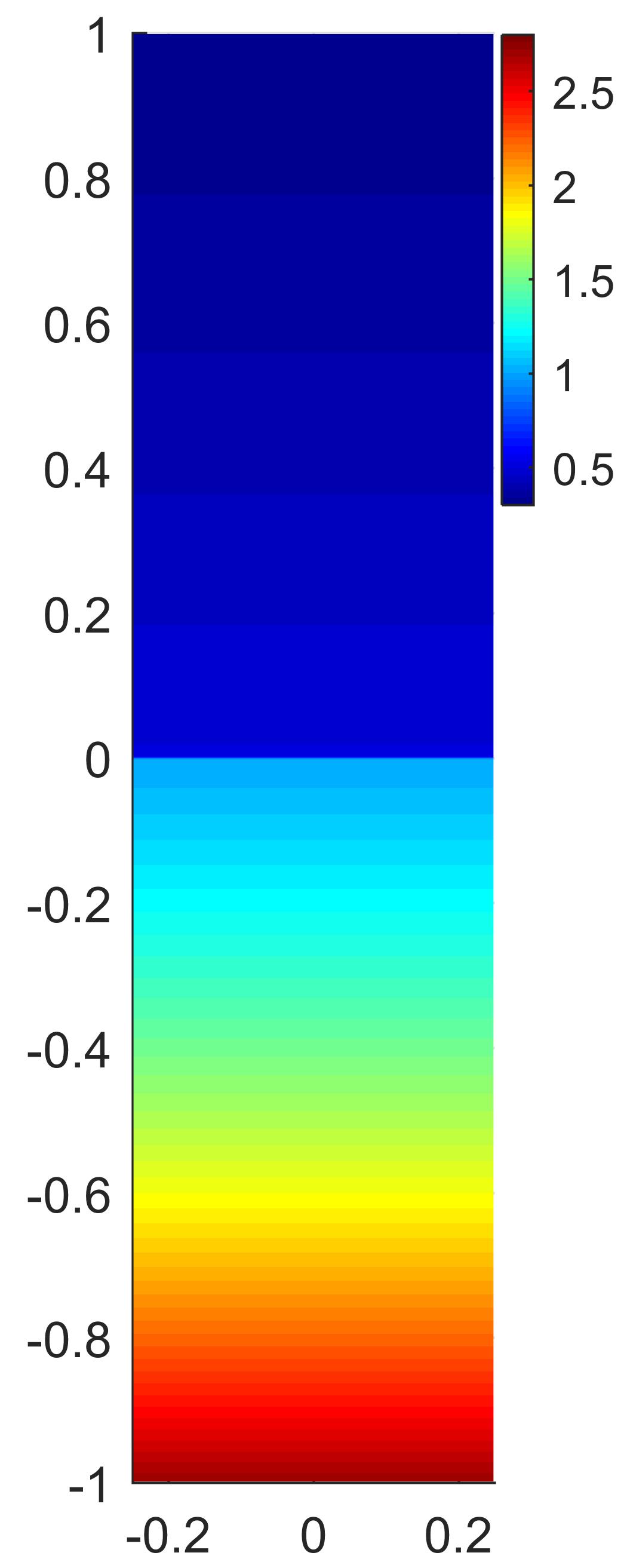}
}
\quad 
\subfigure[t = 9]{
	\includegraphics[width=0.2\textwidth]
	{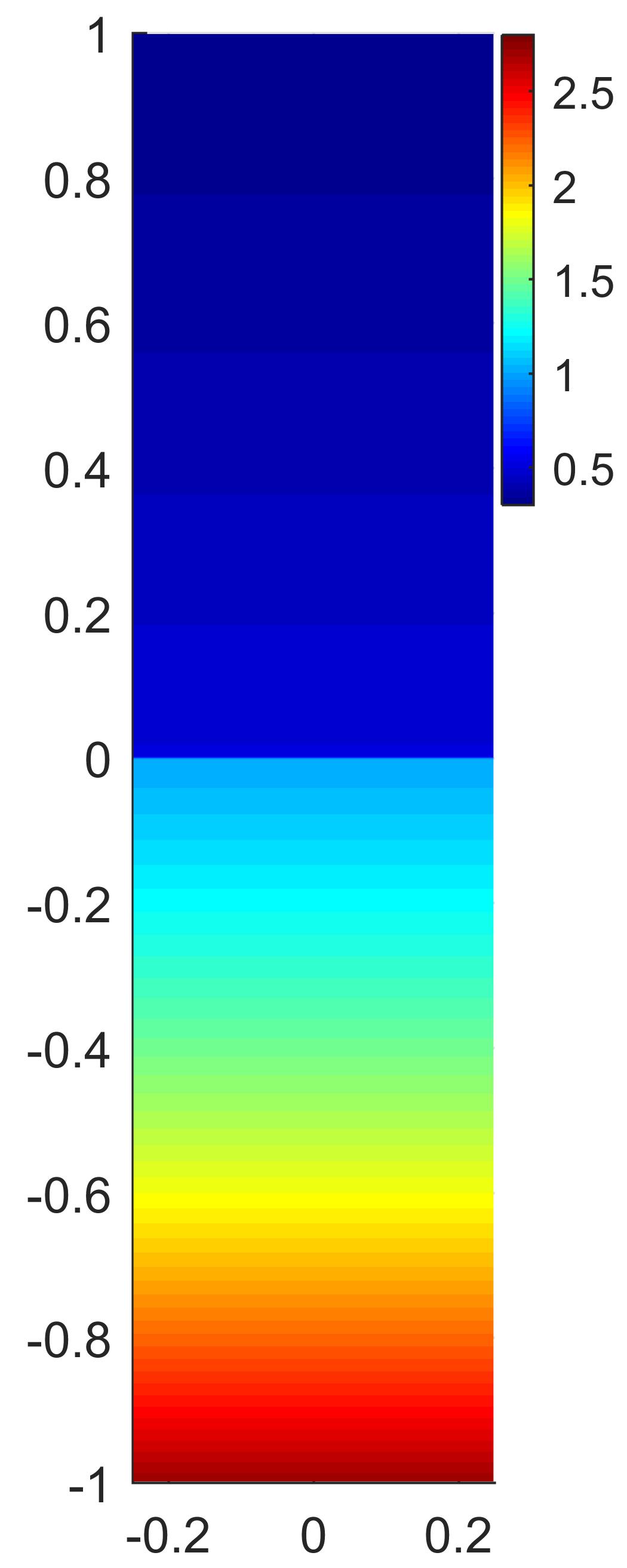}	
}
\quad 
\subfigure[t = 10]{
	\includegraphics[width=0.2\textwidth]
	{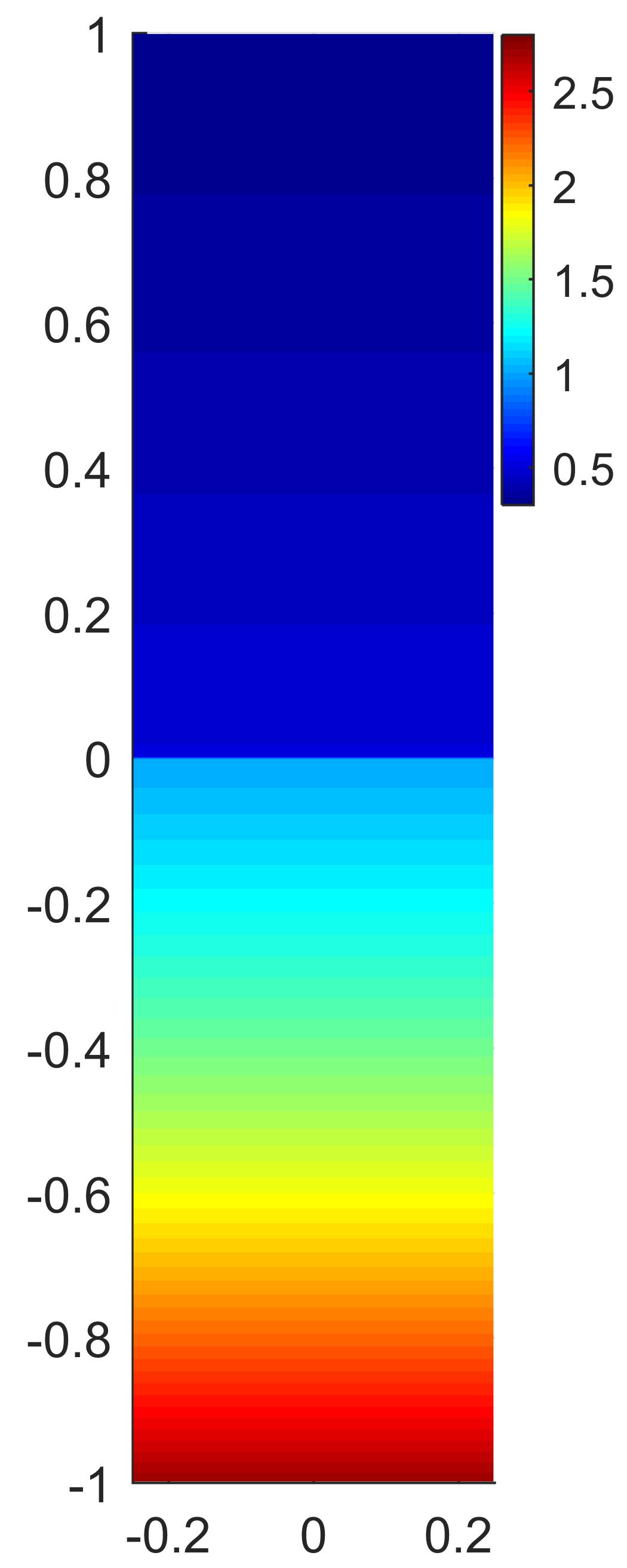}	
}
\\
	\subfigure[$t = 5$]{
	\includegraphics[width=0.2\textwidth]
	{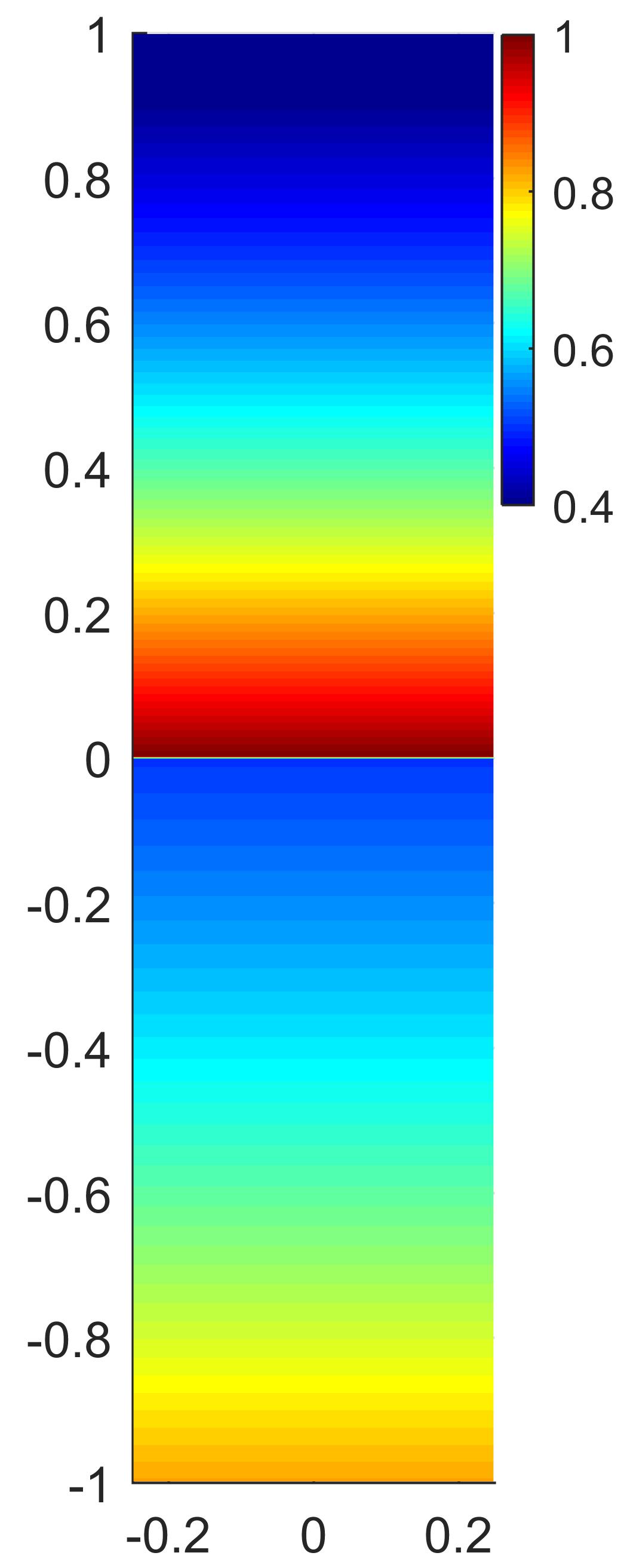}
}
\quad 
\subfigure[$t = 7$]{
	\includegraphics[width=0.2\textwidth]
	{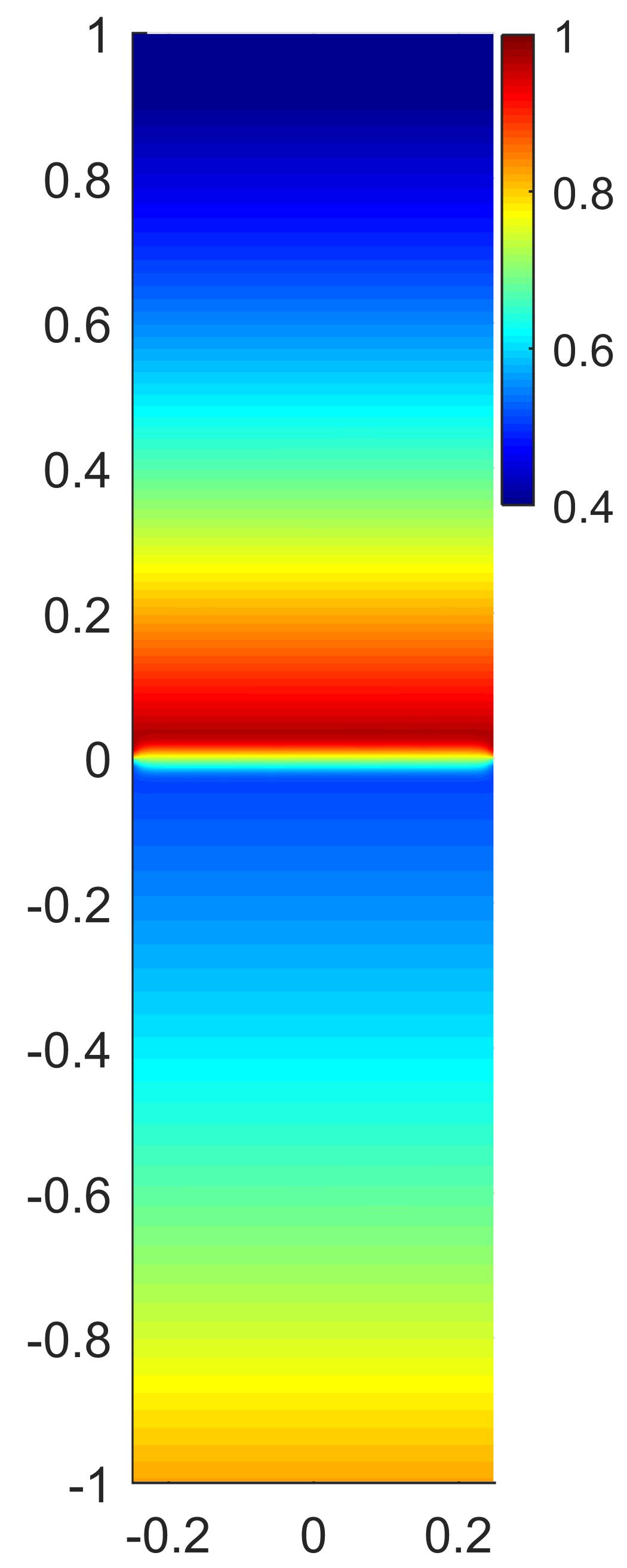}
}
\quad 
\subfigure[$t = 9$]{
	\includegraphics[width=0.2\textwidth]
	{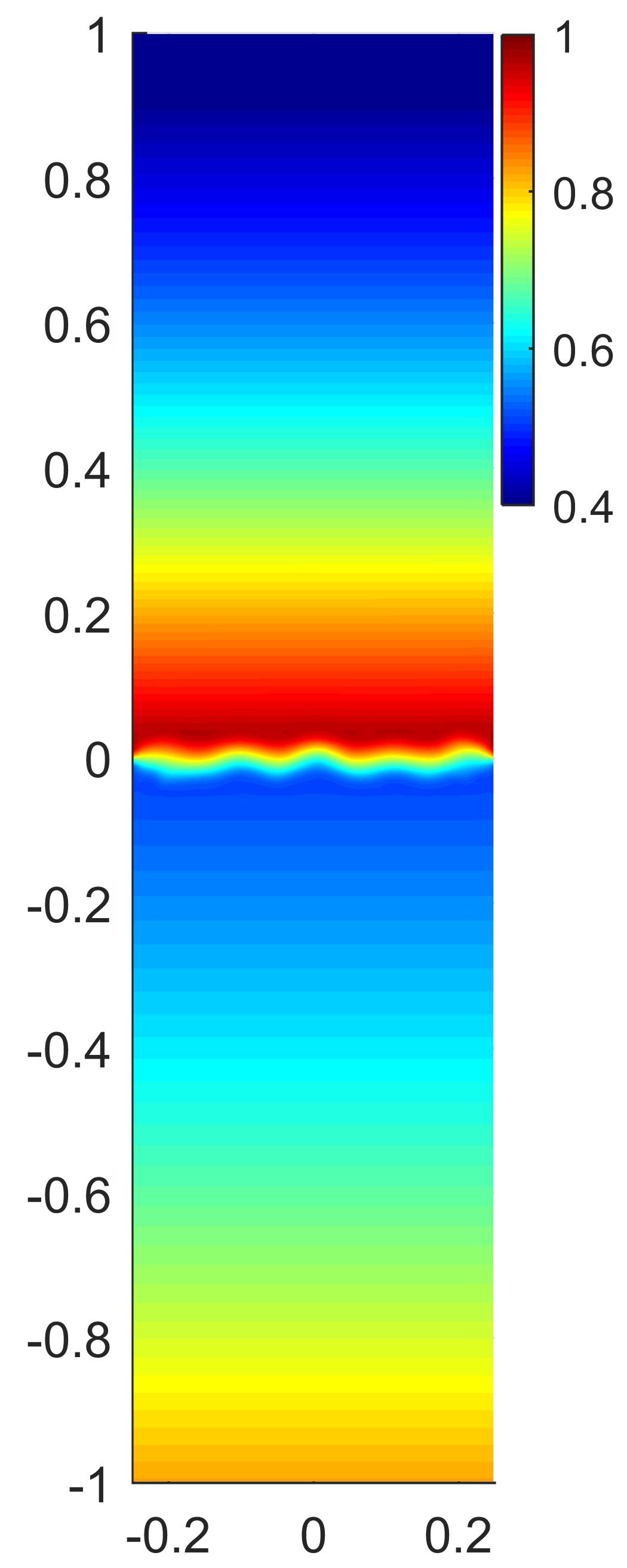}	
}
\quad 
\subfigure[$t = 10$]{
	\includegraphics[width=0.2\textwidth]
	{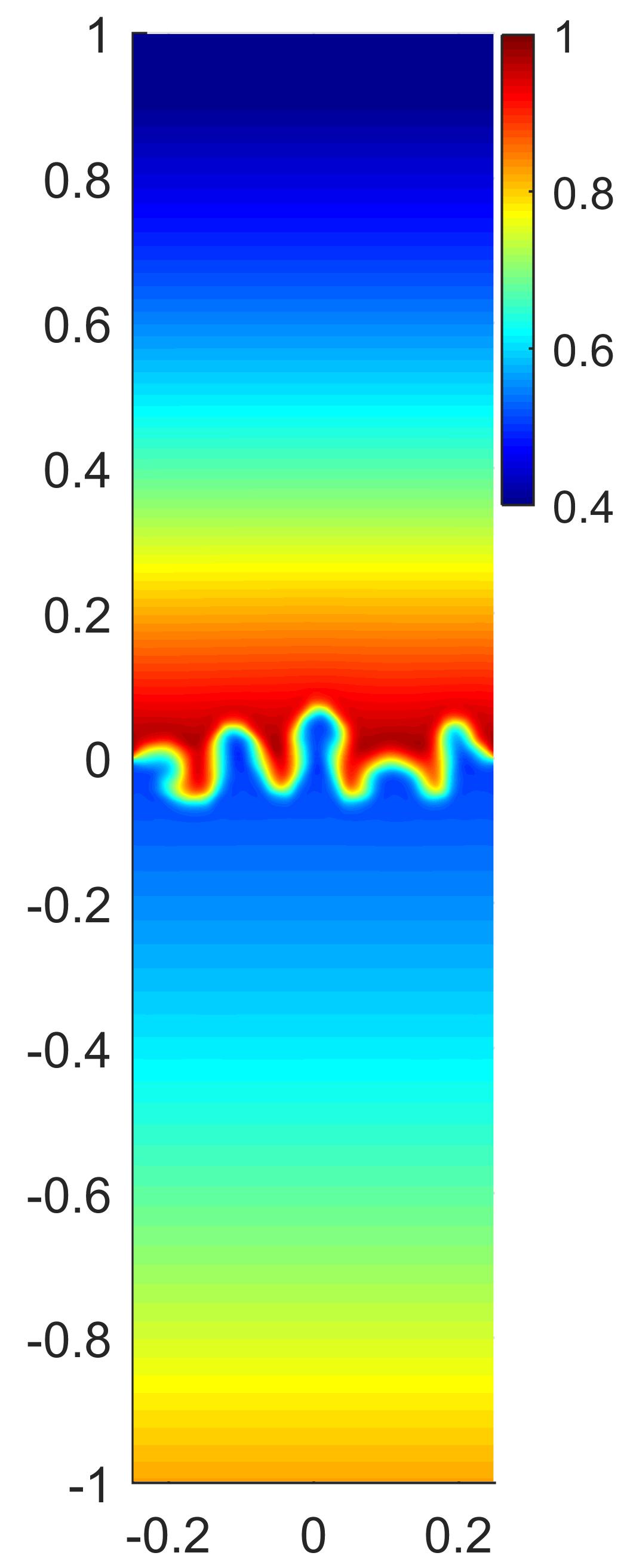}	
}
\caption{Long-time simulations of RT test 1 (top row; the stable configuration) and RT test 2 (bottom row; the unstable configuration) 
	by using the third-order WB CDG scheme 
	with $50\times200$ uniform cells. The WB implementation of the WENO limiter is used with the parameter $M = 200$ in the TVB corrected minmod function \eqref{minmod-corrected}.}
\label{fig:Rayleigh-WENO-case12}
\end{figure}

To check the effectiveness of our WB CDG method in capturing small perturbations near discontinuous 
equilibrium solutions, 	we simulate another classical benchmark RT test \cite{SHI2003690}, which is abbreviated as ``{\bf RT test 3}''
for convenience.
	This test 
%The third RT test (abbreviated as ``{\bf RT test 3}'') considers a  classical benchmark case (see \cite{SHI2003690}), 
 is usually used to validate the ability of a high-order numerical scheme in capturing complicated small wave structures. 
 %around a discontinuous hydrostatic state.  
For comparison purpose, we use the same setup as in \cite{SHI2003690}.   
The gravitational potential function is taken as $\phi(x,y) = -y$, so that  
the acceleration is in the positive $y$-direction. 
%Besides the above two tests, we also simulate another benchmark , whose 
The 
initial condition is a small perturbation to an unstable stationary hydrostatic solution  involving a discontinuity in density: 
\begin{equation*}\label{}
(\rho,u_1,u_2,p)(x,y,0)=
	\begin{cases}
		(2, ~ 0, ~ \tilde{u}(x), ~ 2y+1)^{\top},    &  0 \leq y < 0.5  \,,  \\
		(1, ~ 0, ~ \tilde{u}(x), ~ y+1.5)^{\top},   &  0.5 \leq y \leq 1  \,,
	\end{cases} 
\end{equation*}
where $\tilde{u}(x) = -0.025\sqrt{\gamma p/ \rho} \cos (8 \pi x)$. 
The variables $(\rho,u_1,u_2,p)$ are set as $ = (1,0,0,2.5)$ on the top boundary and as $(2,0,0,1)$ on the bottom. 
Reflective boundary conditions are imposed on both the left and right boundaries. 
The WB implementation of the WENO limiter is used with the parameter $M = 200$ in the TVB corrected minmod function \eqref{minmod-corrected}. 
As in \cite{SHI2003690}, the mesh refinement study is carried out here by using three different uniform square meshes: 
$h=\frac1{240}, \frac1{480}, \frac1{960}$, where $h$ is the spatial step-size in both the $x$- and $y$-directions. 
 Figure \ref{fig:Rayleigh-Taylor-case2} displays the density contours at time $t=1.95$. 
We
 see that our third-order WB CDG method can clearly resolve the complicated wave structures and that the numerical resolutions are comparable 
 to those computed with WENO9 (a ninth-order WENO scheme) in \cite{SHI2003690} with the same mesh sizes.

%\begin{figure}[htbp]
%	\centering
%	\subfigure[$t = 2$, density]{
%		\includegraphics[width=0.2\textwidth]
%		{Figures/Rayleigh-Density-WENO-T2.jpg}
%	}
%	\hfill
%	\subfigure[$t = 2$, pressure]{
%		\includegraphics[width=0.2\textwidth]
%		{Figures/Rayleigh-Pressure-WENO-T2.jpg}
%	}
%	\hfill
%	\subfigure[$t = 5$, density]{
%		\includegraphics[width=0.2\textwidth]
%		{Figures/Rayleigh-Density-WENO-T5.jpg}	
%	}
%	\hfill
%	\subfigure[$t = 5$, pressure]{
%		\includegraphics[width=0.2\textwidth]
%		{Figures/Rayleigh-Pressure-WENO-T5.jpg}	
%	}
%	\caption{Rayleigh–Taylor instability test two, first case, obtained by the third-order WB CDG scheme with $30\times 120$ uniform cells. WB limiter, parameters $ M_{1} = M_{2} = M_{3} = M_{4} = 200$, no cell was flagged as ``troubled cell''.}
%	\label{fig:Rayleigh-Taylor-case1}
%\end{figure}

\begin{figure}[htb]
	\centering
	%\subfigure[$\Delta x = \Delta y = 1/120$]{
	%	\includegraphics[width=0.15\textwidth]
	%	{Figures/Rayleigh-Taylor-30.pdf}
	%}
	%\hfill
	\subfigure[$h = 1/240$]{
		\includegraphics[width=0.2\textwidth]
		{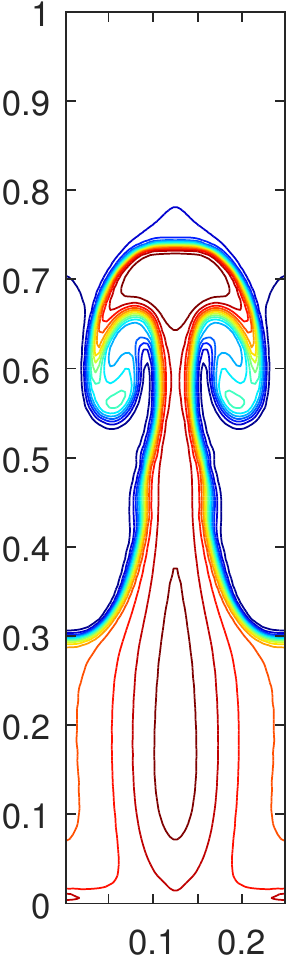}
	}
	\qquad 
	\subfigure[$h = 1/480$]{
		\includegraphics[width=0.2\textwidth]
		{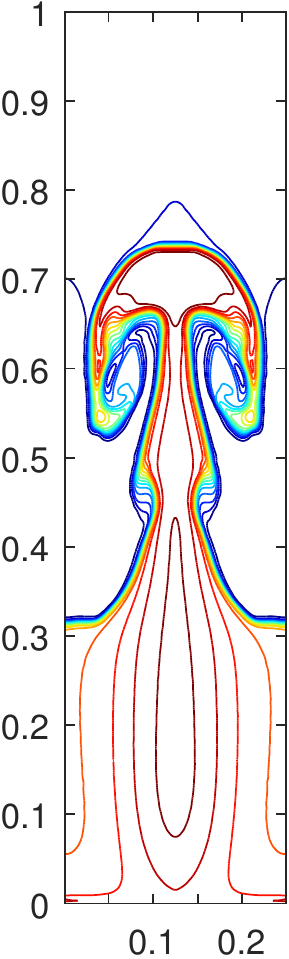}
	}
	\qquad 
	\subfigure[$h = 1/960$]{
		\includegraphics[width=0.2\textwidth]
		{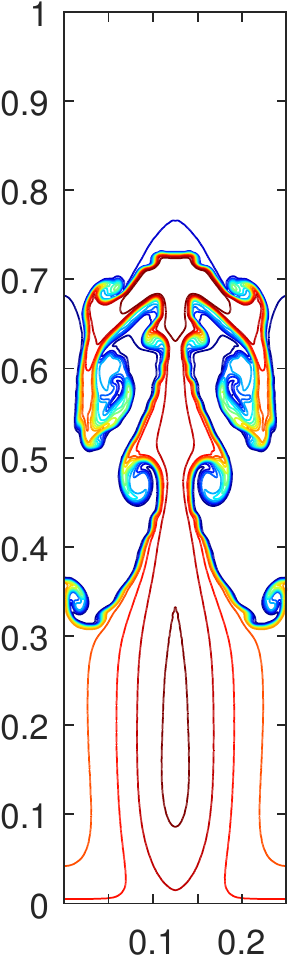}
	}
	\caption{RT test 3 (the perturbation configuration): Contour plots of the density at $t = 1.95$ 
		obtained by the third-order WB CDG scheme with $15$ equally spaced contour lines from $0.952269$ to 
		$2.14589$. }
	\label{fig:Rayleigh-Taylor-case2}
\end{figure}

\section{Conclusions} \label{section:conclusion}

{This paper} designed a high-order positivity-preserving {well-balanced (WB) central discontinuous Galerkin (CDG)} method for the Euler equations under
gravitational fields. A novel WB spatial discretization in the CDG framework was devised with suitable
modifications to the numerical dissipation term and {the} source term approximation, while the desired conservative
and positivity-preserving properties were also simultaneously preserved in the discretization. The modifications were based on a novel projection for the stationary hydrostatic  solution, which {had} the same order of accuracy as the standard $L^2$-projection, {could} be explicitly calculated,
and {was} easy to implement without solving any optimization problems.
Moreover, the novel projection guaranteed the projected stationary solution having
the same cell averages on both the primal and dual meshes.
This feature was a key to obtain the desired properties of our schemes. Based on some convex decomposition techniques and several key properties of the admissible states, we rigorously proved that the resulting WB CDG
method {satisfied} a weak positivity-preserving property, which {implied} that a simple limiter {could} ensure the positivity-preserving property without losing the high-order accuracy and conservativeness. Extensive one- and two-dimensional numerical examples
were provided to demonstrate the robustness, high-order accuracy, WB and positivity-preserving properties of the proposed schemes.

\section*{Acknowledgements}
H.Z.~Tang is partially supported by the National Key R\&D Program of China, Project Number 2020YFA0712000, and the National Natural Science Foundation of China (Nos. 12171227 \& 12126320). K.L.~Wu
is partially supported by the National Natural Science Foundation of China, Project Number 12171227.
H.L.~Jiang wishes to thank Professor Tie Zhou of Peking University very much for his support and  the Department of Mathematics of SUSTech for hospitality during the preparation of this paper.

\bibliographystyle{siamplain}
\bibliography{Reference,references2}

\end{document}